\documentclass[a4paper,12pt,leqno]{amsart}
\usepackage{latexsym}
\usepackage[all]{xy}

\usepackage{amssymb} 
\usepackage{amsmath} 
\usepackage{color}
\usepackage{comment}
\usepackage{mathtools}

\usepackage{graphicx}
\usepackage{stmaryrd}

\definecolor{gray}{gray}{0.7}
\definecolor{Gray}{gray}{0.3}

\usepackage{amscd}
\usepackage{mathdots}

\numberwithin{equation}{section}

\theoremstyle{break}
 \newtheorem{theorem}{Theorem}[section]
 \newtheorem{proposition}[theorem]{Proposition}
 \newtheorem{corollary}[theorem]{Corollary}
 \newtheorem{lemma}[theorem]{Lemma}

 \theoremstyle{definition}
 \newtheorem{definition}[theorem]{Definition}
 \newtheorem{remark}[theorem]{Remark}
 \newtheorem{example}[theorem]{Example}

\allowdisplaybreaks[3]

\def\C{\mathbb C}

\def\gl{\mathfrak{gl}}
\def\mm{m}

\def\ZZ{\mathcal{Z}}
\def\hE{^h E}
\def\hmE{^{h_m} E}

\def\Cyclic{\mathfrak C}
\def\SS{\mathfrak S}

\def\QH{QH}

\def\XX{A}

\DeclareMathOperator{\Span}{span}
\DeclareMathOperator{\Hilb}{Hilb}
\DeclareMathOperator{\Spec}{Spec}
\DeclareMathOperator{\Sing}{Sing}
\DeclareMathOperator{\GL}{GL}

\DeclareMathOperator{\Hess}{Hess}
\DeclareMathOperator{\Pet}{Pet}
\newcommand{\Fl}{Fl}

\begin{document}

\title[Coordinate rings of Hessenberg varieties]{Coordinate rings of regular nilpotent Hessenberg varieties in the open opposite Schubert cell}
\author [T. Horiguchi]{Tatsuya Horiguchi}
\address{National institute of technology, Ube college, 2-14-1, Tokiwadai, Ube, Yamaguchi, 755-8555, Japan}
\email{tatsuya.horiguchi0103@gmail.com}

\author [T. Shirato]{Tomoaki Shirato}
\address{National institute of technology, Ube college, 2-14-1, Tokiwadai, Ube, Yamaguchi, 755-8555, Japan}
\email{stomoaki@ube-k.ac.jp}



\begin{abstract}
Dale Peterson has discovered a surprising result that the quantum cohomology ring of the flag variety $\GL_n(\C)/B$ is isomorphic to the coordinate ring of the intersection of the Peterson variety $\Pet_n$ and the opposite Schubert cell associated with the identity element $\Omega_e^\circ$ in $\GL_n(\C)/B$.
This is an unpublished result, so papers of Kostant and Rietsch are referred for this result.
An explicit presentation of the quantum cohomology ring of $\GL_n(\C)/B$ is given by Ciocan--Fontanine and Givental--Kim. 
In this paper we introduce further quantizations of their presentation so that they reflect the coordinate rings of the intersections of regular nilpotent Hessenberg varieties $\Hess(N,h)$ and $\Omega_e^\circ$ in $\GL_n(\C)/B$.
In other words, we generalize the Peterson's statement to regular nilpotent Hessenberg varieties via the presentation given by Ciocan--Fontanine and Givental--Kim. 

As an application of our theorem, we show that the singular locus of the intersection of some regular nilpotent Hessenberg variety $\Hess(N,h_m)$ and $\Omega_e^\circ$ is the intersection of certain Schubert variety and $\Omega_e^\circ$ where $h_m=(m,n,\ldots,n)$ for $1<m<n$. 
We also see that $\Hess(N,h_2) \cap \Omega_e^\circ$ is related with the cyclic quotient singularity.
\end{abstract}

\maketitle

\setcounter{tocdepth}{1}

\tableofcontents

%%%%%%%%%%%%%%%%%%%%%%%%%%%%%%%%%%
\section{Introduction}
\label{section:introduction}
%%%%%%%%%%%%%%%%%%%%%%%%%%%%%%%%%%

Hessenberg varieties are subvarieties of the full flag variety introduced by De Mari and Shayman and studied by De Mari, Procesi, and Shayman \cite{dMPS, dMS}.
These varieties lie in a fruitful intersection of algebraic combinatorics and representation theory, such as hyperplane arrangements (\cite{AHMMS, EHNT1, SomTym}), Stanley's chromatic symmetric functions (\cite{BroCho, Guay-Paquet, KiemLee, ShaWac}), Postnikov's mixed Eulerian numbers (\cite{BST, Hor, NadTew}), and toric orbifolds associated with partitioned weight polytopes (\cite{BalCro, HMSS}) in a recent development. 
In this paper we generalize a result discovered by Dale Peterson to regular nilpotent Hessenberg varieties in type $A$ via the explicit presentation of the quantum cohomology ring of flag varieties given by Ciocan--Fontanine and Givental--Kim.

Let $n$ be a positive integer. 
The (full) flag variety $\Fl(\C^n)$ is the collection of nested sequences of linear subspaces $V_{\bullet} \coloneqq (V_1 \subset V_2 \subset \cdots \subset V_n = \C^n)$ in $\C^n$ where each $V_i$ denotes an $i$-dimensional subspace of $\C^n$. 
Let $N$ be the regular nilpotent matrix in Jordan canonical form, that is, the nilpotent matrix in Jordan form with exactly one Jordan block. 
The \emph{Peterson variety} $\Pet_n$ is defined to be the subvariety of the flag variety $\Fl(\C^n)$ as follows:
\begin{align*}
\Pet_n \coloneqq \{V_\bullet \in \Fl(\C^n) \mid NV_i \subset V_{i+1} \ \textrm{for all} \ i=1,2,\ldots,n-1 \}.
\end{align*}
Dale Peterson has discovered a surprising connection between a geometry of the Peterson variety $\Pet_n$ and the quantum cohomology\footnote{We work with quantum (and ordinary) cohomology with coefficients in $\C$ throughout this paper.} of the flag variety $\Fl(\C^n)$, as explained below.  
Let $B^{-}$ be the set of lower triangular matrices in the general linear group $\GL_n(\C)$.
Let $\Omega_e^\circ$ be the opposite Schubert cell associated with the identity element $e$, which is the $B^{-}$-orbit of the standard flag $F_\bullet=(F_i)_{i}$ where $F_i=\Span_\C\{e_1,\ldots,e_i \}$ and $e_1,\ldots,e_n$ are the standard basis of $\C^n$. 
Note that $\Omega_e^\circ$ is the open chart around the standard flag $F_\bullet=(F_i)_{i}$ in $\Fl(\C^n)$.
Due to Peterson's statements in \cite{Pet}, the coordinate ring of the intersection $\C[\Pet_n \cap \Omega_e^\circ]$ is isomorphic to the quantum cohomology of the flag variety $QH^*(\Fl(\C^n))$ as $\C$-algebras:
\begin{align*}
\C[\Pet_n \cap \Omega_e^\circ] \cong QH^*(\Fl(\C^n)). 
\end{align*}
This incredible result discovered by Peterson is unpublished, so we also refer the reader to \cite{Kos, Rie} for the result above. 
As $\C$-vector spaces, the quantum cohomology $QH^*(\Fl(\C^n))$ is $\C[q_1,\ldots,q_{n-1}] \otimes_\C H^*(\Fl(\C^n))$ where we call $q_1,\ldots,q_{n-1}$ \emph{quantum parameters}. The product structure on $QH^*(\Fl(\C^n))$ is a certain deformation by quantum parameters of the ordinary cup product on $H^*(\Fl(\C^n))$.
More explicitly, Ciocan--Fontanine and Givental--Kim gave in \cite{Font95, GK} an efficient presentation of the quantum cohomology ring $QH^*(\Fl(\C^n))$ in terms of generators and relations as follows.
Let $\check M_n$ be the following matrix
\begin{align} \label{eq:cMn}
\check M_n \coloneqq \left(
 \begin{array}{@{\,}ccccc@{\,}}
     x_1 & q_1 & 0 & \cdots & 0 \\
     -1 & x_2 & q_2 & \ddots & \vdots \\ 
      0 & \ddots & \ddots & \ddots & 0 \\ 
      \vdots & \ddots & -1 & x_{n-1} & q_{n-1} \\
      0 & \cdots & 0 & -1 & x_n 
 \end{array}
 \right).
\end{align}
The \emph{quantized elementary symmetric polynomial} $\check E_i^{(n)} \ (1 \leq i \leq n)$ in the polynomial ring $\C[x_1,\ldots,x_n,q_1,\ldots,q_{n-1}]$ is defined by the coefficient of $\lambda^{n-i}$ for the characteristic polynomial of $\check M_n$ multiplied by $(-1)^i$, i.e. 
\begin{align*}
\det(\lambda I_n - \check M_n) = \lambda^n - \check E_1^{(n)} \lambda^{n-1} + \check E_2^{(n)} \lambda^{n-2} + \cdots + (-1)^n \check E_n^{(n)}.
\end{align*}
Note that by setting $q_s=0$ for all $1 \leq s \leq n-1$ we have that $\check E_i^{(n)}$ is the $i$-th elementary symmetric polynomial in the variables $x_1,\ldots,x_n$. 
Then it is known from \cite{Font95, GK} that there is an isomorphism of $\C$-algebras
\begin{align*} 
\QH^*(\Fl(\C^n)) \cong \C[x_1,\ldots,x_n,q_1,\ldots,q_{n-1}]/(\check E_1^{(n)}, \ldots, \check E_n^{(n)}).
\end{align*}
Combining Peterson's statement and the presentation above, we obtain the isomorphism 
\begin{align} \label{eq:intro_quantum Pet}
\C[\Pet_n \cap \Omega_e^\circ] \cong \C[x_1,\ldots,x_n,q_1,\ldots,q_{n-1}]/(\check E_1^{(n)}, \ldots, \check E_n^{(n)})
\end{align}
as $\C$-algebras.
In this paper we generalize this isomorphism to regular nilpotent Hessenberg varieties by further quantizing the right hand side above.

Consider a nondecreasing function $h: [n] \rightarrow [n]$ such that $h(j) \geq j$ for all $j=1,\ldots,n$ where $[n] \coloneqq \{1,2,\ldots,n\}$, which is called a \emph{Hessenberg function}. 
We frequently write a Hessenberg function $h$ as $h = (h(1), h(2), \ldots , h(n))$. 
The \emph{regular nilpotent Hessenberg variety} $\Hess(N,h)$ associated with a Hessenberg function $h$ is defined as 
\begin{align*}
\Hess(N,h) \coloneqq \{V_\bullet \in \Fl(\C^n) \mid NV_i \subset V_{h(i)} \ \textrm{for all} \ i=1,2,\ldots,n \}.
\end{align*}
This object is a generalization of the Peterson variety $\Pet_n$ since $\Hess(N,h)$ is equal to $\Pet_n$ whenever $h=(2,3,4,\ldots,n,n)$.
We also note that if $h=(n,n,\ldots,n)$, then $\Hess(N,h)=\Fl(\C^n)$ by definition.
Recall that the flag variety $\Fl(\C^n)$ can be identified with $\GL_n(\C)/B$ where $B$ is the set of upper triangular matrices in $\GL_n(\C)$ so that the first $j$ column vectors of a matrix $g \in \GL_n(\C)$ generate the $j$-th vector space $V_j$ for $j \in [n]$.
Under the identification $\Fl(\C^n) \cong \GL_n(\C)/B$, one can write 
\begin{align*} 
\Hess(N,h) = \{ gB \in \GL_n(\C)/B \mid g^{-1} N g \in H(h) \},
\end{align*}
where $H(h)$ is the set of matrices $(a_{ij})_{i,j \in [n]}$ such that $a_{ij} = 0$ if $i > h(j)$. 
(Note that a matrix $(a_{ij})_{i,j \in [n]}$ in $H(h)$ is not necessarily invertible.)
Since $\Omega_e^\circ$ is an affine open set in $\Fl(\C^n)$ which is naturally identified with the set of lower triangular unipotent matrices, the intersection $\Hess(N,h) \cap \Omega_e^\circ$ is described as 
\begin{align} \label{eq:intro_Hess(N,h) cap Omega}
\Hess(N,h) \cap \Omega_e^\circ = \left\{ g=\left(
 \begin{array}{@{\,}ccccc@{\,}}
     1 &  &  &  &  \\
     x_{21} & 1 &  &  &  \\ 
     x_{31} & x_{32} & 1 &  &  \\ 
     \vdots& \vdots & \ddots & \ddots & \\
     x_{n1} & x_{n2} & \cdots & x_{n \, n-1} & 1 
 \end{array}
 \right) \middle| (g^{-1}Ng)_{ij}=0 \ \textrm{for all} \atop \ j \in [n-1] \ \textrm{and} \ h(j) < i \leq n \right\}. 
\end{align}
Motivated by this, we set
\begin{align*} 
\ZZ(N,h)_e \coloneqq  \Spec \C[x_{ij} \mid 1 \leq j < i \leq n]/\big((g^{-1}Ng)_{ij} \mid j \in [n-1] \ \textrm{and} \ h(j) < i \leq n \big). 
\end{align*}
We remark that this scheme can be regarded as a zero scheme of some section of certain vector bundle over $\GL_n(\C)/B$, which is introduced in \cite{ADGH}. 
See Section~\ref{section:Regular nilpotent Hessenberg varieties} for the details. 

We now generalize the matrix $\check M_n$ in \eqref{eq:cMn} to the following matrix
\begin{align} \label{eq:intro_Mn}
M_n \coloneqq \left(
 \begin{array}{@{\,}ccccc@{\,}}
     x_1 & q_{12} & q_{13} & \cdots & q_{1n} \\
     -1 & x_2 & q_{23} & \cdots & q_{2n} \\ 
      0 & \ddots & \ddots & \ddots & \vdots \\ 
     \vdots & \ddots & -1 & x_{n-1} & q_{n-1 \, n} \\
     0 & \cdots & 0 & -1 & x_n 
 \end{array}
 \right).
\end{align}
Then we define the \emph{$q_{rs}$-quantized elementary symmetric polynomial} $E_i^{(n)} \ (1 \leq i \leq n)$ in the polynomial ring $\C[x_1,\ldots,x_n,q_{rs} \mid 1 \leq r < s \leq n]$ by the coefficient of $\lambda^{n-i}$ for the characteristic polynomial of $M_n$ multiplied by $(-1)^i$, namely 
\begin{align*}
\det(\lambda I_n - M_n) = \lambda^n - E_1^{(n)} \lambda^{n-1} + E_2^{(n)} \lambda^{n-2} + \cdots + (-1)^n E_n^{(n)}.
\end{align*}
Note that by setting $q_{rs}=0$ for $s-r >1$ and $q_{s \, s+1}=q_s$, our polynomial $E_i^{(n)}$ is the (classical) quantized elementary symmetric polynomial $\check E_i^{(n)}$ in $\C[x_1,\ldots,x_n,q_1,\ldots,q_{n-1}]$.
For a Hessenberg function $h: [n] \rightarrow [n]$, we define $\hE_i^{(n)}$ as the polynomial $E_i^{(n)}$ by setting $q_{rs}=0$ for all $2 \leq s \leq n$ and $1 \leq r \leq n-h(n+1-s)$:  
\begin{align*}
\hE_i^{(n)} \coloneqq E_i^{(n)}|_{q_{rs}=0 \ (2 \leq s \leq n \ \textrm{and} \ 1 \leq r \leq n-h(n+1-s))}. 
\end{align*}
We will pictorially explain which variables $q_{rs}$ are set to $0$ in the definition of $\hE_i^{(n)}$ in Example~\ref{example:qrs_(3,3,4,5,5)} and surrounding discussion. 
The main theorem of this paper is as follows:

\begin{theorem} \label{theorem:intro_main}
Let $h: [n] \rightarrow [n]$ be a Hessenberg function. 
Then there is an isomorphism of $\C$-algebras
\begin{align} \label{eq:intro_main} 
\Gamma(\ZZ(N,h)_e, \mathcal{O}_{\ZZ(N,h)_e}) \cong \frac{\C[x_1,\ldots,x_n, q_{rs} \mid 2 \leq s \leq n, n-h(n+1-s)<r<s]}{(\hE_1^{(n)}, \ldots, \, \hE_n^{(n)})}, 
\end{align}
where $\Gamma(\ZZ(N,h)_e, \mathcal{O}_{\ZZ(N,h)_e})$ is the set of global sections. 
\end{theorem}

Note that we explicitly construct the isomorphism above. See Theorem~\ref{theorem:main1} and Proposition~\ref{proposition:inverse_map_varphih} for the correspondence.
In particular, we see that our quantum parameters $q_{rs}$'s correspond to polynomials which define regular nilpotent Hessenberg varieties (up to signs).  
See Corollary~\ref{corollary:main_commutative_diagram}.

As is well-known, the cohomology ring $H^*(\Fl(\C^n))$ is isomorphic to the quotient of the polynomial ring $\C[x_1,\ldots,x_n]$ by the ideal generated by elementary symmetric polynomials, so the following is a consequence of Theorem~\ref{theorem:intro_main}.

\begin{corollary}
There is an isomorphism of $\C$-algebras
\begin{align*}
\Gamma(\ZZ(N,id)_e, \mathcal{O}_{\ZZ(N,id)_e}) \cong H^*(\Fl(\C^n)).
\end{align*}
\end{corollary}

We say that a Hessenberg function $h: [n] \rightarrow [n]$ is \emph{indecomposable} if it satisfies $h(j)>j$ for all $j \in [n-1]$. 
It is known from \cite[Proposition~3.6]{ADGH} that if $h$ is indecomposable, then the affine scheme $\ZZ(N,h)_e$ is reduced.
Therefore, we can conclude the following result from Theorem~\ref{theorem:intro_main}.

\begin{corollary} \label{corollary:intro_main}
If $h: [n] \rightarrow [n]$ is an indecomposable Hessenberg function, then there is an isomorphism of $\C$-algebras
\begin{align*} 
\C[\Hess(N,h) \cap \Omega_e^\circ] \cong \frac{\C[x_1,\ldots,x_n, q_{rs} \mid 2 \leq s \leq n, n-h(n+1-s)<r<s]}{(\hE_1^{(n)}, \ldots, \, \hE_n^{(n)})}, 
\end{align*}
where $\C[\Hess(N,h) \cap \Omega_e^\circ]$ is the coordinate ring of the open set $\Hess(N,h) \cap \Omega_e^\circ$ in $\Hess(N,h)$. 
\end{corollary}

The isomorphism in Corollary~\ref{corollary:intro_main} is a natural generalization of \eqref{eq:intro_quantum Pet} since the Hessenberg function $h=(2,3,4,\ldots,n,n)$ which defines the Peterson variety $\Pet_n$ is indecomposable.
In particular, our proof gives an elementary proof for Peterson's statement via the presentation for the quantum cohomology $QH^*(\Fl(\C^n))$ given by \cite{Font95, GK}.

We next apply Corollary~\ref{corollary:intro_main} to the study of the singular locus of the open set $\Hess(N,h) \cap \Omega_e^\circ$ in $\Hess(N,h)$ for some Hessenberg functions $h$. 
There are partial results for studying singularities of Hessenberg varieties in \cite{AbeInsko, EPS, IY}. 
Indeed, an explicit presentation of the singular locus for the Peteson variety $\Pet_n$ is given by \cite{IY}. 
The singular locus for \emph{nilpotent} Hessenberg varieties of codimension one in the flag variety $\Fl(\C^n)$ is explicitly described in \cite{EPS}. 
Also, a recent paper \cite{AbeInsko} combinatorially determines which permutation flags in arbitrary regular nilpotent Hessenberg variety $\Hess(N,h)$ are singular points.
We focus on the following Hessenberg function
\begin{align} \label{eq:intro_hm}
h_m = (m,n,\ldots,n) \ \ \ \textrm{for} \ 2 \leq m \leq n-1.
\end{align}
We derive an explicit presentation of the singular locus of $\Hess(N,h_m)$ from Corollary~\ref{corollary:intro_main}.  
For this purpose, we first study the singular locus of $\Hess(N,h_2)$.
More precisely, we show that the singularity of $\Hess(N,h_2) \cap \Omega_e^\circ$ is related with a cyclic quotient singularity. 
We briefly explain the cyclic quotient singularity.
Let $\Cyclic_n$ be the cyclic group of order $n$ generated by $\zeta$ where $\zeta$ is a primitive $n$-th root of unity. 
Define the action of $\Cyclic_n$ on $\C^2$ by $\zeta \cdot (x,y)=(\zeta x, \zeta^{-1} y)$ for $\zeta \in \Cyclic_n$ and $(x,y) \in \C^2$.
Then, the quotient space $\C^2/\Cyclic_n$ is called the \emph{cyclic quotient singularity} or the \emph{type $A_{n-1}$-singularity}.

\begin{theorem} \label{theorem:intro_cyclic_quotient_h2}
There is an isomorphism 
\begin{align*}
\Hess(N,h_2) \cap \Omega_e^\circ \cong \C^2/\Cyclic_n \times \C^{\frac{1}{2}(n-1)(n-2)-1}
\end{align*}
where $h_2=(2,n,\ldots,n)$. 
\end{theorem}

Recall that $\Hess(N,h_m) \cap \Omega_e^\circ$ is given in \eqref{eq:intro_Hess(N,h) cap Omega}. 
As a corollary of Theorem~\ref{theorem:intro_cyclic_quotient_h2}, one can give the singular locus of $\Hess(N,h_2) \cap \Omega_e^\circ$ as the solution set of the equations $x_{i1}=0$ for $2 \leq i \leq n$ and $x_{n2}=0$. 
Combining this description and Corollary~\ref{corollary:intro_main}, we can explicitly describe the singular locus of $\Hess(N,h_m) \cap \Omega_e^\circ$ as follows. 

\begin{theorem} \label{theorem:intro_singular_locus_hm}
Let $h_m$ be the Hessenberg function defined in \eqref{eq:intro_hm} for $2 \leq m \leq n-1$.
Then, the singular locus of $\Hess(N,h_m) \cap \Omega_e^\circ$ is described as 
\begin{align*} 
\left\{ g=\left(
 \begin{array}{@{\,}ccccc@{\,}}
     1 &  &  &  &  \\
     x_{21} & 1 &  &  &  \\ 
     x_{31} & x_{32} & 1 &  &  \\ 
     \vdots& \vdots & \ddots & \ddots & \\
     x_{n1} & x_{n2} & \cdots & x_{n \, n-1} & 1 
 \end{array}
 \right) \middle| x_{i1} = 0 \ \textrm{for all} \ 2 \leq i \leq n \atop \ \textrm{and} \ x_{nj} = 0 \ \textrm{for all} \ 2 \leq j \leq m \right\}. 
\end{align*}
\end{theorem}

The explicit presentation in Theorem~\ref{theorem:intro_singular_locus_hm} is certain Schubert variety in the open set $\Omega_e^\circ$, as explained below.  
Let $\SS_n$ be the permutation group on $[n]$. 
For $w \in \SS_n$, the Schubert variety $X_w$ is defined by 
\begin{align*}
X_w=\{V_\bullet \in \Fl(\C^n) \mid \dim(V_p \cap F_q) \geq |\{i \in [p] \mid w(i) \leq q \}| \ \textrm{for all} \ p,q \in [n] \}
\end{align*}
where $F_\bullet=(F_i)_{i}$ is the standard flag, i.e. $F_i=\Span_\C\{e_1,\ldots,e_i \}$ and $e_1,\ldots,e_n$ are the standard basis of $\C^n$. 
For $2 \leq m \leq n-1$, we define the permutation $w_m \in \SS_n$ by
\begin{align*}
w_m \coloneqq 1 \ \ n-1 \ \ n-2 \ \cdots \ n-m+1 \ \ n \ \ n-m \ \ n-m-1 \ \cdots \ 2 
\end{align*}
in one-line notation. (See also \eqref{eq:_definition_wm}.)
Then one can see that 
\begin{align*}
X_{w_m} = \{V_\bullet \in \Fl(\C^n) \mid V_1 = F_1 \ \textrm{and} \ V_m \subset F_{n-1} \},
\end{align*} 
so we obtain the following result from Theorem~\ref{theorem:intro_singular_locus_hm}.

\begin{corollary} \label{corollary:intro_singular_locus_hm}
Let $2 \leq m \leq n-1$.
Then, the singular locus of $\Hess(N,h_m) \cap \Omega_e^\circ$ is equal to
\begin{align*}
\Sing(\Hess(N,h_m) \cap \Omega_e^\circ) = X_{w_m} \cap \Omega_e^\circ.
\end{align*}
\end{corollary}

The paper is organized as follows. After reviewing the definition of Hessenberg varieties
and their defining equations in Section~\ref{section:Hessenberg varieties}, we focus on regular nilpotent Hessenberg varieties in Section~\ref{section:Regular nilpotent Hessenberg varieties}.
Then we state the main theorem (Theorem~\ref{theorem:main1}) in Section~\ref{section:The main theorem1}. 
In Section~\ref{section:Properties of $E_r^{(s)}$} we show that the homomorphism \eqref{eq:intro_main} is well-defined and surjective. 
In order to prove that it is in fact an isomorphism, we use the commutative algebra's tool of Hilbert series and regular sequences. More specifically, we define certain degrees for the variables $\{x_{ij} \mid 1 \leq j < i \leq n \}$ appeared in \eqref{eq:intro_Hess(N,h) cap Omega} and $\{x_1,\ldots,x_n, q_{rs} \mid 1 \leq r < s \leq n \}$ introduced in \eqref{eq:intro_Mn} so that the two sides of \eqref{eq:intro_main} are graded $\C$-algebras in Section~\ref{section:Hilbert series}.
We give a proof of our main theorem in Section~\ref{section:Proof of main theorem 1}. 
Next, turning our attention to the singular locus of $\Hess(N,h_m) \cap \Omega_e^\circ$ where $h_m$ is defined in \eqref{eq:intro_hm}, we give an explicit formula for partial derivatives $\frac{\partial}{\partial x_s} \, \hE_{i}^{(n)}$ and $\frac{\partial}{\partial q_{rs}} \, \hE_{i}^{(n)}$ in Section~\ref{section:Jacobian matrix}.
Then we relate the singularity of $\Hess(N,h_2) \cap \Omega_e^\circ$ to the cyclic quotient singularity (Theorem~\ref{theorem:cyclic_quotient_h2}) in Section~\ref{section:Cyclic quotient}.
One can see that this fact yields an explicit description for the singular locus of $\Hess(N,h_2) \cap \Omega_e^\circ$ (Corollary~\ref{corollary:singular_locus_h2}).
In Section~\ref{section:Singular locus of Hess(N,hm)} we generalize this result to the singular locus of $\Hess(N,h_m) \cap \Omega_e^\circ$ (Theorem~\ref{theorem:singular_locus_hm}) by using our main theorem together with the computations for partial derivatives $\frac{\partial}{\partial x_s} \, \hE_{i}^{(n)}$ and $\frac{\partial}{\partial q_{rs}} \, \hE_{i}^{(n)}$.
We also see that the singular locus of $\Hess(N,h_m) \cap \Omega_e^\circ$ is equal to the intersection of the Schubert variety $X_{w_m}$ and $\Omega_e^\circ$ (Corollary~\ref{corollary:singular_locus_hm}).

\bigskip
\noindent \textbf{Acknowledgements.} 
We are grateful to Hiraku Abe for valuable discussions. 
We are also grateful to Satoshi Murai for helpful comments on the commutative algebra arguments.
%We thank the anonymous referee for valuable and perceptive comments.
This research was partly supported by Osaka Central Advanced Mathematical Institute (MEXT Joint Usage/Research Center on Mathematics and Theoretical Physics).
The first author is supported in part by JSPS Grant-in-Aid for Young Scientists: 19K14508.

\bigskip

%%%%%%%%%%%%%%%%%%%%%%%%%%%%%%%%%%
\section{Hessenberg varieties}
\label{section:Hessenberg varieties}
%%%%%%%%%%%%%%%%%%%%%%%%%%%%%%%%%%

In this section we recall the definitions of Hessenberg varieties in type $A_{n-1}$ and their defining equations.
We use the notation $[n] \coloneqq \{1,2,\ldots,n\}$ throughout this paper.

Let $n$ be a positive integer.
A \emph{Hessenberg function} is a function $h: [n] \to [n]$ satisfying the following two conditions
\begin{enumerate}
\item $h(1) \leq h(2) \leq \cdots \leq h(n)$; 
\item $h(j) \geq j$ for all $j \in [n]$.
\end{enumerate}
Note that $h(n)=n$ by definition.
We frequently denote a Hessenberg function by listing its values in sequence, namely $h = (h(1), h(2), \ldots , h(n))$. 
It is useful to see a Hessenberg function $h$ pictorially by drawing a configuration of boxes on a square grid of size $n \times n$ whose shaded boxes consist of boxes in the $i$-th row and the $j$-th column such that $i \leq h(j)$ for $i, j \in [n]$.

\begin{example} \label{example:Hess_func_(3,3,4,5,5)}
Let $n=5$. Then, $h=(3,3,4,5,5)$ is a Hessenberg function and the configuration of the shaded boxes is shown in Figure~\ref{picture:Hessenberg_function}.
\begin{figure}[h]
\begin{center}
\begin{picture}(75,75)
\put(0,63){\colorbox{gray}}
\put(0,67){\colorbox{gray}}
\put(0,72){\colorbox{gray}}
\put(4,63){\colorbox{gray}}
\put(4,67){\colorbox{gray}}
\put(4,72){\colorbox{gray}}
\put(9,63){\colorbox{gray}}
\put(9,67){\colorbox{gray}}
\put(9,72){\colorbox{gray}}

\put(15,63){\colorbox{gray}}
\put(15,67){\colorbox{gray}}
\put(15,72){\colorbox{gray}}
\put(19,63){\colorbox{gray}}
\put(19,67){\colorbox{gray}}
\put(19,72){\colorbox{gray}}
\put(24,63){\colorbox{gray}}
\put(24,67){\colorbox{gray}}
\put(24,72){\colorbox{gray}}

\put(30,63){\colorbox{gray}}
\put(30,67){\colorbox{gray}}
\put(30,72){\colorbox{gray}}
\put(34,63){\colorbox{gray}}
\put(34,67){\colorbox{gray}}
\put(34,72){\colorbox{gray}}
\put(39,63){\colorbox{gray}}
\put(39,67){\colorbox{gray}}
\put(39,72){\colorbox{gray}}

\put(45,63){\colorbox{gray}}
\put(45,67){\colorbox{gray}}
\put(45,72){\colorbox{gray}}
\put(49,63){\colorbox{gray}}
\put(49,67){\colorbox{gray}}
\put(49,72){\colorbox{gray}}
\put(54,63){\colorbox{gray}}
\put(54,67){\colorbox{gray}}
\put(54,72){\colorbox{gray}}

\put(60,63){\colorbox{gray}}
\put(60,67){\colorbox{gray}}
\put(60,72){\colorbox{gray}}
\put(64,63){\colorbox{gray}}
\put(64,67){\colorbox{gray}}
\put(64,72){\colorbox{gray}}
\put(69,63){\colorbox{gray}}
\put(69,67){\colorbox{gray}}
\put(69,72){\colorbox{gray}}

\put(0,48){\colorbox{gray}}
\put(0,52){\colorbox{gray}}
\put(0,57){\colorbox{gray}}
\put(4,48){\colorbox{gray}}
\put(4,52){\colorbox{gray}}
\put(4,57){\colorbox{gray}}
\put(9,48){\colorbox{gray}}
\put(9,52){\colorbox{gray}}
\put(9,57){\colorbox{gray}}

\put(15,48){\colorbox{gray}}
\put(15,52){\colorbox{gray}}
\put(15,57){\colorbox{gray}}
\put(19,48){\colorbox{gray}}
\put(19,52){\colorbox{gray}}
\put(19,57){\colorbox{gray}}
\put(24,48){\colorbox{gray}}
\put(24,52){\colorbox{gray}}
\put(24,57){\colorbox{gray}}

\put(30,48){\colorbox{gray}}
\put(30,52){\colorbox{gray}}
\put(30,57){\colorbox{gray}}
\put(34,48){\colorbox{gray}}
\put(34,52){\colorbox{gray}}
\put(34,57){\colorbox{gray}}
\put(39,48){\colorbox{gray}}
\put(39,52){\colorbox{gray}}
\put(39,57){\colorbox{gray}}

\put(45,48){\colorbox{gray}}
\put(45,52){\colorbox{gray}}
\put(45,57){\colorbox{gray}}
\put(49,48){\colorbox{gray}}
\put(49,52){\colorbox{gray}}
\put(49,57){\colorbox{gray}}
\put(54,48){\colorbox{gray}}
\put(54,52){\colorbox{gray}}
\put(54,57){\colorbox{gray}}

\put(60,48){\colorbox{gray}}
\put(60,52){\colorbox{gray}}
\put(60,57){\colorbox{gray}}
\put(64,48){\colorbox{gray}}
\put(64,52){\colorbox{gray}}
\put(64,57){\colorbox{gray}}
\put(69,48){\colorbox{gray}}
\put(69,52){\colorbox{gray}}
\put(69,57){\colorbox{gray}}

\put(0,33){\colorbox{gray}}
\put(0,37){\colorbox{gray}}
\put(0,42){\colorbox{gray}}
\put(4,33){\colorbox{gray}}
\put(4,37){\colorbox{gray}}
\put(4,42){\colorbox{gray}}
\put(9,33){\colorbox{gray}}
\put(9,37){\colorbox{gray}}
\put(9,42){\colorbox{gray}}

\put(15,33){\colorbox{gray}}
\put(15,37){\colorbox{gray}}
\put(15,42){\colorbox{gray}}
\put(19,33){\colorbox{gray}}
\put(19,37){\colorbox{gray}}
\put(19,42){\colorbox{gray}}
\put(24,33){\colorbox{gray}}
\put(24,37){\colorbox{gray}}
\put(24,42){\colorbox{gray}}

\put(30,33){\colorbox{gray}}
\put(30,37){\colorbox{gray}}
\put(30,42){\colorbox{gray}}
\put(34,33){\colorbox{gray}}
\put(34,37){\colorbox{gray}}
\put(34,42){\colorbox{gray}}
\put(39,33){\colorbox{gray}}
\put(39,37){\colorbox{gray}}
\put(39,42){\colorbox{gray}}

\put(45,33){\colorbox{gray}}
\put(45,37){\colorbox{gray}}
\put(45,42){\colorbox{gray}}
\put(49,33){\colorbox{gray}}
\put(49,37){\colorbox{gray}}
\put(49,42){\colorbox{gray}}
\put(54,33){\colorbox{gray}}
\put(54,37){\colorbox{gray}}
\put(54,42){\colorbox{gray}}

\put(60,33){\colorbox{gray}}
\put(60,37){\colorbox{gray}}
\put(60,42){\colorbox{gray}}
\put(64,33){\colorbox{gray}}
\put(64,37){\colorbox{gray}}
\put(64,42){\colorbox{gray}}
\put(69,33){\colorbox{gray}}
\put(69,37){\colorbox{gray}}
\put(69,42){\colorbox{gray}}

%\put(0,18){\colorbox{gray}}
%\put(0,22){\colorbox{gray}}
%\put(0,27){\colorbox{gray}}
%\put(4,18){\colorbox{gray}}
%\put(4,22){\colorbox{gray}}
%\put(4,27){\colorbox{gray}}
%\put(9,18){\colorbox{gray}}
%\put(9,22){\colorbox{gray}}
%\put(9,27){\colorbox{gray}}

%\put(15,18){\colorbox{gray}}
%\put(15,22){\colorbox{gray}}
%\put(15,27){\colorbox{gray}}
%\put(19,18){\colorbox{gray}}
%\put(19,22){\colorbox{gray}}
%\put(19,27){\colorbox{gray}}
%\put(24,18){\colorbox{gray}}
%\put(24,22){\colorbox{gray}}
%\put(24,27){\colorbox{gray}}

\put(30,18){\colorbox{gray}}
\put(30,22){\colorbox{gray}}
\put(30,27){\colorbox{gray}}
\put(34,18){\colorbox{gray}}
\put(34,22){\colorbox{gray}}
\put(34,27){\colorbox{gray}}
\put(39,18){\colorbox{gray}}
\put(39,22){\colorbox{gray}}
\put(39,27){\colorbox{gray}}

\put(45,18){\colorbox{gray}}
\put(45,22){\colorbox{gray}}
\put(45,27){\colorbox{gray}}
\put(49,18){\colorbox{gray}}
\put(49,22){\colorbox{gray}}
\put(49,27){\colorbox{gray}}
\put(54,18){\colorbox{gray}}
\put(54,22){\colorbox{gray}}
\put(54,27){\colorbox{gray}}

\put(60,18){\colorbox{gray}}
\put(60,22){\colorbox{gray}}
\put(60,27){\colorbox{gray}}
\put(64,18){\colorbox{gray}}
\put(64,22){\colorbox{gray}}
\put(64,27){\colorbox{gray}}
\put(69,18){\colorbox{gray}}
\put(69,22){\colorbox{gray}}
\put(69,27){\colorbox{gray}}

%\put(0,3){\colorbox{gray}}
%\put(0,7){\colorbox{gray}}
%\put(0,12){\colorbox{gray}}
%\put(4,3){\colorbox{gray}}
%\put(4,7){\colorbox{gray}}
%\put(4,12){\colorbox{gray}}
%\put(9,3){\colorbox{gray}}
%\put(9,7){\colorbox{gray}}
%\put(9,12){\colorbox{gray}}

%\put(15,3){\colorbox{gray}}
%\put(15,7){\colorbox{gray}}
%\put(15,12){\colorbox{gray}}
%\put(19,3){\colorbox{gray}}
%\put(19,7){\colorbox{gray}}
%\put(19,12){\colorbox{gray}}
%\put(24,3){\colorbox{gray}}
%\put(24,7){\colorbox{gray}}
%\put(24,12){\colorbox{gray}}

%\put(30,3){\colorbox{gray}}
%\put(30,7){\colorbox{gray}}
%\put(30,12){\colorbox{gray}}
%\put(34,3){\colorbox{gray}}
%\put(34,7){\colorbox{gray}}
%\put(34,12){\colorbox{gray}}
%\put(39,3){\colorbox{gray}}
%\put(39,7){\colorbox{gray}}
%\put(39,12){\colorbox{gray}}

\put(45,3){\colorbox{gray}}
\put(45,7){\colorbox{gray}}
\put(45,12){\colorbox{gray}}
\put(49,3){\colorbox{gray}}
\put(49,7){\colorbox{gray}}
\put(49,12){\colorbox{gray}}
\put(54,3){\colorbox{gray}}
\put(54,7){\colorbox{gray}}
\put(54,12){\colorbox{gray}}

\put(60,3){\colorbox{gray}}
\put(60,7){\colorbox{gray}}
\put(60,12){\colorbox{gray}}
\put(64,3){\colorbox{gray}}
\put(64,7){\colorbox{gray}}
\put(64,12){\colorbox{gray}}
\put(69,3){\colorbox{gray}}
\put(69,7){\colorbox{gray}}
\put(69,12){\colorbox{gray}}

\put(0,0){\framebox(15,15)}
\put(15,0){\framebox(15,15)}
\put(30,0){\framebox(15,15)}
\put(45,0){\framebox(15,15)}
\put(60,0){\framebox(15,15)}
\put(0,15){\framebox(15,15)}
\put(15,15){\framebox(15,15)}
\put(30,15){\framebox(15,15)}
\put(45,15){\framebox(15,15)}
\put(60,15){\framebox(15,15)}
\put(0,30){\framebox(15,15)}
\put(15,30){\framebox(15,15)}
\put(30,30){\framebox(15,15)}
\put(45,30){\framebox(15,15)}
\put(60,30){\framebox(15,15)}
\put(0,45){\framebox(15,15)}
\put(15,45){\framebox(15,15)}
\put(30,45){\framebox(15,15)}
\put(45,45){\framebox(15,15)}
\put(60,45){\framebox(15,15)}
\put(0,60){\framebox(15,15)}
\put(15,60){\framebox(15,15)}
\put(30,60){\framebox(15,15)}
\put(45,60){\framebox(15,15)}
\put(60,60){\framebox(15,15)}
\end{picture}
\end{center}
\vspace{-10pt}
\caption{The configuration corresponding to $h=(3,3,4,5,5)$.}
\label{picture:Hessenberg_function}
\end{figure}
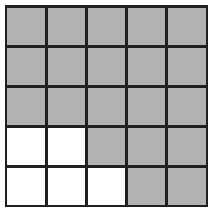
\end{example}

Let $\gl_n(\C)$ be the set of $n \times n$ matrices.
For a Hessenberg function $h$, we define 
\begin{align} \label{eq:Hessenberg space}
H(h) \coloneqq \{(a_{ij})_{i,j \in [n]} \in \gl_n(\C) \mid a_{ij} = 0 \ \textrm{if} \ i > h(j) \},
\end{align}
which is called the \emph{Hessenberg space associated to $h$}.

The (full) flag variety $\Fl(\C^n)$ is the set of nested complex linear subspaces of $\C^n$:
\begin{align*}
\Fl(\C^n) \coloneqq \{ V_\bullet \coloneqq (V_1 \subset V_2 \subset \dots \subset V_n = \C^n) \mid \dim_\C V_i = i \ \textrm{for all} \ i \in [n] \}.
\end{align*}
Let $B$ be the set of upper triangular matrices in the general linear group $\GL_n(\C)$. 
As is well-known, the flag variety $\Fl(\C^n)$ can be identified with $\GL_n(\C)/B$.
Indeed, each flag $V_\bullet \in \Fl(\C^n)$ is determined by a matrix $g$ whose first $j$ column vectors generate the $j$-th vector space $V_j$ for $j \in [n]$.
The correspondence $gB \mapsto V_{\bullet}$ above gives the identification $\Fl(\C^n) \cong \GL_n(\C)/B$. 
For a linear operator $\XX: \C^n \rightarrow \C^n$ and a Hessenberg function $h: [n] \to [n]$, the Hessenberg variety $\Hess(\XX,h)$ is defined to be the following subvariety of the flag variety $\Fl(\C^n)$:
\begin{align} \label{eq:Hessenberg_variety_flag}
\Hess(\XX,h) \coloneqq \{ V_\bullet \in \Fl(\C^n) \mid \XX V_i \subset V_{h(i)} \ \textrm{for all} \ i \in [n] \}.
\end{align}
Hessenberg varieties are introduced in \cite{dMPS, dMS}.
Note that if $h=(n,n,\ldots,n)$, then $\Hess(\XX,h)=\Fl(\C^n)$.
By identifying $\Fl(\C^n)$ with $\GL_n(\C)/B$, we can write the definition above as 
\begin{align} \label{eq:Hessenberg_variety}
\Hess(\XX,h) = \{ gB \in \GL_n(\C)/B \mid g^{-1} \XX g \in H(h) \},
\end{align}
where $H(h)$ is the Hessenberg space defined in \eqref{eq:Hessenberg space}. 

\begin{definition} \label{definition:FijX}
Let $\XX \in \gl_n(\C)$ and $i,j \in [n]$. 
We define a polynomial $F_{i,j}^{\XX}$ on $\GL_n(\C)$ by 
\begin{align*} 
F_{i,j}^{\XX}(g) \coloneqq \det (v_1 \, \ldots \, v_{i-1} \ \XX v_j \ v_{i+1} \, \ldots \, v_n),
\end{align*}
where $g=(v_1 \, \ldots \, v_n) \in \GL_n(\C)$ is the decomposition into column vectors.
In other words, the polynomial $F_{i,j}^{\XX}(g)$ is the determinant of the matrix obtained from $g$ by replacing the $i$-th column vector of $g$ to the $j$-th column vector of $\XX g$.
\end{definition}

\begin{lemma} \label{lemma:defining equation}
Let $\XX \in \gl_n(\C)$. 
For $g \in \GL_n(\C)$ and $i,j \in [n]$, we have
\begin{align*}
\frac{1}{\det(g)}F_{i,j}^{\XX}(g) = (g^{-1} \XX g)_{ij}
\end{align*}
where $(g^{-1} \XX g)_{ij}$ denotes the $(i,j)$-th entry of the matrix $g^{-1} \XX g$. 
In particular, we have
\begin{align*} 
\Hess(\XX,h) =\{ gB \in \GL_n(\C)/B \mid F_{i,j}^{\XX}(g)=0 \ \textrm{for all} \ 1 \leq j \leq n-1 \ \textrm{and} \ h(j) < i \leq n \}.
\end{align*}
\end{lemma}

\begin{proof}
Let $c_{ij}$ be the $(i,j)$ cofactor of $g$, namely $c_{ij}$ is obtained by multiplying the $(i,j)$ minor of $g$ by $(-1)^{i+j}$ . 
Set $\tilde g=(c_{ij})_{i,j \in [n]}^t=(c_{ji})_{i,j \in [n]}$. 
Since $g^{-1}=\frac{1}{\det(g)} \, \tilde g$, it suffices to show that 
\begin{align} \label{eq:cofactor}
F_{i,j}^{\XX}(g) = (\tilde g \XX g)_{ij}.
\end{align}
We write $g=(v_1 \, \ldots \, v_n)$ and $\XX v_j=(b_{1j},\ldots,b_{nj})^t$.
By the definition of $\tilde g$ we have
\begin{align*}
(\tilde g \cdot \XX g)_{ij} = \sum_{k=1}^n c_{ki} b_{kj} = \det (v_1 \, \ldots \, v_{i-1} \ \XX v_j \ v_{i+1} \,\ldots \, v_n)
\end{align*}
where we used the cofactor expansion along the $i$-th column for the last equality.
This yields \eqref{eq:cofactor} as desired.
The latter statement follows from the former statement and \eqref{eq:Hessenberg_variety}.
\end{proof}

\bigskip

%%%%%%%%%%%%%%%%%%%%%%%%%%%%%%%%%%
\section{Regular nilpotent Hessenberg varieties}
\label{section:Regular nilpotent Hessenberg varieties}
%%%%%%%%%%%%%%%%%%%%%%%%%%%%%%%%%%

In this section we review geometric properties for regular nilpotent Hessenberg varieties.

Let $N$ be a regular nilpotent element in $\gl_n(\C)$, i.e. a matrix whose Jordan canonical form consists of exactly one Jordan block with corresponding eigenvalue equal to $0$.
For a Hessenberg function $h$, $\Hess(N,h)$ is called the \emph{regular nilpotent Hessenberg variety}.
If $h=(2,3,4,\ldots,n,n)$, then $\Pet_n \coloneqq \Hess(N,h=(2,3,4,\ldots,n,n))$ is called the Peterson variety, which is an object of an intense study by Dale Peterson \cite{Pet}.
Surprisingly, the Peterson variety is related with the quantum cohomology of the flag variety (\cite{Kos,Rie}). 
We will explain the relation in Section~\ref{section:The main theorem1}.
For any $p \in \GL_n(\C)$, one can see that $\Hess(N,h) \cong \Hess(p^{-1}Np,h)$ which sends $gB$ to $p^{-1}gB$.
Thus, we may assume that $N$ is in Jordan form:
\begin{equation} \label{eq:regular nilpotent} 
N = 
\begin{pmatrix}
0 & 1 & &   \\
     & \ddots & \ddots & \\
     &   & 0 & 1 \\
     &  & & 0 \\ 
\end{pmatrix}. 
\end{equation}

We summarize geometric properties of $\Hess(N,h)$ as follows.

\begin{theorem} [\cite{AndTym, IY, Kos, SomTym}] \label{theorem:property_Hess(N,h)} 
Let $\Pet_n$ be the Peterson variety. 
Let $\Hess(N,h)$ be the regular nilpotent Hessenberg variety associated with a Hessenberg function $h$.
\begin{enumerate}
\item[(i)] The Peterson variety $\Pet_n$ is singular for $n \geq 3$. Also, $\Pet_n$ is normal if and only if $n \leq 3$.  
\item[(ii)] The complex dimension of $\Hess(N,h)$ is given by $\sum_{j=1}^n (h(j)-j)$. 
In particular, we have $\dim_\C \Pet_n = n-1$.
\item[(iii)] $\Hess(N,h)$ is irreducible.
\end{enumerate}
\end{theorem}

As a property of regular nilpotent Hessenberg varieties, a special case of $\Hess(N,h)$ can be decomposed into a product of smaller regular nilpotent Hessenberg varieties, as explained below.
To explain this, we first recall the following terminology from \cite[Definition~4.4]{Dre-Thesis}.

\begin{definition} \label{definition:indecomposable}
A Hessenberg function $h$ is \emph{decomposable} if $h(j)=j$ for some $j \in [n-1]$.
A Hessenberg function $h$ is \emph{indecomposable} if $h(j)>j$ for all $j \in [n-1]$.
Note that an indecomposable Hessenberg function $h$ satisfies $h(n-1)=h(n)=n$.
\end{definition}

If $h$ is a decomposable Hessenberg function, i.e. $h(j)=j$ for some $j \in [n-1]$, then the Hessenberg function $h$ can be decomposed into two smaller Hessenberg functions $h_1$ and $h_2$ defined by $h_1=(h(1),\ldots,h(j))$ and $h_2=(h(j+1)-j,\ldots,h(n)-j)$ as shown in Figure~\ref{picture:decomposition_Hessenberg_function}.

\begin{figure}[h]
\begin{center}
\begin{picture}(175,75)
\put(0,63){\colorbox{gray}}
\put(0,67){\colorbox{gray}}
\put(0,72){\colorbox{gray}}
\put(4,63){\colorbox{gray}}
\put(4,67){\colorbox{gray}}
\put(4,72){\colorbox{gray}}
\put(9,63){\colorbox{gray}}
\put(9,67){\colorbox{gray}}
\put(9,72){\colorbox{gray}}

\put(15,63){\colorbox{gray}}
\put(15,67){\colorbox{gray}}
\put(15,72){\colorbox{gray}}
\put(19,63){\colorbox{gray}}
\put(19,67){\colorbox{gray}}
\put(19,72){\colorbox{gray}}
\put(24,63){\colorbox{gray}}
\put(24,67){\colorbox{gray}}
\put(24,72){\colorbox{gray}}

\put(30,63){\colorbox{gray}}
\put(30,67){\colorbox{gray}}
\put(30,72){\colorbox{gray}}
\put(34,63){\colorbox{gray}}
\put(34,67){\colorbox{gray}}
\put(34,72){\colorbox{gray}}
\put(39,63){\colorbox{gray}}
\put(39,67){\colorbox{gray}}
\put(39,72){\colorbox{gray}}

\put(45,63){\colorbox{gray}}
\put(45,67){\colorbox{gray}}
\put(45,72){\colorbox{gray}}
\put(49,63){\colorbox{gray}}
\put(49,67){\colorbox{gray}}
\put(49,72){\colorbox{gray}}
\put(54,63){\colorbox{gray}}
\put(54,67){\colorbox{gray}}
\put(54,72){\colorbox{gray}}

\put(60,63){\colorbox{gray}}
\put(60,67){\colorbox{gray}}
\put(60,72){\colorbox{gray}}
\put(64,63){\colorbox{gray}}
\put(64,67){\colorbox{gray}}
\put(64,72){\colorbox{gray}}
\put(69,63){\colorbox{gray}}
\put(69,67){\colorbox{gray}}
\put(69,72){\colorbox{gray}}

\put(0,48){\colorbox{gray}}
\put(0,52){\colorbox{gray}}
\put(0,57){\colorbox{gray}}
\put(4,48){\colorbox{gray}}
\put(4,52){\colorbox{gray}}
\put(4,57){\colorbox{gray}}
\put(9,48){\colorbox{gray}}
\put(9,52){\colorbox{gray}}
\put(9,57){\colorbox{gray}}

\put(15,48){\colorbox{gray}}
\put(15,52){\colorbox{gray}}
\put(15,57){\colorbox{gray}}
\put(19,48){\colorbox{gray}}
\put(19,52){\colorbox{gray}}
\put(19,57){\colorbox{gray}}
\put(24,48){\colorbox{gray}}
\put(24,52){\colorbox{gray}}
\put(24,57){\colorbox{gray}}

\put(30,48){\colorbox{gray}}
\put(30,52){\colorbox{gray}}
\put(30,57){\colorbox{gray}}
\put(34,48){\colorbox{gray}}
\put(34,52){\colorbox{gray}}
\put(34,57){\colorbox{gray}}
\put(39,48){\colorbox{gray}}
\put(39,52){\colorbox{gray}}
\put(39,57){\colorbox{gray}}

\put(45,48){\colorbox{gray}}
\put(45,52){\colorbox{gray}}
\put(45,57){\colorbox{gray}}
\put(49,48){\colorbox{gray}}
\put(49,52){\colorbox{gray}}
\put(49,57){\colorbox{gray}}
\put(54,48){\colorbox{gray}}
\put(54,52){\colorbox{gray}}
\put(54,57){\colorbox{gray}}

\put(60,48){\colorbox{gray}}
\put(60,52){\colorbox{gray}}
\put(60,57){\colorbox{gray}}
\put(64,48){\colorbox{gray}}
\put(64,52){\colorbox{gray}}
\put(64,57){\colorbox{gray}}
\put(69,48){\colorbox{gray}}
\put(69,52){\colorbox{gray}}
\put(69,57){\colorbox{gray}}

%\put(0,33){\colorbox{gray}}
%\put(0,37){\colorbox{gray}}
%\put(0,42){\colorbox{gray}}
%\put(4,33){\colorbox{gray}}
%\put(4,37){\colorbox{gray}}
%\put(4,42){\colorbox{gray}}
%\put(9,33){\colorbox{gray}}
%\put(9,37){\colorbox{gray}}
%\put(9,42){\colorbox{gray}}

\put(15,33){\colorbox{gray}}
\put(15,37){\colorbox{gray}}
\put(15,42){\colorbox{gray}}
\put(19,33){\colorbox{gray}}
\put(19,37){\colorbox{gray}}
\put(19,42){\colorbox{gray}}
\put(24,33){\colorbox{gray}}
\put(24,37){\colorbox{gray}}
\put(24,42){\colorbox{gray}}

\put(30,33){\colorbox{gray}}
\put(30,37){\colorbox{gray}}
\put(30,42){\colorbox{gray}}
\put(34,33){\colorbox{gray}}
\put(34,37){\colorbox{gray}}
\put(34,42){\colorbox{gray}}
\put(39,33){\colorbox{gray}}
\put(39,37){\colorbox{gray}}
\put(39,42){\colorbox{gray}}

\put(45,33){\colorbox{gray}}
\put(45,37){\colorbox{gray}}
\put(45,42){\colorbox{gray}}
\put(49,33){\colorbox{gray}}
\put(49,37){\colorbox{gray}}
\put(49,42){\colorbox{gray}}
\put(54,33){\colorbox{gray}}
\put(54,37){\colorbox{gray}}
\put(54,42){\colorbox{gray}}

\put(60,33){\colorbox{gray}}
\put(60,37){\colorbox{gray}}
\put(60,42){\colorbox{gray}}
\put(64,33){\colorbox{gray}}
\put(64,37){\colorbox{gray}}
\put(64,42){\colorbox{gray}}
\put(69,33){\colorbox{gray}}
\put(69,37){\colorbox{gray}}
\put(69,42){\colorbox{gray}}

%\put(0,18){\colorbox{gray}}
%\put(0,22){\colorbox{gray}}
%\put(0,27){\colorbox{gray}}
%\put(4,18){\colorbox{gray}}
%\put(4,22){\colorbox{gray}}
%\put(4,27){\colorbox{gray}}
%\put(9,18){\colorbox{gray}}
%\put(9,22){\colorbox{gray}}
%\put(9,27){\colorbox{gray}}

%\put(15,18){\colorbox{gray}}
%\put(15,22){\colorbox{gray}}
%\put(15,27){\colorbox{gray}}
%\put(19,18){\colorbox{gray}}
%\put(19,22){\colorbox{gray}}
%\put(19,27){\colorbox{gray}}
%\put(24,18){\colorbox{gray}}
%\put(24,22){\colorbox{gray}}
%\put(24,27){\colorbox{gray}}

%\put(30,18){\colorbox{gray}}
%\put(30,22){\colorbox{gray}}
%\put(30,27){\colorbox{gray}}
%\put(34,18){\colorbox{gray}}
%\put(34,22){\colorbox{gray}}
%\put(34,27){\colorbox{gray}}
%\put(39,18){\colorbox{gray}}
%\put(39,22){\colorbox{gray}}
%\put(39,27){\colorbox{gray}}

\put(45,18){\colorbox{gray}}
\put(45,22){\colorbox{gray}}
\put(45,27){\colorbox{gray}}
\put(49,18){\colorbox{gray}}
\put(49,22){\colorbox{gray}}
\put(49,27){\colorbox{gray}}
\put(54,18){\colorbox{gray}}
\put(54,22){\colorbox{gray}}
\put(54,27){\colorbox{gray}}

\put(60,18){\colorbox{gray}}
\put(60,22){\colorbox{gray}}
\put(60,27){\colorbox{gray}}
\put(64,18){\colorbox{gray}}
\put(64,22){\colorbox{gray}}
\put(64,27){\colorbox{gray}}
\put(69,18){\colorbox{gray}}
\put(69,22){\colorbox{gray}}
\put(69,27){\colorbox{gray}}

%\put(0,3){\colorbox{gray}}
%\put(0,7){\colorbox{gray}}
%\put(0,12){\colorbox{gray}}
%\put(4,3){\colorbox{gray}}
%\put(4,7){\colorbox{gray}}
%\put(4,12){\colorbox{gray}}
%\put(9,3){\colorbox{gray}}
%\put(9,7){\colorbox{gray}}
%\put(9,12){\colorbox{gray}}

%\put(15,3){\colorbox{gray}}
%\put(15,7){\colorbox{gray}}
%\put(15,12){\colorbox{gray}}
%\put(19,3){\colorbox{gray}}
%\put(19,7){\colorbox{gray}}
%\put(19,12){\colorbox{gray}}
%\put(24,3){\colorbox{gray}}
%\put(24,7){\colorbox{gray}}
%\put(24,12){\colorbox{gray}}

%\put(30,3){\colorbox{gray}}
%\put(30,7){\colorbox{gray}}
%\put(30,12){\colorbox{gray}}
%\put(34,3){\colorbox{gray}}
%\put(34,7){\colorbox{gray}}
%\put(34,12){\colorbox{gray}}
%\put(39,3){\colorbox{gray}}
%\put(39,7){\colorbox{gray}}
%\put(39,12){\colorbox{gray}}

\put(45,3){\colorbox{gray}}
\put(45,7){\colorbox{gray}}
\put(45,12){\colorbox{gray}}
\put(49,3){\colorbox{gray}}
\put(49,7){\colorbox{gray}}
\put(49,12){\colorbox{gray}}
\put(54,3){\colorbox{gray}}
\put(54,7){\colorbox{gray}}
\put(54,12){\colorbox{gray}}

\put(60,3){\colorbox{gray}}
\put(60,7){\colorbox{gray}}
\put(60,12){\colorbox{gray}}
\put(64,3){\colorbox{gray}}
\put(64,7){\colorbox{gray}}
\put(64,12){\colorbox{gray}}
\put(69,3){\colorbox{gray}}
\put(69,7){\colorbox{gray}}
\put(69,12){\colorbox{gray}}

\put(0,0){\framebox(15,15)}
\put(15,0){\framebox(15,15)}
\put(30,0){\framebox(15,15)}
\put(45,0){\framebox(15,15)}
\put(60,0){\framebox(15,15)}
\put(0,15){\framebox(15,15)}
\put(15,15){\framebox(15,15)}
\put(30,15){\framebox(15,15)}
\put(45,15){\framebox(15,15)}
\put(60,15){\framebox(15,15)}
\put(0,30){\framebox(15,15)}
\put(15,30){\framebox(15,15)}
\put(30,30){\framebox(15,15)}
\put(45,30){\framebox(15,15)}
\put(60,30){\framebox(15,15)}
\put(0,45){\framebox(15,15)}
\put(15,45){\framebox(15,15)}
\put(30,45){\framebox(15,15)}
\put(45,45){\framebox(15,15)}
\put(60,45){\framebox(15,15)}
\put(0,60){\framebox(15,15)}
\put(15,60){\framebox(15,15)}
\put(30,60){\framebox(15,15)}
\put(45,60){\framebox(15,15)}
\put(60,60){\framebox(15,15)}

\put(80,35){$\leadsto$}

\put(100,63){\colorbox{gray}}
\put(100,67){\colorbox{gray}}
\put(100,72){\colorbox{gray}}
\put(104,63){\colorbox{gray}}
\put(104,67){\colorbox{gray}}
\put(104,72){\colorbox{gray}}
\put(109,63){\colorbox{gray}}
\put(109,67){\colorbox{gray}}
\put(109,72){\colorbox{gray}}

\put(115,63){\colorbox{gray}}
\put(115,67){\colorbox{gray}}
\put(115,72){\colorbox{gray}}
\put(119,63){\colorbox{gray}}
\put(119,67){\colorbox{gray}}
\put(119,72){\colorbox{gray}}
\put(124,63){\colorbox{gray}}
\put(124,67){\colorbox{gray}}
\put(124,72){\colorbox{gray}}

\put(130,63){\colorbox{gray}}
\put(130,67){\colorbox{gray}}
\put(130,72){\colorbox{gray}}
\put(134,63){\colorbox{gray}}
\put(134,67){\colorbox{gray}}
\put(134,72){\colorbox{gray}}
\put(139,63){\colorbox{gray}}
\put(139,67){\colorbox{gray}}
\put(139,72){\colorbox{gray}}

\put(100,48){\colorbox{gray}}
\put(100,52){\colorbox{gray}}
\put(100,57){\colorbox{gray}}
\put(104,48){\colorbox{gray}}
\put(104,52){\colorbox{gray}}
\put(104,57){\colorbox{gray}}
\put(109,48){\colorbox{gray}}
\put(109,52){\colorbox{gray}}
\put(109,57){\colorbox{gray}}

\put(115,48){\colorbox{gray}}
\put(115,52){\colorbox{gray}}
\put(115,57){\colorbox{gray}}
\put(119,48){\colorbox{gray}}
\put(119,52){\colorbox{gray}}
\put(119,57){\colorbox{gray}}
\put(124,48){\colorbox{gray}}
\put(124,52){\colorbox{gray}}
\put(124,57){\colorbox{gray}}

\put(130,48){\colorbox{gray}}
\put(130,52){\colorbox{gray}}
\put(130,57){\colorbox{gray}}
\put(134,48){\colorbox{gray}}
\put(134,52){\colorbox{gray}}
\put(134,57){\colorbox{gray}}
\put(139,48){\colorbox{gray}}
\put(139,52){\colorbox{gray}}
\put(139,57){\colorbox{gray}}

%\put(0,33){\colorbox{gray}}
%\put(0,37){\colorbox{gray}}
%\put(0,42){\colorbox{gray}}
%\put(4,33){\colorbox{gray}}
%\put(4,37){\colorbox{gray}}
%\put(4,42){\colorbox{gray}}
%\put(9,33){\colorbox{gray}}
%\put(9,37){\colorbox{gray}}
%\put(9,42){\colorbox{gray}}

\put(115,33){\colorbox{gray}}
\put(115,37){\colorbox{gray}}
\put(115,42){\colorbox{gray}}
\put(119,33){\colorbox{gray}}
\put(119,37){\colorbox{gray}}
\put(119,42){\colorbox{gray}}
\put(124,33){\colorbox{gray}}
\put(124,37){\colorbox{gray}}
\put(124,42){\colorbox{gray}}

\put(130,33){\colorbox{gray}}
\put(130,37){\colorbox{gray}}
\put(130,42){\colorbox{gray}}
\put(134,33){\colorbox{gray}}
\put(134,37){\colorbox{gray}}
\put(134,42){\colorbox{gray}}
\put(139,33){\colorbox{gray}}
\put(139,37){\colorbox{gray}}
\put(139,42){\colorbox{gray}}

\put(145,18){\colorbox{gray}}
\put(145,22){\colorbox{gray}}
\put(145,27){\colorbox{gray}}
\put(149,18){\colorbox{gray}}
\put(149,22){\colorbox{gray}}
\put(149,27){\colorbox{gray}}
\put(154,18){\colorbox{gray}}
\put(154,22){\colorbox{gray}}
\put(154,27){\colorbox{gray}}

\put(160,18){\colorbox{gray}}
\put(160,22){\colorbox{gray}}
\put(160,27){\colorbox{gray}}
\put(164,18){\colorbox{gray}}
\put(164,22){\colorbox{gray}}
\put(164,27){\colorbox{gray}}
\put(169,18){\colorbox{gray}}
\put(169,22){\colorbox{gray}}
\put(169,27){\colorbox{gray}}

\put(145,3){\colorbox{gray}}
\put(145,7){\colorbox{gray}}
\put(145,12){\colorbox{gray}}
\put(149,3){\colorbox{gray}}
\put(149,7){\colorbox{gray}}
\put(149,12){\colorbox{gray}}
\put(154,3){\colorbox{gray}}
\put(154,7){\colorbox{gray}}
\put(154,12){\colorbox{gray}}

\put(160,3){\colorbox{gray}}
\put(160,7){\colorbox{gray}}
\put(160,12){\colorbox{gray}}
\put(164,3){\colorbox{gray}}
\put(164,7){\colorbox{gray}}
\put(164,12){\colorbox{gray}}
\put(169,3){\colorbox{gray}}
\put(169,7){\colorbox{gray}}
\put(169,12){\colorbox{gray}}

\put(145,0){\framebox(15,15)}
\put(160,0){\framebox(15,15)}
\put(145,15){\framebox(15,15)}
\put(160,15){\framebox(15,15)}
\put(100,30){\framebox(15,15)}
\put(115,30){\framebox(15,15)}
\put(130,30){\framebox(15,15)}
\put(100,45){\framebox(15,15)}
\put(115,45){\framebox(15,15)}
\put(130,45){\framebox(15,15)}
\put(100,60){\framebox(15,15)}
\put(115,60){\framebox(15,15)}
\put(130,60){\framebox(15,15)}
\end{picture}
\end{center}
\vspace{-10pt}
\caption{The decomposition of $h=(2,3,3,5,5)$ into $h_1=(2,3,3)$ and $h_2=(2,2)$.}
\label{picture:decomposition_Hessenberg_function}
\end{figure}
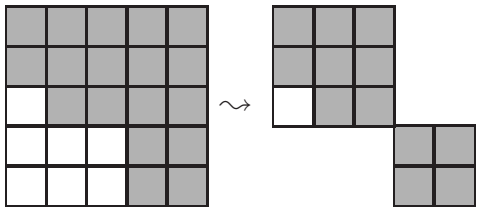

Then, every $V_\bullet \in \Hess(N,h)$ has $V_j = \Span_\C \{e_1,\ldots,e_j\}$ where $e_1,\ldots, e_n$ denote the standard basis of $\C^n$ and hence we have 
\begin{align} \label{eq:product_Hess} 
\Hess(N,h) \cong \Hess(N_1, h_1) \times \Hess(N_2, h_2)
\end{align}
where $N_1$ and $N_2$ are the regular nilpotent matrices in Jordan canonical form of size $j$ and $n-j$, respectively (\cite[Theorem~4.5]{Dre-Thesis}).

In a recent paper, Abe--Insko gave the condition that $\Hess(N,h)$ is normal as follows.

\begin{theorem} $($\cite[Theorem~1.3]{AbeInsko}$)$ \label{theorem:Hess(N,h)_normal} 
Let $h$ be an indecomposable Hessenberg function. 
Then the regular nilpotent Hessenberg variety $\Hess(N,h)$ is normal if and only if $h$ satisfies the condition that $h(i-1) > i$ or $h(i) > i + 1$ for all $1 < i < n-1$. 
\end{theorem}

There is a natural partial order on Hessenberg functions. For two Hessenberg functions $h:[n] \rightarrow [n]$ and $h':[n] \rightarrow [n]$, we say $h' \subset h$ if $h'(j) \leq h(j)$ for all $j \in [n]$. 
Note that if $h' \subset h$, then $\Hess(N,h')$ is a closed subvariety of $\Hess(N,h)$ by the definition \eqref{eq:Hessenberg_variety_flag}.
Since $h=(2,3,4,\ldots,n,n)$ is the minimal Hessenberg function among indecomposable Hessenberg functions with respect to the partial order $\subset$, the Peterson variety $\Pet_n$ is the minimal Hessenberg variety among indecomposable regular nilpotent Hessenberg varieties with respect to the inclusion.

Let $h$ be a Hessenberg function.
One can see that the Hessenberg space $H(h)$ in \eqref{eq:Hessenberg space} is stable under the adjoint action of $B$, so this induces a $B$-action on the quotient space $\gl_n(\C)/H(h)$.  
We denote by $\overline{x}$ the image of $x \in \gl_n(\C)$ under the surjection $\gl_n(\C) \rightarrow \gl_n(\C)/H(h)$. 
In \cite{ADGH} we consider the vector bundle $\GL_n(\C) \times_B (\gl_n(\C)/H(h))$ over the flag variety $\GL_n(\C)/B$ and its section $s_\XX$ for $\XX \in \gl_n(\C)$ defined by
\begin{align*}
s_\XX: \GL_n(\C)/B \rightarrow \GL_n(\C) \times_B (\gl_n(\C)/H(h)); \ gB \mapsto [g,\overline{g^{-1} \XX g}],
\end{align*}
where we denote by $[g,\overline{x}]$ the image of $(g,\overline{x}) \in \GL_n(\C) \times (\gl_n(\C)/H(h))$ under the surjection $\GL_n(\C) \times (\gl_n(\C)/H(h)) \rightarrow \GL_n(\C) \times_B (\gl_n(\C)/H(h))$ such that $[g,\overline{x}]=[gb,\overline{b^{-1}xb}]$ for all $b \in B$. 
By the definition \eqref{eq:Hessenberg_variety} one can see that the zero set of $s_\XX$ is the Hessenberg variety $\Hess(\XX,h)$. 

In general, let $\pi:E \rightarrow X$ be a vector bundle of rank $r$ over a scheme $X$. If $s$ is a section of $E$, then the \emph{zero scheme} $\ZZ(s)$ of $s$ is defined as follows (cf.  \cite{Ful98}).  
Let $(U_i, \varphi_i)_i$ be a local trivialization of $E$, i.e. an open covering $\{U_i\}_i$ of $X$ and isomorphisms $\varphi_i$ of $\pi^{-1}(U_i)$ with $U_i \times \C^r$ over $U_i$ such that the composites $\varphi_i \circ \varphi_j^{-1}$ are linear. 
Let $s_i: U_i \rightarrow \C^r$ determine the section $s$ on $U_i$, $s_i=(s_{i_1},\ldots,s_{i_r})$, $s_{i_m}$ in the coordinate ring of $U_i$; then $\ZZ(s)$ is defined in $U_i$ by the ideal generated by $s_{i_1},\ldots,s_{i_r}$. 

Motivated by the discussion above, we use the following definition introduced in \cite{ADGH}.

\begin{definition} $($\cite[Definition~3.1]{ADGH}$)$
Let $N \in \gl_n(\C)$ be the regular nilpotent element in \eqref{eq:regular nilpotent}.
For a Hessenberg function $h$, let $\ZZ(N,h)$ denote the zero scheme in $\GL_n(\C)/B$ of the section $s_N$.
\end{definition}

Locally around $gB \in \GL_n(\C)/B$, the section $s_N$ is represented by a collection of regular functions, and the scheme $\ZZ(N, h)$ is locally the zero scheme of those functions. (See \cite[Lemma~3.4]{ADGH} for the detail.)

\begin{theorem} $($\cite[Proposition~3.6]{ADGH}$)$ \label{theorem:reduced}
If $h$ is an indecomposable Hessenberg function, then the zero scheme $\ZZ(N,h)$ of the section $s_N$ is reduced. 
\end{theorem}

We analyze the intersection of $\Hess(N,h)$ and the opposite Schubert cell $\Omega_e^\circ$ associated with the identity element $e$.
Recall that $\Omega_e^\circ$ is the $B^{-}$-orbit of the identity flag $eB/B$ in $\GL_n(\C)/B$ where $B^{-}$ denotes the set of lower triangular matrices in $\GL_n(\C)$. 
In particular, each flag $V_\bullet \in \Omega_e^\circ$ has $V_j$ spanned by the first $j$ columns of a matrix with $1$'s in the diagonal positions, and $0$'s to the right of these $1$'s (cf. \cite[Section~10.2]{Ful97}).
Thus, one can see that the opposite Schubert cell $\Omega_e^\circ$ can be regarded as the following affine space:
\begin{align*} 
\Omega_e^\circ \cong \left\{ \left.
\left(
 \begin{array}{@{\,}ccccc@{\,}}
     1 &  &  &  &  \\
     x_{21} & 1 &  &  &  \\ 
     x_{31} & x_{32} & 1 &  &  \\ 
     \vdots& \vdots & \ddots & \ddots & \\
     x_{n1} & x_{n2} & \cdots & x_{n \, n-1} & 1 
 \end{array}
 \right) \right| x_{ij} \in \C \ (1 \leq j < i \leq n) \right\} \cong \C^{\frac{1}{2}n(n-1)}. 
\end{align*}
Note that $\Omega_e^\circ$ is an affine open set of the flag variety $\GL_n(\C)/B$ and its coordinate ring is isomorphic to the polynomial ring $\C[x_{ij} \mid 1 \leq j < i \leq n]$. 
Thus, we may identify $\Omega_e^\circ$ as $\Spec \C[x_{ij} \mid 1 \leq j < i \leq n]$.
From now on, we write  
\begin{align*} 
F_{i,j} \coloneqq F_{i,j}^N(g) \ \textrm{for} \ 1 \leq j < i \leq n \ \textrm{where} \ g=\left(
 \begin{array}{@{\,}ccccc@{\,}}
     1 &  &  &  &  \\
     x_{21} & 1 &  &  &  \\ 
     x_{31} & x_{32} & 1 &  &  \\ 
     \vdots& \vdots & \ddots & \ddots & \\
     x_{n1} & x_{n2} & \cdots & x_{n \, n-1} & 1 
 \end{array}
 \right)
\end{align*}
for simplicity.
Since the first column vector of $Ng$ is $(x_{21},x_{31},x_{41},\ldots,x_{n1},0)^t$ and the $j$-th column vector of $Ng$ is $(\underbrace{0,\ldots,0}_{j-2},\underbrace{1,x_{j+1 \, j},\ldots,x_{nj},0}_{n-j+2})^t$ for $j \geq 2$, we can explicitly write 
\begin{align}  
F_{i,1} &= \left|
 \begin{array}{@{\,}ccccc@{\,}}
     1 & 0 & \cdots & 0 & x_{21} \\
     x_{21} & 1 & \ddots & \vdots & x_{31} \\ 
     \vdots & x_{32} & \ddots & 0 & \vdots \\ 
     \vdots& \vdots & \ddots & 1 & \vdots \\
     x_{i1} & x_{i2} & \cdots & x_{i \, i-1} & x_{i+1 \, 1} 
 \end{array}
 \right| \ \ \ \textrm{for} \ j=1; \label{eq:Fi1explicit} \\
F_{i,j} &= \left|
 \begin{array}{@{\,}ccccc@{\,}}
     1 & 0 & \cdots & 0 & 1 \\
     x_{j\,j-1} & 1 & \ddots & \vdots & x_{j+1\,j} \\ 
     x_{j+1\,j-1} & x_{j+1\,j} & \ddots & 0 & \vdots \\ 
     \vdots& \vdots & \ddots & 1 & \vdots \\
     x_{i\,j-1} & x_{ij} & \cdots & x_{i \, i-1} & x_{i+1 \, j} 
 \end{array}
 \right| \ \ \ \textrm{for} \ j \geq 2  \label{eq:Fijexplicit} 
\end{align}
from Definition~\ref{definition:FijX}.
Here, we take the convention that $x_{n+1\,j}=0$. 
The determinant of $g \in \Omega_e^\circ$ is $1$, so we have
\begin{align} \label{eq:FijN}
F_{i,j} = F_{i,j}^{N}(g) = (g^{-1}Ng)_{ij} \ \ \ \textrm{for} \ g \in \Omega_e^\circ
\end{align}
by Lemma~\ref{lemma:defining equation}.
We set 
\begin{align} \label{eq:scheme_intersection}
\ZZ(N,h)_e \coloneqq &\ZZ(N,h) \cap \Spec \C[x_{ij} \mid 1 \leq j < i \leq n] \\ 
 = &\Spec \C[x_{ij} \mid 1 \leq j < i \leq n]/((g^{-1}Ng)_{ij} \mid j \in [n-1] \ \textrm{and} \ h(j) < i \leq n). \notag 
\end{align}
By the discussion above, we have the following.

\begin{proposition} \label{proposition:coordinate_ring_Fij}
Let $h$ be a Hessenberg function. Then
\begin{align*} 
\ZZ(N,h)_e = \Spec \C[x_{ij} \mid 1 \leq j < i \leq n]/(F_{i,j} \mid j \in [n-1] \ \textrm{and} \ h(j) < i \leq n). 
\end{align*}
In other words, the set of global sections $\Gamma(\ZZ(N,h)_e, \mathcal{O}_{\ZZ(N,h)_e})$ is given by 
\begin{align} \label{eq:global_section_Fij} 
\Gamma(\ZZ(N,h)_e, \mathcal{O}_{\ZZ(N,h)_e}) \cong \C[x_{ij} \mid 1 \leq j < i \leq n]/(F_{i,j} \mid j \in [n-1] \ \textrm{and} \ h(j) < i \leq n).
\end{align}
In particular, if $h$ is indecomposable, then the coordinate ring $\C[\Hess(N,h) \cap \Omega_e^\circ]$ of the open set $\Hess(N,h) \cap \Omega_e^\circ$ in $\Hess(N,h)$ is 
\begin{align} \label{eq:coordinate_ring_Fij} 
\C[\Hess(N,h) \cap \Omega_e^\circ] = \C[x_{ij} \mid 1 \leq j < i \leq n]/(F_{i,j} \mid j \in [n-2] \ \textrm{and} \ h(j) < i \leq n) 
\end{align}
by Theorem~\ref{theorem:reduced}.
\end{proposition}

\bigskip

%%%%%%%%%%%%%%%%%%%%%%%%%%%%%%%%%%
\section{The main theorem}
\label{section:The main theorem1}
%%%%%%%%%%%%%%%%%%%%%%%%%%%%%%%%%%

It can be extracted from Peterson's statements in \cite{Pet} that the coordinate ring $\C[\Pet_n \cap \Omega_e^\circ]$ is isomorphic to the quantum cohomology ring of the flag variety $\QH^*(\Fl(\C^n))$. 
This result is unpublished, so we also refer the reader to \cite{Kos, Rie} for the result.
In this section we first review an explicit presentation for the quantum cohomology ring $\QH^*(\Fl(\C^n))$ given by \cite{Font95, GK}. 
Then we introduce a further quantization of the ring presentation for $\QH^*(\Fl(\C^n))$ in natural way. 
Our main theorem relates $\Gamma(\ZZ(N,h)_e, \mathcal{O}_{\ZZ(N,h)_e})$ and our quantized ring presentation.

The quantum cohomology ring of the flag variety $\QH^*(\Fl(\C^n))$ is given by 
\begin{align*}
\QH^*(\Fl(\C^n)) = \C[q_1,\ldots,q_{n-1}] \otimes_\C H^*(\Fl(\C^n))
\end{align*}
as $\C$-vector spaces. 
Here, $q_1,\ldots,q_{n-1}$ are called the \emph{quantum parameters}. 
The product structure of $\QH^*(\Fl(\C^n))$ is a certain deformation of the cup product in the ordinary cohomology $H^*(\Fl(\C^n))$.
In order to describe an explicit presentation for $\QH^*(\Fl(\C^n))$, we need quantized elementary symmetric polynomials. 
Consider the matrix
\begin{align*}
\check M_n \coloneqq \left(
 \begin{array}{@{\,}ccccc@{\,}}
     x_1 & q_1 & 0 & \cdots & 0 \\
     -1 & x_2 & q_2 & \ddots & \vdots \\ 
      0 & \ddots & \ddots & \ddots & 0 \\ 
      \vdots & \ddots & -1 & x_{n-1} & q_{n-1} \\
      0 & \cdots & 0 & -1 & x_n 
 \end{array}
 \right)
\end{align*}
and define \emph{quantized elementary symmetric polynomials} $\check E_1^{(n)}, \ldots, \check E_n^{(n)}$ in the polynomial ring $\C[x_1,\ldots,x_n,q_1,\ldots,q_{n-1}]$ by the following equation
\begin{align*}
\det(\lambda I_n - \check M_n) = \lambda^n - \check E_1^{(n)} \lambda^{n-1} + \check E_2^{(n)} \lambda^{n-2} + \cdots + (-1)^n \check E_n^{(n)}
\end{align*}
where $I_n$ is the identity matrix of size $n$.
By using the cofactor expansion along the $s$-th column for $\det(\lambda I_s - \check M_s)$, we have
\begin{align*}
\det(\lambda I_s - \check M_s) = (\lambda-x_s) \det(\lambda I_{s-1} - \check M_{s-1}) +q_{s-1} \det(\lambda I_{s-2} - \check M_{s-2}). 
\end{align*}
This implies the following recursive formula
\begin{align*}
\check E_r^{(s)} = \check E_r^{(s-1)} + \check E_{r-1}^{(s-1)} x_s + \check E_{r-2}^{(s-2)} q_{s-1} \ \ \ \textrm{for} \ 1 \leq r \leq s \leq n
\end{align*}
where we take the convention that $\check E_0^{(s-1)}=1, \check E_{-1}^{(s-2)}=0$ whenever $r=1$ and $\check E_{s}^{(s-1)}=0$ whenever $r=s$. 
By setting $q_s = 0$ for all $s \in [n-1]$, $\check E_i^{(n)}$ is the (ordinary) $i$-th elementary symmetric polynomial in the variables $x_1,\ldots,x_n$.

\begin{theorem} [\cite{Font95, GK}] \label{theorem:quantum cohomology of flag}
There is an isomorphism of $\C$-algebras
\begin{align} \label{eq:quantum cohomology of flag}
\QH^*(\Fl(\C^n)) \cong \C[x_1,\ldots,x_n,q_1,\ldots,q_{n-1}]/(\check E_1^{(n)}, \ldots, \check E_n^{(n)}).
\end{align}
\end{theorem}

\begin{remark}
The isomorphism above is stated in \cite[Theorem~1]{GK} for integral coefficients. 
An explicit presentation of the quantum cohomology ring of partial flag varieties is given by  
\cite{AstSad, Font99, Kim}. 
\end{remark}

\begin{remark} \label{remark:usual cohomology}
By setting $q_s = 0$ for all $s \in [n-1]$, the isomorphism \eqref{eq:quantum cohomology of flag}  yields the well-known presentation for the (ordinary) cohomology ring of $\Fl(\C^n)$. 
Note that $x_i$ is geometrically the first Chern class of the dual of the $i$-th tautological line bundle over $\Fl(\C^n)$. See for example \cite[Section~10.2, Proposition~3]{Ful97}.
\end{remark}

\begin{theorem}[D.~Peterson and \cite{Kos, Rie}] \label{theorem:Peterson} 
There is an isomorphism of $\C$-algebras
\begin{align*} 
\C[\Pet_n \cap \Omega_e^\circ] \cong \QH^*(\Fl(\C^n)). 
\end{align*}
\end{theorem}

\begin{remark}
Although we restricted the above discussion to the case of the full flag variety,
the isomorphism above is in fact stated in \cite[Theorem~4.2]{Rie} for partial flag varieties. 
We also refer the reader to a recent paper \cite{C. H. Chow} for general Lie types. 
\end{remark}

By Theorems~\ref{theorem:quantum cohomology of flag} and \ref{theorem:Peterson} we obtain an isomorphism 
\begin{align} \label{eq:Petiso}
\C[\Pet_n \cap \Omega_e^\circ] \cong \C[x_1,\ldots,x_n,q_1,\ldots,q_{n-1}]/(\check E_1^{(n)}, \ldots, \check E_n^{(n)})
\end{align}
as $\C$-algebras.
One can see from \cite[Theorems~3.3 and 4.2]{Rie} that the isomorphism above sends $x_{ij}$ to $\check E_{i-j}^{(n-j)}$ under the presentation \eqref{eq:coordinate_ring_Fij} for $h=(2,3,4,\ldots,n,n)$. 
We generalize \eqref{eq:Petiso} to arbitrary Hessenberg function $h$.

\begin{definition} \label{definition:Eij}
Consider the matrix
\begin{align*}
M_n \coloneqq \left(
 \begin{array}{@{\,}ccccc@{\,}}
     x_1 & q_{12} & q_{13} & \cdots & q_{1n} \\
     -1 & x_2 & q_{23} & \cdots & q_{2n} \\ 
      0 & \ddots & \ddots & \ddots & \vdots \\ 
      \vdots & \ddots & -1 & x_{n-1} & q_{n-1 \, n} \\
      0 & \cdots & 0 & -1 & x_n 
 \end{array}
 \right)
\end{align*}
and define \emph{$q_{rs}$-quantized elementary symmetric polynomials} $E_1^{(n)}, \ldots, E_n^{(n)}$ in the polynomial ring $\C[x_1,\ldots,x_n,q_{rs} \mid 1 \leq r < s \leq n]$ by the following equation
\begin{align*}
\det(\lambda I_n - M_n) = \lambda^n - E_1^{(n)} \lambda^{n-1} + E_2^{(n)} \lambda^{n-2} + \cdots + (-1)^n E_n^{(n)}.
\end{align*}
In other words, $E_i^{(n)}$ is the coefficient of $\lambda^{n-i}$ for $\det(\lambda I_n - M_n)$ multiplied by $(-1)^i$. 
\end{definition}

\begin{remark}
If $q_{rs}=0$ for $s-r >1$ and $q_{s \, s+1}=q_s$, then the polynomial $E_i^{(n)}$ is equal to the (classical) quantized elementary symmetric polynomial $\check E_i^{(n)}$ in the polynomial ring $\C[x_1,\ldots,x_n,q_1,\ldots,q_{n-1}]$.
\end{remark}

\begin{lemma} \label{lemma:recursive_qij}
For $1 \leq r \leq s \leq n$, we have
\begin{align*}
E_r^{(s)} = E_r^{(s-1)} + E_{r-1}^{(s-1)} x_s + \sum_{k=1}^{r-1} E_{r-1-k}^{(s-1-k)} q_{s-k \, s} 
\end{align*}
with the convention that $E_0^{(*)}=1$ for arbitrary $*$, $\sum_{k=1}^{r-1} E_{r-1-k}^{(s-1-k)} q_{s-k \, s}=0$ whenever $r=1$, and $E_{s}^{(s-1)}=0$ whenever $r=s$. 
\end{lemma}

\begin{proof}
It follows from the cofactor expansion along the $s$-th column for $\det(\lambda I_s - M_s)$ that $\det(\lambda I_s - M_s)$ is equal to
\begin{align} \label{eq:proof_recursive}
&(-1)^{s+2} q_{1s} + (-1)^{s+3} q_{2s} \det(\lambda I_1 - M_1) + \cdots + (-1)^{s+s} q_{s-1 \, s} \det(\lambda I_{s-2} - M_{s-2}) \\
& \hspace{20pt} + (-1)^{s+s} (\lambda - x_s) \det(\lambda I_{s-1} - M_{s-1}). \notag
\end{align}
Since we have 
\begin{align*}
\det(\lambda I_k - M_k) = \lambda^k - E_1^{(k)} \lambda^{k-1} + E_2^{(k)} \lambda^{k-2} + \cdots + (-1)^k E_k^{(k)}
\end{align*}
for each $1 \leq k \leq s-1$ by the definition, the coefficient of $\lambda^{s-r}$ for \eqref{eq:proof_recursive} is given by
\begin{align*}
&(-1)^r q_{s-r+1 \, s} + (-1)^{r} q_{s-r+2 \, s} E_1^{(s-r+1)} + \cdots + (-1)^{r} q_{s-1 \, s} E_{r-2}^{(s-2)} \\
& \hspace{20pt} + (-x_s) (-1)^{r-1} E_{r-1}^{(s-1)} + (-1)^{r} E_r^{(s-1)}.  
\end{align*}
Namely, the coefficient of $\lambda^{s-r}$ for $\det(\lambda I_s - M_s)$ is 
\begin{align*}
(-1)^r \big(q_{s-r+1 \, s} + q_{s-r+2 \, s} E_1^{(s-r+1)} + \cdots + q_{s-1 \, s} E_{r-2}^{(s-2)} + x_s E_{r-1}^{(s-1)} + E_r^{(s-1)} \big),
\end{align*}
as desired.
\end{proof}

\begin{example}
Let $n=3$. Then the $E_r^{(s)}$ have the following form.
\begin{align*}
&E_0^{(1)} = 1, \ \ \ E_1^{(1)} = E_0^{(0)}x_1 =x_1, \\
&E_0^{(2)} = 1, \ \ \ E_1^{(2)} = E_1^{(1)} + E_0^{(1)}x_2 = x_1+x_2, \ \ \ E_2^{(2)} = E_1^{(1)}x_2 + E_0^{(0)}q_{12} = x_1x_2+q_{12}, \\
&E_0^{(3)} = 1, \ \ \ E_1^{(3)} = E_1^{(2)} + E_0^{(2)}x_3 = x_1+x_2+x_3, \\ 
&E_2^{(3)} = E_2^{(2)} + E_{1}^{(2)}x_3 + E_0^{(1)}q_{23} = x_1x_2+x_1x_3+x_2x_3+q_{12}+q_{23}, \\
&E_3^{(3)} = E_{2}^{(2)}x_3 + E_1^{(1)}q_{23} + E_0^{(0)}q_{13}= x_1x_2x_3+x_1q_{23}+x_3q_{12}+q_{13}. 
\end{align*}
\end{example} 

\begin{definition} \label{definition:hEi(n)}
Let $h: [n] \to [n]$ be a Hessenberg function. 
For each $1 \leq i \leq j \leq n$, we define $\hE_i^{(j)}$ as the polynomial $E_i^{(j)}$ by setting $q_{rs}=0$ for all $2 \leq s \leq n$ and $1 \leq r \leq n-h(n+1-s)$:  
\begin{align*}
\hE_i^{(j)} \coloneqq E_i^{(j)}|_{q_{rs}=0 \ (2 \leq s \leq n \ \textrm{and} \ 1 \leq r \leq n-h(n+1-s))}. 
\end{align*}
\end{definition}

The surviving variables in the polynomial $\hE_i^{(j)}$ are pictorially shown as follows.
Let $w_0$ be the longest element of the symmetric group $\SS_n$ on $n$ letters $[n]$. 
Consider the matrix
\begin{align} \label{eq:w0Mnw0}
w_0 M_n w_0 = \left(
 \begin{array}{@{\,}ccccc@{\,}}
     x_n & -1 & 0 & \cdots & 0 \\
     q_{n-1 \, n} & x_{n-1} & -1 & \ddots & \vdots \\ 
     \vdots  & \ddots & \ddots & \ddots & 0 \\ 
     q_{2n} & \cdots & q_{23} & x_2 & -1 \\
     q_{1n} & \cdots & q_{13} & q_{12} & x_1 
 \end{array}
 \right)
\end{align}
where $w_0$ is regarded as the permutation matrix. 
Then $w_0 M_n w_0 \in H(h)$ if and only if $q_{rs}=0$ for all $2 \leq s \leq n$ and $1 \leq r \leq n-h(n+1-s)$ where $H(h)$ is the Hessenberg space defined in \eqref{eq:Hessenberg space}.
In other words, the surviving variables in the polynomial $\hE_i^{(j)}$ are pictorially the variables $q_{rs}$ arranged in $w_0 M_n w_0$ such that $q_{rs}$ belongs to the configuration of the shaded boxes for the Hessenberg function $h$.

\begin{example} \label{example:qrs_(3,3,4,5,5)}
Let $n=5$ and $h=(3,3,4,5,5)$, which is depicted in Example~\ref{example:Hess_func_(3,3,4,5,5)}. 
Then $\hE_i^{(j)}$ is a polynomial in the variables $x_1,\ldots,x_5,q_{12},q_{23},q_{34},q_{35},q_{45}$ for $1 \leq i \leq j \leq 5$, as shown in Figure~\ref{picture:qrs}.
\begin{figure}[h]
\begin{center}
\begin{picture}(75,75)
\put(0,63){\colorbox{gray}}
\put(0,67){\colorbox{gray}}
\put(0,72){\colorbox{gray}}
\put(4,63){\colorbox{gray}}
\put(4,67){\colorbox{gray}}
\put(4,72){\colorbox{gray}}
\put(9,63){\colorbox{gray}}
\put(9,67){\colorbox{gray}}
\put(9,72){\colorbox{gray}}

\put(15,63){\colorbox{gray}}
\put(15,67){\colorbox{gray}}
\put(15,72){\colorbox{gray}}
\put(19,63){\colorbox{gray}}
\put(19,67){\colorbox{gray}}
\put(19,72){\colorbox{gray}}
\put(24,63){\colorbox{gray}}
\put(24,67){\colorbox{gray}}
\put(24,72){\colorbox{gray}}

\put(30,63){\colorbox{gray}}
\put(30,67){\colorbox{gray}}
\put(30,72){\colorbox{gray}}
\put(34,63){\colorbox{gray}}
\put(34,67){\colorbox{gray}}
\put(34,72){\colorbox{gray}}
\put(39,63){\colorbox{gray}}
\put(39,67){\colorbox{gray}}
\put(39,72){\colorbox{gray}}

\put(45,63){\colorbox{gray}}
\put(45,67){\colorbox{gray}}
\put(45,72){\colorbox{gray}}
\put(49,63){\colorbox{gray}}
\put(49,67){\colorbox{gray}}
\put(49,72){\colorbox{gray}}
\put(54,63){\colorbox{gray}}
\put(54,67){\colorbox{gray}}
\put(54,72){\colorbox{gray}}

\put(60,63){\colorbox{gray}}
\put(60,67){\colorbox{gray}}
\put(60,72){\colorbox{gray}}
\put(64,63){\colorbox{gray}}
\put(64,67){\colorbox{gray}}
\put(64,72){\colorbox{gray}}
\put(69,63){\colorbox{gray}}
\put(69,67){\colorbox{gray}}
\put(69,72){\colorbox{gray}}

\put(0,48){\colorbox{gray}}
\put(0,52){\colorbox{gray}}
\put(0,57){\colorbox{gray}}
\put(4,48){\colorbox{gray}}
\put(4,52){\colorbox{gray}}
\put(4,57){\colorbox{gray}}
\put(9,48){\colorbox{gray}}
\put(9,52){\colorbox{gray}}
\put(9,57){\colorbox{gray}}

\put(15,48){\colorbox{gray}}
\put(15,52){\colorbox{gray}}
\put(15,57){\colorbox{gray}}
\put(19,48){\colorbox{gray}}
\put(19,52){\colorbox{gray}}
\put(19,57){\colorbox{gray}}
\put(24,48){\colorbox{gray}}
\put(24,52){\colorbox{gray}}
\put(24,57){\colorbox{gray}}

\put(30,48){\colorbox{gray}}
\put(30,52){\colorbox{gray}}
\put(30,57){\colorbox{gray}}
\put(34,48){\colorbox{gray}}
\put(34,52){\colorbox{gray}}
\put(34,57){\colorbox{gray}}
\put(39,48){\colorbox{gray}}
\put(39,52){\colorbox{gray}}
\put(39,57){\colorbox{gray}}

\put(45,48){\colorbox{gray}}
\put(45,52){\colorbox{gray}}
\put(45,57){\colorbox{gray}}
\put(49,48){\colorbox{gray}}
\put(49,52){\colorbox{gray}}
\put(49,57){\colorbox{gray}}
\put(54,48){\colorbox{gray}}
\put(54,52){\colorbox{gray}}
\put(54,57){\colorbox{gray}}

\put(60,48){\colorbox{gray}}
\put(60,52){\colorbox{gray}}
\put(60,57){\colorbox{gray}}
\put(64,48){\colorbox{gray}}
\put(64,52){\colorbox{gray}}
\put(64,57){\colorbox{gray}}
\put(69,48){\colorbox{gray}}
\put(69,52){\colorbox{gray}}
\put(69,57){\colorbox{gray}}

\put(0,33){\colorbox{gray}}
\put(0,37){\colorbox{gray}}
\put(0,42){\colorbox{gray}}
\put(4,33){\colorbox{gray}}
\put(4,37){\colorbox{gray}}
\put(4,42){\colorbox{gray}}
\put(9,33){\colorbox{gray}}
\put(9,37){\colorbox{gray}}
\put(9,42){\colorbox{gray}}

\put(15,33){\colorbox{gray}}
\put(15,37){\colorbox{gray}}
\put(15,42){\colorbox{gray}}
\put(19,33){\colorbox{gray}}
\put(19,37){\colorbox{gray}}
\put(19,42){\colorbox{gray}}
\put(24,33){\colorbox{gray}}
\put(24,37){\colorbox{gray}}
\put(24,42){\colorbox{gray}}

\put(30,33){\colorbox{gray}}
\put(30,37){\colorbox{gray}}
\put(30,42){\colorbox{gray}}
\put(34,33){\colorbox{gray}}
\put(34,37){\colorbox{gray}}
\put(34,42){\colorbox{gray}}
\put(39,33){\colorbox{gray}}
\put(39,37){\colorbox{gray}}
\put(39,42){\colorbox{gray}}

\put(45,33){\colorbox{gray}}
\put(45,37){\colorbox{gray}}
\put(45,42){\colorbox{gray}}
\put(49,33){\colorbox{gray}}
\put(49,37){\colorbox{gray}}
\put(49,42){\colorbox{gray}}
\put(54,33){\colorbox{gray}}
\put(54,37){\colorbox{gray}}
\put(54,42){\colorbox{gray}}

\put(60,33){\colorbox{gray}}
\put(60,37){\colorbox{gray}}
\put(60,42){\colorbox{gray}}
\put(64,33){\colorbox{gray}}
\put(64,37){\colorbox{gray}}
\put(64,42){\colorbox{gray}}
\put(69,33){\colorbox{gray}}
\put(69,37){\colorbox{gray}}
\put(69,42){\colorbox{gray}}

%\put(0,18){\colorbox{gray}}
%\put(0,22){\colorbox{gray}}
%\put(0,27){\colorbox{gray}}
%\put(4,18){\colorbox{gray}}
%\put(4,22){\colorbox{gray}}
%\put(4,27){\colorbox{gray}}
%\put(9,18){\colorbox{gray}}
%\put(9,22){\colorbox{gray}}
%\put(9,27){\colorbox{gray}}

%\put(15,18){\colorbox{gray}}
%\put(15,22){\colorbox{gray}}
%\put(15,27){\colorbox{gray}}
%\put(19,18){\colorbox{gray}}
%\put(19,22){\colorbox{gray}}
%\put(19,27){\colorbox{gray}}
%\put(24,18){\colorbox{gray}}
%\put(24,22){\colorbox{gray}}
%\put(24,27){\colorbox{gray}}

\put(30,18){\colorbox{gray}}
\put(30,22){\colorbox{gray}}
\put(30,27){\colorbox{gray}}
\put(34,18){\colorbox{gray}}
\put(34,22){\colorbox{gray}}
\put(34,27){\colorbox{gray}}
\put(39,18){\colorbox{gray}}
\put(39,22){\colorbox{gray}}
\put(39,27){\colorbox{gray}}

\put(45,18){\colorbox{gray}}
\put(45,22){\colorbox{gray}}
\put(45,27){\colorbox{gray}}
\put(49,18){\colorbox{gray}}
\put(49,22){\colorbox{gray}}
\put(49,27){\colorbox{gray}}
\put(54,18){\colorbox{gray}}
\put(54,22){\colorbox{gray}}
\put(54,27){\colorbox{gray}}

\put(60,18){\colorbox{gray}}
\put(60,22){\colorbox{gray}}
\put(60,27){\colorbox{gray}}
\put(64,18){\colorbox{gray}}
\put(64,22){\colorbox{gray}}
\put(64,27){\colorbox{gray}}
\put(69,18){\colorbox{gray}}
\put(69,22){\colorbox{gray}}
\put(69,27){\colorbox{gray}}

%\put(0,3){\colorbox{gray}}
%\put(0,7){\colorbox{gray}}
%\put(0,12){\colorbox{gray}}
%\put(4,3){\colorbox{gray}}
%\put(4,7){\colorbox{gray}}
%\put(4,12){\colorbox{gray}}
%\put(9,3){\colorbox{gray}}
%\put(9,7){\colorbox{gray}}
%\put(9,12){\colorbox{gray}}

%\put(15,3){\colorbox{gray}}
%\put(15,7){\colorbox{gray}}
%\put(15,12){\colorbox{gray}}
%\put(19,3){\colorbox{gray}}
%\put(19,7){\colorbox{gray}}
%\put(19,12){\colorbox{gray}}
%\put(24,3){\colorbox{gray}}
%\put(24,7){\colorbox{gray}}
%\put(24,12){\colorbox{gray}}

%\put(30,3){\colorbox{gray}}
%\put(30,7){\colorbox{gray}}
%\put(30,12){\colorbox{gray}}
%\put(34,3){\colorbox{gray}}
%\put(34,7){\colorbox{gray}}
%\put(34,12){\colorbox{gray}}
%\put(39,3){\colorbox{gray}}
%\put(39,7){\colorbox{gray}}
%\put(39,12){\colorbox{gray}}

\put(45,3){\colorbox{gray}}
\put(45,7){\colorbox{gray}}
\put(45,12){\colorbox{gray}}
\put(49,3){\colorbox{gray}}
\put(49,7){\colorbox{gray}}
\put(49,12){\colorbox{gray}}
\put(54,3){\colorbox{gray}}
\put(54,7){\colorbox{gray}}
\put(54,12){\colorbox{gray}}

\put(60,3){\colorbox{gray}}
\put(60,7){\colorbox{gray}}
\put(60,12){\colorbox{gray}}
\put(64,3){\colorbox{gray}}
\put(64,7){\colorbox{gray}}
\put(64,12){\colorbox{gray}}
\put(69,3){\colorbox{gray}}
\put(69,7){\colorbox{gray}}
\put(69,12){\colorbox{gray}}

\put(0,0){\framebox(15,15){\tiny $q_{15}$}}
\put(15,0){\framebox(15,15){\tiny $q_{14}$}}
\put(30,0){\framebox(15,15){\tiny $q_{13}$}}
\put(45,0){\framebox(15,15){\tiny $q_{12}$}}
\put(60,0){\framebox(15,15){\tiny $x_1$}}
\put(0,15){\framebox(15,15){\tiny $q_{25}$}}
\put(15,15){\framebox(15,15){\tiny $q_{24}$}}
\put(30,15){\framebox(15,15){\tiny $q_{23}$}}
\put(45,15){\framebox(15,15){\tiny $x_2$}}
\put(60,15){\framebox(15,15)}
\put(0,30){\framebox(15,15){\tiny $q_{35}$}}
\put(15,30){\framebox(15,15){\tiny $q_{34}$}}
\put(30,30){\framebox(15,15){\tiny $x_3$}}
\put(45,30){\framebox(15,15)}
\put(60,30){\framebox(15,15)}
\put(0,45){\framebox(15,15){\tiny $q_{45}$}}
\put(15,45){\framebox(15,15){\tiny $x_4$}}
\put(30,45){\framebox(15,15)}
\put(45,45){\framebox(15,15)}
\put(60,45){\framebox(15,15)}
\put(0,60){\framebox(15,15){\tiny $x_5$}}
\put(15,60){\framebox(15,15)}
\put(30,60){\framebox(15,15)}
\put(45,60){\framebox(15,15)}
\put(60,60){\framebox(15,15)}
\end{picture}
\end{center}
\vspace{-10pt}
\caption{The polynomial $\hE_i^{(j)} \in \C[x_1,\ldots,x_5,q_{12},q_{23},q_{34},q_{35},q_{45}]$ for $h=(3,3,4,5,5)$.}
\label{picture:qrs}
\end{figure}
\end{example}

\begin{remark}
In our setting the flag variety $\Fl(\C^n)$ is identified with $\GL_n(\C)/B$, while $\Fl(\C^n)$ is regarded as $\GL_n(\C)/B^{-}$ in \cite{Rie}.
Recall that the conjugation by $w_0$ gives an isomorphism $\GL_n(\C)/B \cong \GL_n(\C)/B^{-}$ since $B^{-}=w_0 B w_0$. 
This relation might affect the reason why we take the conjugation by $w_0$ in \eqref{eq:w0Mnw0}.
\end{remark}

We now state the main theorem of this paper.

\begin{theorem} \label{theorem:main1}
Let $h: [n] \to [n]$ be a Hessenberg function and $\ZZ(N,h)_e$ the intersection defined in \eqref{eq:scheme_intersection}.
Then there is an isomorphism of $\C$-algebras
\begin{align} \label{eq:maintheorem1-1}
\Gamma(\ZZ(N,h)_e, \mathcal{O}_{\ZZ(N,h)_e}) \cong \frac{\C[x_1,\ldots,x_n, q_{rs} \mid 2 \leq s \leq n, n-h(n+1-s)<r<s]}{(\hE_1^{(n)}, \ldots, \, \hE_n^{(n)})}, 
\end{align}
which sends $x_{ij}$ to $\hE_{i-j}^{(n-j)}$ under the presentation \eqref{eq:global_section_Fij}.
In particular, if $h$ is indecomposable, then there is an isomorphism of $\C$-algebras
\begin{align} \label{eq:maintheorem1-2}
\C[\Hess(N,h) \cap \Omega_e^\circ] \cong \frac{\C[x_1,\ldots,x_n, q_{rs} \mid 2 \leq s \leq n, n-h(n+1-s)<r<s]}{(\hE_1^{(n)}, \ldots, \, \hE_n^{(n)})}, 
\end{align}
which sends $x_{ij}$ to $\hE_{i-j}^{(n-j)}$ under the presentation \eqref{eq:coordinate_ring_Fij}. 
\end{theorem}

We will prove Theorem~\ref{theorem:main1} in Section~\ref{section:Proof of main theorem 1}. 
For this purpose, one first see that the homomorphism in \eqref{eq:maintheorem1-1} is well-defined and surjective in the next section.  

\begin{remark}
We will introduce certain degrees for the variables $\{x_{ij} \mid 1 \leq j < i \leq n \}$ and $\{x_1,\ldots,x_n, q_{rs} \mid 1 \leq r < s \leq n \}$ so that the two sides of \eqref{eq:maintheorem1-1} are graded  $\C$-algebras (see Section~\ref{section:Hilbert series} for the detail). 
We will prove that \eqref{eq:maintheorem1-1} is in fact an isomorphism as graded $\C$-algebras in Section~\ref{section:Proof of main theorem 1}. 
\end{remark}

\begin{remark}
The isomorphism \eqref{eq:maintheorem1-2} in the case of $h=(2,3,4,\ldots,n,n)$ is exactly \eqref{eq:Petiso}.
\end{remark}

\bigskip

%%%%%%%%%%%%%%%%%%%%%%%%%%%%%%%%%%
\section{Properties of $E_r^{(s)}$}
\label{section:Properties of $E_r^{(s)}$}
%%%%%%%%%%%%%%%%%%%%%%%%%%%%%%%%%%

In this section we see relations between $x_s$'s, $q_{rs}$'s, and $E_r^{(s)}$'s. 
Then we construct an explicit map from $\Gamma(\ZZ(N,h)_e, \mathcal{O}_{\ZZ(N,h)_e})$ to our quotient ring $\C[x_1,\ldots,x_n, q_{rs} \mid 2 \leq s \leq n, n-h(n+1-s)<r<s]/(\hE_1^{(n)}, \ldots, \, \hE_n^{(n)})$.  

\begin{lemma} \label{lemma:relation_qij}
For $1 \leq r < s \leq n$, we have 
\begin{align*}
q_{rs} = \left|
 \begin{array}{@{\,}cccccc@{\,}}
     1 & 0 & \cdots & \cdots & 0 & E_{1}^{(s)}-E_{1}^{(s-1)} \\
     E_1^{(s-1)} & 1 & 0 &  & \vdots & E_{2}^{(s)}-E_{2}^{(s-1)} \\ 
     E_2^{(s-1)}  & E_1^{(s-2)} & 1 & \ddots & \vdots & E_{3}^{(s)}-E_{3}^{(s-1)} \\ 
     \vdots & \ddots & \ddots & \ddots & 0 & \vdots \\
     \vdots &  & \ddots & E_1^{(r+1)} & 1 & E_{s-r}^{(s)}-E_{s-r}^{(s-1)} \\
     E_{s-r}^{(s-1)} & \cdots & \cdots & E_2^{(r+1)} & E_1^{(r)} & E_{s-r+1}^{(s)}-E_{s-r+1}^{(s-1)} 
 \end{array}
 \right| 
\end{align*}
in the polynomial ring $\C[x_1,\ldots,x_n,q_{rs} \mid 1 \leq r < s \leq n]$.
\end{lemma}

\begin{proof}
For $1 \leq s \leq n$, it follows from Lemma~\ref{lemma:recursive_qij} that 
\begin{align*}
\begin{cases}
E_1^{(s)} &= E_1^{(s-1)} + x_s \\
E_2^{(s)} &= E_2^{(s-1)} + E_{1}^{(s-1)} x_s + q_{s-1 \, s} \\
E_3^{(s)} &= E_3^{(s-1)} + E_{2}^{(s-1)} x_s + E_{1}^{(s-2)} q_{s-1 \, s} + q_{s-2 \, s} \\
&\vdots \\
E_{s-r+1}^{(s)} &= E_{s-r+1}^{(s-1)} + E_{s-r}^{(s-1)} x_s + E_{s-r-1}^{(s-2)} q_{s-1 \, s} + E_{s-r-2}^{(s-3)} q_{s-2 \, s} + \cdots + E_1^{(r)} q_{r+1 \, s} + q_{rs}. 
\end{cases}
\end{align*}
In other words, we have
\begin{align*}
\left(
 \begin{array}{@{\,}cccccc@{\,}}
     1 &  &  &  & & \\
     E_1^{(s-1)} & 1 & &  &  &  \\ 
     E_2^{(s-1)}  & E_1^{(s-2)} & 1 & & & \\ 
     \vdots & \ddots & \ddots & \ddots & & \\
     \vdots &  & \ddots & E_1^{(r+1)} & 1 & \\
     E_{s-r}^{(s-1)} & \cdots & \cdots & E_2^{(r+1)} & E_1^{(r)} & 1 
 \end{array}
 \right) \left(
 \begin{array}{@{\,}c@{\,}}
     x_s \\
     q_{s-1 \, s} \\ 
     q_{s-2 \, s} \\ 
     \vdots \\ 
     q_{r+1 \, s} \\
     q_{rs}  
 \end{array}
 \right)
= \left(
 \begin{array}{@{\,}c@{\,}}
     E_1^{(s)}-E_1^{(s-1)} \\
     E_2^{(s)}-E_2^{(s-1)} \\ 
     E_3^{(s)}-E_3^{(s-1)} \\ 
     \vdots \\ 
     E_{s-r}^{(s)}-E_{s-r}^{(s-1)} \\
     E_{s-r+1}^{(s)}-E_{s-r+1}^{(s-1)} \\
 \end{array}
 \right).
\end{align*}
By Cramer's rule we obtain the desired result. 
\end{proof}

Set 
\begin{align} \label{eq:Qn}
Q_n \coloneqq \C[x_1,\ldots,x_n, q_{rs} \mid  1 \leq r < s \leq n]/(E_1^{(n)}, \ldots, E_n^{(n)})
\end{align}
and define a map $\varphi$ from $\C[x_{ij} \mid 1 \leq j < i \leq n]$ to the ring $Q_n$ by 
\begin{align} \label{eq:varphi}
\varphi: \C[x_{ij} \mid 1 \leq j < i \leq n] \to Q_n; \ x_{ij} \mapsto E_{i-j}^{(n-j)}.
\end{align}

\begin{proposition} \label{proposition:key}
The map $\varphi$ above sends 
\begin{align*}
\varphi(x_{n-s+1 \, n-s} - x_{n-s+2 \, n-s+1}) &= x_s \ \ \ \textrm{for} \ s \in [n] \\ 
\varphi(-F_{n+1-r,n+1-s}) &= q_{rs} \ \ \ \textrm{for} \ 1 \leq r < s \leq n.
\end{align*} 
Here, we take the convention that $x_{n-s+2 \, n-s+1}=0$ whenever $s=1$ and $x_{n-s+1 \, n-s}=0$ whenever $s=n$.
In particular, $\varphi$ is surjective.
\end{proposition}

\begin{proof}
First, one can see from Definition~\ref{definition:Eij} that 
\begin{align*}
E_1^{(s)} =x_1+\dots+x_s
\end{align*}
for any $s \in [n]$, which implies that $x_s=E_1^{(s)} -E_1^{(s-1)}$. 
By the definition of $\varphi$ we have $\varphi(x_{n-s+1 \, n-s} - x_{n-s+2 \, n-s+1}) = E_1^{(s)} -E_1^{(s-1)} = x_s$ for $s \in [n]$. (Note that if $s=n$, then we used the relation $E_1^{(n)}=0$ in the quotient ring $Q_n=\C[x_1,\ldots,x_n, q_{rs} \mid  1 \leq r < s \leq n]/(E_1^{(n)}, \ldots, E_n^{(n)})$.)

Next, it follows from \eqref{eq:Fi1explicit} that 
\begin{align} \label{eq:proofvarphi1}
\varphi(F_{n+1-r,1}) = \left|
 \begin{array}{@{\,}cccccc@{\,}}
     1 & 0 & \cdots & \cdots & 0 & E_1^{(n-1)} \\
     E_1^{(n-1)} & 1 & 0 &  & \vdots & E_2^{(n-1)} \\ 
     E_2^{(n-1)}  & E_1^{(n-2)} & 1 & \ddots & \vdots & \vdots \\ 
     \vdots & \ddots & \ddots & \ddots & 0 & \vdots \\
     \vdots &  & \ddots & E_1^{(r+1)} & 1 & E_{n-r}^{(n-1)} \\
     E_{n-r}^{(n-1)} & \cdots & \cdots & E_2^{(r+1)} & E_1^{(r)} & E_{n-r+1}^{(n-1)} 
 \end{array}
 \right|  
\end{align} 
for $1 \leq r \leq n-1$. 
On the other hand, by Lemma~\ref{lemma:relation_qij} we have 
\begin{align} \label{eq:proofvarphi2}
q_{rn} = \left|
 \begin{array}{@{\,}cccccc@{\,}}
     1 & 0 & \cdots & \cdots & 0 & -E_{1}^{(n-1)} \\
     E_1^{(n-1)} & 1 & 0 &  & \vdots & -E_{2}^{(n-1)} \\ 
     E_2^{(n-1)}  & E_1^{(n-2)} & 1 & \ddots & \vdots & -E_{3}^{(n-1)} \\ 
     \vdots & \ddots & \ddots & \ddots & 0 & \vdots \\
     \vdots &  & \ddots & E_1^{(r+1)} & 1 & -E_{n-r}^{(n-1)} \\
     E_{n-r}^{(n-1)} & \cdots & \cdots & E_2^{(r+1)} & E_1^{(r)} & -E_{n-r+1}^{(n-1)} 
 \end{array}
 \right| 
\end{align}
since $E_1^{(n)}=0, \ldots, E_{n-r+1}^{(n)}=0$ in the quotient ring $Q_n=\C[x_1,\ldots,x_n, q_{rs} \mid  1 \leq r < s \leq n]/(E_1^{(n)}, \ldots, E_n^{(n)})$. 
By \eqref{eq:proofvarphi1} and \eqref{eq:proofvarphi2} we have $\varphi(F_{n-r+1,1})=-q_{rn}$ for $1 \leq r \leq n-1$ as desired.

If $s <n$, then by \eqref{eq:Fijexplicit} we have 
\begin{align*} 
\varphi(F_{n+1-r,n+1-s}) &= \left|
 \begin{array}{@{\,}cccccc@{\,}}
     1 & 0 & \cdots & \cdots & 0 & 1 \\
     E_1^{(s)} & 1 & 0 &  & \vdots & E_1^{(s-1)} \\ 
     E_2^{(s)}  & E_1^{(s-1)} & 1 & \ddots & \vdots & E_2^{(s-1)} \\ 
     \vdots & E_2^{(s-1)} & \ddots & \ddots & 0 & \vdots \\
     \vdots & \vdots & \ddots & E_1^{(r+1)} & 1 & \vdots \\
     E_{s-r+1}^{(s)} & E_{s-r}^{(s-1)} & \cdots & E_2^{(r+1)} & E_1^{(r)} & E_{s-r+1}^{(s-1)} 
 \end{array}
 \right| \\
 &= \left|
 \begin{array}{@{\,}cccccc@{\,}}
     0 & 0 & \cdots & \cdots & 0 & 1 \\
     E_1^{(s)}-E_1^{(s-1)} & 1 & 0 &  & \vdots & E_1^{(s-1)} \\ 
     E_2^{(s)}-E_2^{(s-1)}  & E_1^{(s-1)} & 1 & \ddots & \vdots & E_2^{(s-1)} \\ 
     \vdots & E_2^{(s-1)} & \ddots & \ddots & 0 & \vdots \\
     \vdots & \vdots & \ddots & E_1^{(r+1)} & 1 & \vdots \\
     E_{s-r+1}^{(s)}-E_{s-r+1}^{(s-1)} & E_{s-r}^{(s-1)} & \cdots & E_2^{(r+1)} & E_1^{(r)} & E_{s-r+1}^{(s-1)} 
 \end{array}
 \right| \\ 
 & \qquad \textrm{(by subtracting the last column from the first column)} \\  
 &= -\left|
 \begin{array}{@{\,}cccccc@{\,}}
     1 & 0 & \cdots & \cdots & 0 & 0 \\
     E_1^{(s-1)} & 1 & 0 &  & \vdots & E_1^{(s)}-E_1^{(s-1)} \\ 
     E_2^{(s-1)}  & E_1^{(s-1)} & 1 & \ddots & \vdots & E_2^{(s)}-E_2^{(s-1)} \\ 
     \vdots & E_2^{(s-1)} & \ddots & \ddots & 0 & \vdots \\
     \vdots & \vdots & \ddots & E_1^{(r+1)} & 1 & \vdots \\
     E_{s-r+1}^{(s-1)} & E_{s-r}^{(s-1)} & \cdots & E_2^{(r+1)} & E_1^{(r)} & E_{s-r+1}^{(s)}-E_{s-r+1}^{(s-1)}
 \end{array}
 \right| \\ 
 & \quad \qquad \textrm{(by changing the first column and the last column)} \\  
  &= -\left|
 \begin{array}{@{\,}cccccc@{\,}}
     1 & 0 & \cdots & \cdots & 0 & E_{1}^{(s)}-E_{1}^{(s-1)} \\
     E_1^{(s-1)} & 1 & 0 &  & \vdots & E_{2}^{(s)}-E_{2}^{(s-1)} \\ 
     E_2^{(s-1)}  & E_1^{(s-2)} & 1 & \ddots & \vdots & E_{3}^{(s)}-E_{3}^{(s-1)} \\ 
     \vdots & \ddots & \ddots & \ddots & 0 & \vdots \\
     \vdots &  & \ddots & E_1^{(r+1)} & 1 & E_{s-r}^{(s)}-E_{s-r}^{(s-1)} \\
     E_{s-r}^{(s-1)} & \cdots & \cdots & E_2^{(r+1)} & E_1^{(r)} & E_{s-r+1}^{(s)}-E_{s-r+1}^{(s-1)} 
 \end{array}
 \right|  \\ 
 \end{align*} 
for $1 \leq r < s < n$, which is $-q_{rs}$ from Lemma~\ref{lemma:relation_qij}.
This completes the proof.
\end{proof}

It follows from Proposition~\ref{proposition:key} that the image of $\{-F_{i,j} \mid j \in [n-1] \ \textrm{and} \ h(j) < i \leq n \}$ under the map $\varphi$ in \eqref{eq:varphi} is $\{q_{rs} \mid 2 \leq s \leq n \ \textrm{and} \ 1 \leq r \leq n-h(n+1-s) \}$.
Thus, the surjective map $\varphi$ induces the surjective homomorphism 
\begin{align} 
&\varphi_h: \C[x_{ij} \mid 1 \leq j < i \leq n]/(F_{i,j} \mid j \in [n-1] \ \textrm{and} \ h(j) < i \leq n) \label{eq:varphih} \\
& \qquad \twoheadrightarrow Q_n/(q_{rs} \mid 2 \leq s \leq n \ \textrm{and} \ 1 \leq r \leq n-h(n+1-s)); \ x_{ij} \mapsto \, \hE_{i-j}^{(n-j)}. \notag 
\end{align}

\bigskip

%%%%%%%%%%%%%%%%%%%%%%%%%%%%%%%%%%
\section{Hilbert series}
\label{section:Hilbert series}
%%%%%%%%%%%%%%%%%%%%%%%%%%%%%%%%%%

In order to prove that $\varphi_h$ in \eqref{eq:varphih} is an isomorphism, we introduce certain degrees for the variables $\{x_{ij} \mid 1 \leq j < i \leq n \}$ and $\{x_1,\ldots,x_n, q_{rs} \mid 1 \leq r < s \leq n \}$ so that the two sides of \eqref{eq:varphih} are graded $\C$-algebras.
We then show that the two sides have identical Hilbert series.

Let $\C[x_{ij} \mid 1 \leq j < i \leq n]$ and $\C[x_1,\ldots,x_n,q_{rs} \mid 1 \leq r < s \leq n]$ be the polynomial rings equipped with a grading defined by
\begin{align}
\deg x_{ij} &= 2(i-j) \ \ \ \textrm{for} \ 1 \leq j < i \leq n; \label{eq:grading1} \\
\deg x_s &= 2 \ \ \ \textrm{for} \ s \in [n]; \label{eq:grading2} \\
\deg q_{rs} &= 2(s-r+1) \ \ \ \textrm{for} \ 1 \leq r < s \leq n. \label{eq:grading3} 
\end{align}

\begin{remark}
As mentioned in Remark~\ref{remark:usual cohomology}, $x_s$'s are
degree $2$ elements in the cohomology ring of the flag variety by forgetting quantum parameters.
Motivated by this fact, our definition for degrees are concentrated in even degrees.
\end{remark}

\begin{lemma} \label{lemma:grading1}
For $1 \leq r \leq s \leq n$, the polynomial $E_r^{(s)}$ is homogeneous of degree $2r$ in the polynomial ring $\C[x_1,\ldots,x_n,q_{rs} \mid 1 \leq r < s \leq n]$.
\end{lemma}

\begin{proof}
We prove this by induction on $s$. 
The base case is $s=1$, which is clear since $E_1^{(1)}=x_1$. 
Now suppose that $s > 1$ and assume by induction that the claim is true for arbitrary $s'$ with
$s' \leq s-1$.
From Lemma~\ref{lemma:recursive_qij} we have
\begin{align*}
E_r^{(s)} = E_r^{(s-1)} + E_{r-1}^{(s-1)} x_s + \sum_{k=1}^{r-1} E_{r-1-k}^{(s-1-k)} q_{s-k \, s}. 
\end{align*}
By the inductive hypothesis with \eqref{eq:grading2} and \eqref{eq:grading3}, one can see that $E_r^{(s)}$ is homogeneous of degree $2r$.
\end{proof}

In order to see that the polynomial $F_{i,j}$ is homogeneous, we introduce the following polynomials.

\begin{definition} \label{definition:tildeFij}
Let $1 \leq j \leq n-1$ and $j \leq \mm_j <n$.
For $i>\mm_j$, we define polynomials $\tilde F^{\langle\mm_j\rangle}_{i,j}$ by 
\begin{align*} 
\tilde F^{\langle\mm_1\rangle}_{i,1} \coloneqq \left|
 \begin{array}{@{\,}cccccc@{\,}}
     1 & 0 & \cdots & \cdots & 0 & x_{21} \\
     x_{21} & 1 & \ddots &  & \vdots & x_{31} \\ 
     x_{31} & x_{32} & \ddots & \ddots & \vdots & \vdots \\ 
     \vdots& \vdots & \ddots & \ddots & 0 & \vdots \\
     x_{\mm_1 \,1} & x_{\mm_1\,2} & \cdots & x_{\mm_1\,\mm_1-1} & 1 & x_{\mm_1+1\,1} \\
     x_{i1} & x_{i2} & \cdots & x_{i \, \mm_1-1} & x_{i \, \mm_1} & x_{i+1 \, 1}   
 \end{array}
 \right|; 
\end{align*}
\begin{align*} 
\tilde F^{\langle\mm_j\rangle}_{i,j} \coloneqq \left|
 \begin{array}{@{\,}cccccc@{\,}}
     1 & 0 & \cdots & \cdots & 0 & 1 \\
     x_{j \, j-1} & 1 & \ddots &  & \vdots & x_{j+1 \, j} \\ 
     x_{j+1 \, j-1} & x_{j+1 \, j} & \ddots & \ddots & \vdots & \vdots \\ 
     \vdots& \vdots & \ddots & \ddots & 0 & \vdots \\
     x_{\mm_j\,j-1} & x_{\mm_j\,j} & \cdots & x_{\mm_j\,\mm_j-1} & 1 & x_{\mm_j+1\,j} \\
     x_{i \, j-1} & x_{i \, j} & \cdots & x_{i \, \mm_j-1} & x_{i \, \mm_j} & x_{i+1 \, j} 
 \end{array}
 \right| \ \ \ \textrm{for} \ j \geq 2.
\end{align*}
Here, we take the convention that $x_{n+1\,j}=0$ for $j \in [n-1]$.
\end{definition}

By \eqref{eq:Fi1explicit} and \eqref{eq:Fijexplicit} one has
\begin{align} \label{eq:Fij_and_tildeFij}
F_{i,j} = \tilde F^{\langle i-1 \rangle}_{i,j} 
\end{align}
for $1 \leq j < i \leq n$. 

\begin{lemma} \label{lemma:grading2}
For $1 \leq j < i \leq n$, the polynomial $F_{i,j}$ is homogeneous of degree $2(i-j+1)$ in the polynomial ring $\C[x_{ij} \mid 1 \leq j < i \leq n]$.
\end{lemma}

\begin{proof}
It suffices to show that $\tilde F^{\langle\mm_j\rangle}_{i,j}$ is homogeneous of degree $2(i-j+1)$ for $j \leq \mm_j <i$ by \eqref{eq:Fij_and_tildeFij}. 
At first, we fix $j$.
We prove this statement by induction on $\mm_j$ with this fixed $j$.
The base case is $\mm_j =j$.
For arbitrary $i > j$, since we have 
\begin{align*}
\tilde F^{\langle 1 \rangle}_{i,1}&=\left|
 \begin{array}{@{\,}cc@{\,}}
     1 & x_{21} \\
     x_{i1} & x_{i+1 \, 1} 
 \end{array}
 \right| = x_{i+1 \, 1} - x_{i1}x_{21} \ \textrm{for} \ j=1; \\
\tilde F^{\langle j \rangle}_{i,j}&=\left|
 \begin{array}{@{\,}ccc@{\,}}
     1 & 0 & 1 \\
     x_{j \, j-1} & 1 & x_{j+1 \,j} \\
     x_{i \, j-1} & x_{ij} & x_{i+1 \, j} 
 \end{array}
 \right| = x_{i+1 \, j} + x_{j \, j-1}x_{ij} - x_{i \, j-1} - x_{j+1 \, j}x_{ij} \ \textrm{for} \ j>1,
\end{align*}
one can easily see from \eqref{eq:grading1} that $\tilde F^{\langle j \rangle}_{i,j}$ is homogeneous of degree $2(i-j+1)$.
This shows the base case.

We proceed to the inductive step. 
Suppose now that $\mm_j > j$ and that the claim holds for $\mm_j-1$ with any allowable choices of the first subscript $i'$ in $\tilde F^{\langle\mm_j -1\rangle}_{i',j}$. 
By the cofactor expansion along the second-to-last column, we have
\begin{align*} 
\tilde F^{\langle\mm_j\rangle}_{i,j} = \tilde F^{\langle\mm_j-1\rangle}_{i,j} - x_{i \, \mm_j} \tilde F^{\langle\mm_j-1\rangle}_{\mm_j,j}.
\end{align*}
By the inductive hypothesis and \eqref{eq:grading1}, the right hand side above is homogeneous of degree $2(i-j+1)$, as desired.
This completes the proof. 
\end{proof}

Recall from the end of Section~\ref{section:Properties of $E_r^{(s)}$} that we constructed the map $\varphi_h$ in \eqref{eq:varphih} from $\C[x_{ij} \mid 1 \leq j < i \leq n]/(F_{i,j} \mid j \in [n-1] \ \textrm{and} \ h(j) < i \leq n)$ to $Q_n/(q_{rs} \mid 2 \leq s \leq n \ \textrm{and} \ 1 \leq r \leq n-h(n+1-s))$.
One can see from Lemmas~\ref{lemma:grading1} and \ref{lemma:grading2} that these are graded $\C$-algebras.
(Note that $Q_n$ is also a graded $\C$-algebra.)
For the rest of this section, we prove that these graded $\C$-algebras have the same Hilbert series.

\begin{definition}
Let $R=\oplus_{i=0}^\infty R_i$ be a graded $\C$-algebra where each homogeneous component $R_i$ of degree $i$ is a finite-dimensional vector space over $\C$. Then its \emph{Hilbert series} is defined to be
\begin{align*}
\Hilb(R,t) \coloneqq \sum_{i=0}^{\infty} \dim_\C R_i \, t^i.
\end{align*}
A sequence of homogeneous polynomials $\theta_1,\ldots,\theta_p \in R$ of positive degrees is a \emph{regular sequence} in $R$ if $\theta_k$ is a non-zero divisor of $R/(\theta_1,\ldots,\theta_{k-1})$ for all $1 \leq k \leq p$.
\end{definition}

The following facts are well-known in commutative algebra. 
See \cite[Chapter~I, Section~5]{Stan96}. (See also \cite[Proposition~5.1]{FHM}.)

\begin{lemma} \label{lemma:regular_sequence_fact1}
Let $R=\oplus_{i=0}^\infty R_i$ be a graded $\C$-algebra with $\dim_\C R_i < \infty$ for each $i$.
A sequence of homogeneous polynomials $\theta_1,\ldots,\theta_p \in R$ of positive degrees is a regular sequence in $R$ if and only if the Hilbert series of $R/(\theta_1,\ldots,\theta_p)$ is given by
\begin{align*}
\Hilb(R/(\theta_1,\ldots,\theta_p),t) = \Hilb(R,t) \cdot \prod_{k=1}^p (1-t^{\deg \theta_k}). 
\end{align*}
\end{lemma}

\begin{lemma} \label{lemma:regular_sequence_fact2}
Let $R$ be a polynomial ring $\C[y_1,\ldots,y_n]$. 
A sequence of homogeneous polynomials $\theta_1,\ldots,\theta_n \in R$ of positive degrees is a regular sequence in $R$ if and only if the solution set of the equations $\theta_1=0,\ldots,\theta_n=0$ in $\C^n$ consists only of the origin $\{0\}$.
\end{lemma}

We remark that the number of the homogeneous polynomials $\theta_1,\ldots,\theta_n$ is equal to the number of the variables $y_1,\ldots,y_n$ in the polynomial ring $\C[y_1,\ldots,y_n]$ in Lemma~\ref{lemma:regular_sequence_fact2}. 
By using two lemmas above, we compute the Hilbert series of $\C[x_{ij} \mid 1 \leq j < i \leq n]/(F_{i,j} \mid j \in [n-1] \ \textrm{and} \ h(j) < i \leq n)$ and $Q_n/(q_{rs} \mid 2 \leq s \leq n \ \textrm{and} \ 1 \leq r \leq n-h(n+1-s))$.

\begin{lemma} \label{lemma:regular_sequence1}
The polynomials $F_{i,j} \ (1 \leq j < i \leq n)$ form a regular sequence in the polynomial ring $\C[x_{ij} \mid 1 \leq j < i \leq n]$.
\end{lemma}

\begin{proof}
By Lemma~\ref{lemma:regular_sequence_fact2} it is enough to show that the solution set of the equations $F_{i,j}=0 \ (1 \leq j < i \leq n)$ in $\C^{\frac{1}{2}n(n-1)}$ (with the variables $x_{ij} \ (1 \leq j < i \leq n)$) consists only of the origin $\{0\}$.
The intersection of the zero set of $F_{i,j}$ for all $1 \leq j < i \leq n$ is $\Hess(N,id) \cap \Omega_e^\circ$ by Lemma~\ref{lemma:defining equation}. 
However, since $\Hess(N,id)$ consists only of the point $\{eB \}$, this means that the equations $F_{i,j}=0 \ (1 \leq j<i \leq n)$ implies that $x_{ij}=0$ for all $1 \leq j < i \leq n$, as desired. 
\end{proof}

\begin{proposition} \label{proposition:Hilbert1}
The Hilbert series of $\C[x_{ij} \mid 1 \leq j < i \leq n]/(F_{i,j} \mid j \in [n-1] \ \textrm{and} \ h(j) < i \leq n)$ equipped with a grading in \eqref{eq:grading1} is equal to 
\begin{align*}
&\Hilb(\C[x_{ij} \mid 1 \leq j < i \leq n]/(F_{i,j} \mid j \in [n-1] \ \textrm{and} \ h(j) < i \leq n),t) \\
= &\prod_{1 \leq j \leq n-1 \atop j < i \leq h(j)} \frac{1}{1-t^{2(i-j+1)}} \cdot \prod_{k=1}^{n-1} (1+t^2+t^4+\cdots+t^{2k}).
\end{align*}
\end{proposition}

\begin{proof}
Since a subsequence of a regular sequence is also a regular sequence from the definition of a regular sequence, by Lemma~\ref{lemma:regular_sequence1} the polynomials $F_{i,j} \ (j \in [n-1],  h(j) < i \leq n)$ form a regular sequence in the polynomial ring $\C[x_{ij} \mid 1 \leq j < i \leq n]$.
Thus, it follows from Lemma~\ref{lemma:regular_sequence_fact1} that 
\begin{align} \label{eq:Hilbert1_proof1}
&\Hilb(\C[x_{ij} \mid 1 \leq j < i \leq n]/(F_{i,j} \mid j \in [n-1] \ \textrm{and} \ h(j) < i \leq n),t) \\
= &\Hilb(\C[x_{ij} \mid 1 \leq j < i \leq n],t) \cdot \prod_{1 \leq j \leq n-1 \atop h(j) < i \leq n} (1-t^{\deg F_{i,j}}) \notag \\
= &\prod_{1 \leq j < i \leq n} \frac{1}{1-t^{2(i-j)}} \cdot \prod_{1 \leq j \leq n-1 \atop h(j) < i \leq n} (1-t^{2(i-j+1)}) \ \ \ \ \ \textrm{(by Lemma~\ref{lemma:grading2})}. \notag
\end{align}
Here, we note that 
\begin{align*} 
\prod_{1 \leq j < i \leq n} \frac{1}{1-t^{2(i-j)}} \cdot (1-t^2)^{n-1} = \prod_{1 \leq j < i \leq n} \frac{1}{1-t^{2(i-j+1)}} \cdot (1-t^4)(1-t^6) \cdots (1-t^{2n}). 
\end{align*}
In fact, exponents appeared on the left hand side and exponents on the right hand side are described as numbers in shaded boxes of the left figure and the right figure in Figure~\ref{picture:proofHilbert}, respectively.
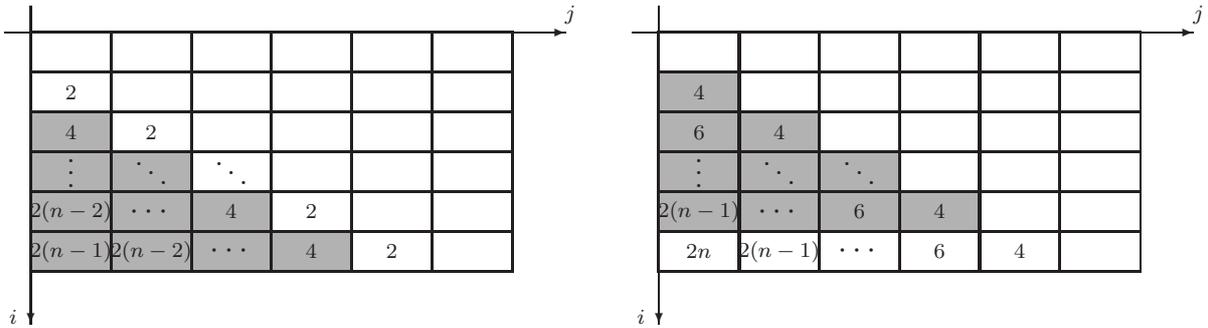
\begin{figure}[h]
\begin{center}
\begin{picture}(430,90)
%\put(0,63){\colorbox{gray}}
%\put(0,67){\colorbox{gray}}
%\put(0,72){\colorbox{gray}}
%\put(4,63){\colorbox{gray}}
%\put(4,67){\colorbox{gray}}
%\put(4,72){\colorbox{gray}}
%\put(9,63){\colorbox{gray}}
%\put(9,67){\colorbox{gray}}
%\put(9,72){\colorbox{gray}}
%\put(15,63){\colorbox{gray}}
%\put(15,67){\colorbox{gray}}
%\put(15,72){\colorbox{gray}}
%\put(19,63){\colorbox{gray}}
%\put(19,67){\colorbox{gray}}
%\put(19,72){\colorbox{gray}}
%\put(24,63){\colorbox{gray}}
%\put(24,67){\colorbox{gray}}
%\put(24,72){\colorbox{gray}}

\put(0,48){\colorbox{gray}}
\put(0,52){\colorbox{gray}}
\put(0,57){\colorbox{gray}}
\put(4,48){\colorbox{gray}}
\put(4,52){\colorbox{gray}}
\put(4,57){\colorbox{gray}}
\put(9,48){\colorbox{gray}}
\put(9,52){\colorbox{gray}}
\put(9,57){\colorbox{gray}}
\put(15,48){\colorbox{gray}}
\put(15,52){\colorbox{gray}}
\put(15,57){\colorbox{gray}}
\put(19,48){\colorbox{gray}}
\put(19,52){\colorbox{gray}}
\put(19,57){\colorbox{gray}}
\put(24,48){\colorbox{gray}}
\put(24,52){\colorbox{gray}}
\put(24,57){\colorbox{gray}}

%\put(30,48){\colorbox{gray}}
%\put(30,52){\colorbox{gray}}
%\put(30,57){\colorbox{gray}}
%\put(34,48){\colorbox{gray}}
%\put(34,52){\colorbox{gray}}
%\put(34,57){\colorbox{gray}}
%\put(39,48){\colorbox{gray}}
%\put(39,52){\colorbox{gray}}
%\put(39,57){\colorbox{gray}}
%\put(45,48){\colorbox{gray}}
%\put(45,52){\colorbox{gray}}
%\put(45,57){\colorbox{gray}}
%\put(49,48){\colorbox{gray}}
%\put(49,52){\colorbox{gray}}
%\put(49,57){\colorbox{gray}}
%\put(54,48){\colorbox{gray}}
%\put(54,52){\colorbox{gray}}
%\put(54,57){\colorbox{gray}}

\put(0,33){\colorbox{gray}}
\put(0,37){\colorbox{gray}}
\put(0,42){\colorbox{gray}}
\put(4,33){\colorbox{gray}}
\put(4,37){\colorbox{gray}}
\put(4,42){\colorbox{gray}}
\put(9,33){\colorbox{gray}}
\put(9,37){\colorbox{gray}}
\put(9,42){\colorbox{gray}}
\put(15,33){\colorbox{gray}}
\put(15,37){\colorbox{gray}}
\put(15,42){\colorbox{gray}}
\put(19,33){\colorbox{gray}}
\put(19,37){\colorbox{gray}}
\put(19,42){\colorbox{gray}}
\put(24,33){\colorbox{gray}}
\put(24,37){\colorbox{gray}}
\put(24,42){\colorbox{gray}}

\put(30,33){\colorbox{gray}}
\put(30,37){\colorbox{gray}}
\put(30,42){\colorbox{gray}}
\put(34,33){\colorbox{gray}}
\put(34,37){\colorbox{gray}}
\put(34,42){\colorbox{gray}}
\put(39,33){\colorbox{gray}}
\put(39,37){\colorbox{gray}}
\put(39,42){\colorbox{gray}}
\put(45,33){\colorbox{gray}}
\put(45,37){\colorbox{gray}}
\put(45,42){\colorbox{gray}}
\put(49,33){\colorbox{gray}}
\put(49,37){\colorbox{gray}}
\put(49,42){\colorbox{gray}}
\put(54,33){\colorbox{gray}}
\put(54,37){\colorbox{gray}}
\put(54,42){\colorbox{gray}}

%\put(60,33){\colorbox{gray}}
%\put(60,37){\colorbox{gray}}
%\put(60,42){\colorbox{gray}}
%\put(64,33){\colorbox{gray}}
%\put(64,37){\colorbox{gray}}
%\put(64,42){\colorbox{gray}}
%\put(69,33){\colorbox{gray}}
%\put(69,37){\colorbox{gray}}
%\put(69,42){\colorbox{gray}}
%\put(75,33){\colorbox{gray}}
%\put(75,37){\colorbox{gray}}
%\put(75,42){\colorbox{gray}}
%\put(79,33){\colorbox{gray}}
%\put(79,37){\colorbox{gray}}
%\put(79,42){\colorbox{gray}}
%\put(84,33){\colorbox{gray}}
%\put(84,37){\colorbox{gray}}
%\put(84,42){\colorbox{gray}}

\put(0,18){\colorbox{gray}}
\put(0,22){\colorbox{gray}}
\put(0,27){\colorbox{gray}}
\put(4,18){\colorbox{gray}}
\put(4,22){\colorbox{gray}}
\put(4,27){\colorbox{gray}}
\put(9,18){\colorbox{gray}}
\put(9,22){\colorbox{gray}}
\put(9,27){\colorbox{gray}}
\put(15,18){\colorbox{gray}}
\put(15,22){\colorbox{gray}}
\put(15,27){\colorbox{gray}}
\put(19,18){\colorbox{gray}}
\put(19,22){\colorbox{gray}}
\put(19,27){\colorbox{gray}}
\put(24,18){\colorbox{gray}}
\put(24,22){\colorbox{gray}}
\put(24,27){\colorbox{gray}}

\put(30,18){\colorbox{gray}}
\put(30,22){\colorbox{gray}}
\put(30,27){\colorbox{gray}}
\put(34,18){\colorbox{gray}}
\put(34,22){\colorbox{gray}}
\put(34,27){\colorbox{gray}}
\put(39,18){\colorbox{gray}}
\put(39,22){\colorbox{gray}}
\put(39,27){\colorbox{gray}}
\put(45,18){\colorbox{gray}}
\put(45,22){\colorbox{gray}}
\put(45,27){\colorbox{gray}}
\put(49,18){\colorbox{gray}}
\put(49,22){\colorbox{gray}}
\put(49,27){\colorbox{gray}}
\put(54,18){\colorbox{gray}}
\put(54,22){\colorbox{gray}}
\put(54,27){\colorbox{gray}}

\put(60,18){\colorbox{gray}}
\put(60,22){\colorbox{gray}}
\put(60,27){\colorbox{gray}}
\put(64,18){\colorbox{gray}}
\put(64,22){\colorbox{gray}}
\put(64,27){\colorbox{gray}}
\put(69,18){\colorbox{gray}}
\put(69,22){\colorbox{gray}}
\put(69,27){\colorbox{gray}}
\put(75,18){\colorbox{gray}}
\put(75,22){\colorbox{gray}}
\put(75,27){\colorbox{gray}}
\put(79,18){\colorbox{gray}}
\put(79,22){\colorbox{gray}}
\put(79,27){\colorbox{gray}}
\put(84,18){\colorbox{gray}}
\put(84,22){\colorbox{gray}}
\put(84,27){\colorbox{gray}}

%\put(90,18){\colorbox{gray}}
%\put(90,22){\colorbox{gray}}
%\put(90,27){\colorbox{gray}}
%\put(94,18){\colorbox{gray}}
%\put(94,22){\colorbox{gray}}
%\put(94,27){\colorbox{gray}}
%\put(99,18){\colorbox{gray}}
%\put(99,22){\colorbox{gray}}
%\put(99,27){\colorbox{gray}}
%\put(105,18){\colorbox{gray}}
%\put(105,22){\colorbox{gray}}
%\put(105,27){\colorbox{gray}}
%\put(109,18){\colorbox{gray}}
%\put(109,22){\colorbox{gray}}
%\put(109,27){\colorbox{gray}}
%\put(114,18){\colorbox{gray}}
%\put(114,22){\colorbox{gray}}
%\put(114,27){\colorbox{gray}}

\put(0,3){\colorbox{gray}}
\put(0,7){\colorbox{gray}}
\put(0,12){\colorbox{gray}}
\put(4,3){\colorbox{gray}}
\put(4,7){\colorbox{gray}}
\put(4,12){\colorbox{gray}}
\put(9,3){\colorbox{gray}}
\put(9,7){\colorbox{gray}}
\put(9,12){\colorbox{gray}}
\put(15,3){\colorbox{gray}}
\put(15,7){\colorbox{gray}}
\put(15,12){\colorbox{gray}}
\put(19,3){\colorbox{gray}}
\put(19,7){\colorbox{gray}}
\put(19,12){\colorbox{gray}}
\put(24,3){\colorbox{gray}}
\put(24,7){\colorbox{gray}}
\put(24,12){\colorbox{gray}}

\put(30,3){\colorbox{gray}}
\put(30,7){\colorbox{gray}}
\put(30,12){\colorbox{gray}}
\put(34,3){\colorbox{gray}}
\put(34,7){\colorbox{gray}}
\put(34,12){\colorbox{gray}}
\put(39,3){\colorbox{gray}}
\put(39,7){\colorbox{gray}}
\put(39,12){\colorbox{gray}}
\put(45,3){\colorbox{gray}}
\put(45,7){\colorbox{gray}}
\put(45,12){\colorbox{gray}}
\put(49,3){\colorbox{gray}}
\put(49,7){\colorbox{gray}}
\put(49,12){\colorbox{gray}}
\put(54,3){\colorbox{gray}}
\put(54,7){\colorbox{gray}}
\put(54,12){\colorbox{gray}}

\put(60,3){\colorbox{gray}}
\put(60,7){\colorbox{gray}}
\put(60,12){\colorbox{gray}}
\put(64,3){\colorbox{gray}}
\put(64,7){\colorbox{gray}}
\put(64,12){\colorbox{gray}}
\put(69,3){\colorbox{gray}}
\put(69,7){\colorbox{gray}}
\put(69,12){\colorbox{gray}}
\put(75,3){\colorbox{gray}}
\put(75,7){\colorbox{gray}}
\put(75,12){\colorbox{gray}}
\put(79,3){\colorbox{gray}}
\put(79,7){\colorbox{gray}}
\put(79,12){\colorbox{gray}}
\put(84,3){\colorbox{gray}}
\put(84,7){\colorbox{gray}}
\put(84,12){\colorbox{gray}}

\put(90,3){\colorbox{gray}}
\put(90,7){\colorbox{gray}}
\put(90,12){\colorbox{gray}}
\put(94,3){\colorbox{gray}}
\put(94,7){\colorbox{gray}}
\put(94,12){\colorbox{gray}}
\put(99,3){\colorbox{gray}}
\put(99,7){\colorbox{gray}}
\put(99,12){\colorbox{gray}}
\put(105,3){\colorbox{gray}}
\put(105,7){\colorbox{gray}}
\put(105,12){\colorbox{gray}}
\put(109,3){\colorbox{gray}}
\put(109,7){\colorbox{gray}}
\put(109,12){\colorbox{gray}}
\put(114,3){\colorbox{gray}}
\put(114,7){\colorbox{gray}}
\put(114,12){\colorbox{gray}}

%\put(120,3){\colorbox{gray}}
%\put(120,7){\colorbox{gray}}
%\put(120,12){\colorbox{gray}}
%\put(124,3){\colorbox{gray}}
%\put(124,7){\colorbox{gray}}
%\put(124,12){\colorbox{gray}}
%\put(129,3){\colorbox{gray}}
%\put(129,7){\colorbox{gray}}
%\put(129,12){\colorbox{gray}}
%\put(135,3){\colorbox{gray}}
%\put(135,7){\colorbox{gray}}
%\put(135,12){\colorbox{gray}}
%\put(139,3){\colorbox{gray}}
%\put(139,7){\colorbox{gray}}
%\put(139,12){\colorbox{gray}}
%\put(144,3){\colorbox{gray}}
%\put(144,7){\colorbox{gray}}
%\put(144,12){\colorbox{gray}}

\put(0,0){\framebox(30,15){\tiny $2(n-1)$}}
\put(30,0){\framebox(30,15){\tiny $2(n-2)$}}
\put(60,0){\framebox(30,15){$\cdots$}}
\put(90,0){\framebox(30,15){\tiny $4$}}
\put(120,0){\framebox(30,15){\tiny $2$}}
\put(150,0){\framebox(30,15)}
\put(0,15){\framebox(30,15){\tiny $2(n-2)$}}
\put(30,15){\framebox(30,15){$\cdots$}}
\put(60,15){\framebox(30,15){\tiny $4$}}
\put(90,15){\framebox(30,15){\tiny $2$}}
\put(120,15){\framebox(30,15)}
\put(150,15){\framebox(30,15)}
\put(0,30){\framebox(30,15){\vspace{6pt}$\vdots$}}
\put(30,30){\framebox(30,15){\vspace{6pt}$\ddots$}}
\put(60,30){\framebox(30,15){\vspace{6pt}$\ddots$}}
\put(90,30){\framebox(30,15)}
\put(120,30){\framebox(30,15)}
\put(150,30){\framebox(30,15)}
\put(0,45){\framebox(30,15){\tiny $4$}}
\put(30,45){\framebox(30,15){\tiny $2$}}
\put(60,45){\framebox(30,15)}
\put(90,45){\framebox(30,15)}
\put(120,45){\framebox(30,15)}
\put(150,45){\framebox(30,15)}
\put(0,60){\framebox(30,15){\tiny $2$}}
\put(30,60){\framebox(30,15)}
\put(60,60){\framebox(30,15)}
\put(90,60){\framebox(30,15)}
\put(120,60){\framebox(30,15)}
\put(150,60){\framebox(30,15)}
\put(0,75){\framebox(30,15)}
\put(30,75){\framebox(30,15)}
\put(60,75){\framebox(30,15)}
\put(90,75){\framebox(30,15)}
\put(120,75){\framebox(30,15)}
\put(150,75){\framebox(30,15)}

\put(0,100){\vector(0,-1){120}}
\put(-10,90){\vector(1,0){210}}

\put(-8,-20){\tiny $i$}
\put(200,95){\tiny $j$}

\hspace{-65pt}

\put(300,63){\colorbox{gray}}
\put(300,67){\colorbox{gray}}
\put(300,72){\colorbox{gray}}
\put(304,63){\colorbox{gray}}
\put(304,67){\colorbox{gray}}
\put(304,72){\colorbox{gray}}
\put(309,63){\colorbox{gray}}
\put(309,67){\colorbox{gray}}
\put(309,72){\colorbox{gray}}
\put(315,63){\colorbox{gray}}
\put(315,67){\colorbox{gray}}
\put(315,72){\colorbox{gray}}
\put(319,63){\colorbox{gray}}
\put(319,67){\colorbox{gray}}
\put(319,72){\colorbox{gray}}
\put(324,63){\colorbox{gray}}
\put(324,67){\colorbox{gray}}
\put(324,72){\colorbox{gray}}

\put(300,48){\colorbox{gray}}
\put(300,52){\colorbox{gray}}
\put(300,57){\colorbox{gray}}
\put(304,48){\colorbox{gray}}
\put(304,52){\colorbox{gray}}
\put(304,57){\colorbox{gray}}
\put(309,48){\colorbox{gray}}
\put(309,52){\colorbox{gray}}
\put(309,57){\colorbox{gray}}
\put(315,48){\colorbox{gray}}
\put(315,52){\colorbox{gray}}
\put(315,57){\colorbox{gray}}
\put(319,48){\colorbox{gray}}
\put(319,52){\colorbox{gray}}
\put(319,57){\colorbox{gray}}
\put(324,48){\colorbox{gray}}
\put(324,52){\colorbox{gray}}
\put(324,57){\colorbox{gray}}

\put(330,48){\colorbox{gray}}
\put(330,52){\colorbox{gray}}
\put(330,57){\colorbox{gray}}
\put(334,48){\colorbox{gray}}
\put(334,52){\colorbox{gray}}
\put(334,57){\colorbox{gray}}
\put(339,48){\colorbox{gray}}
\put(339,52){\colorbox{gray}}
\put(339,57){\colorbox{gray}}
\put(345,48){\colorbox{gray}}
\put(345,52){\colorbox{gray}}
\put(345,57){\colorbox{gray}}
\put(349,48){\colorbox{gray}}
\put(349,52){\colorbox{gray}}
\put(349,57){\colorbox{gray}}
\put(354,48){\colorbox{gray}}
\put(354,52){\colorbox{gray}}
\put(354,57){\colorbox{gray}}

\put(300,33){\colorbox{gray}}
\put(300,37){\colorbox{gray}}
\put(300,42){\colorbox{gray}}
\put(304,33){\colorbox{gray}}
\put(304,37){\colorbox{gray}}
\put(304,42){\colorbox{gray}}
\put(309,33){\colorbox{gray}}
\put(309,37){\colorbox{gray}}
\put(309,42){\colorbox{gray}}
\put(315,33){\colorbox{gray}}
\put(315,37){\colorbox{gray}}
\put(315,42){\colorbox{gray}}
\put(319,33){\colorbox{gray}}
\put(319,37){\colorbox{gray}}
\put(319,42){\colorbox{gray}}
\put(324,33){\colorbox{gray}}
\put(324,37){\colorbox{gray}}
\put(324,42){\colorbox{gray}}

\put(330,33){\colorbox{gray}}
\put(330,37){\colorbox{gray}}
\put(330,42){\colorbox{gray}}
\put(334,33){\colorbox{gray}}
\put(334,37){\colorbox{gray}}
\put(334,42){\colorbox{gray}}
\put(339,33){\colorbox{gray}}
\put(339,37){\colorbox{gray}}
\put(339,42){\colorbox{gray}}
\put(345,33){\colorbox{gray}}
\put(345,37){\colorbox{gray}}
\put(345,42){\colorbox{gray}}
\put(349,33){\colorbox{gray}}
\put(349,37){\colorbox{gray}}
\put(349,42){\colorbox{gray}}
\put(354,33){\colorbox{gray}}
\put(354,37){\colorbox{gray}}
\put(354,42){\colorbox{gray}}

\put(360,33){\colorbox{gray}}
\put(360,37){\colorbox{gray}}
\put(360,42){\colorbox{gray}}
\put(364,33){\colorbox{gray}}
\put(364,37){\colorbox{gray}}
\put(364,42){\colorbox{gray}}
\put(369,33){\colorbox{gray}}
\put(369,37){\colorbox{gray}}
\put(369,42){\colorbox{gray}}
\put(375,33){\colorbox{gray}}
\put(375,37){\colorbox{gray}}
\put(375,42){\colorbox{gray}}
\put(379,33){\colorbox{gray}}
\put(379,37){\colorbox{gray}}
\put(379,42){\colorbox{gray}}
\put(384,33){\colorbox{gray}}
\put(384,37){\colorbox{gray}}
\put(384,42){\colorbox{gray}}

\put(300,18){\colorbox{gray}}
\put(300,22){\colorbox{gray}}
\put(300,27){\colorbox{gray}}
\put(304,18){\colorbox{gray}}
\put(304,22){\colorbox{gray}}
\put(304,27){\colorbox{gray}}
\put(309,18){\colorbox{gray}}
\put(309,22){\colorbox{gray}}
\put(309,27){\colorbox{gray}}
\put(315,18){\colorbox{gray}}
\put(315,22){\colorbox{gray}}
\put(315,27){\colorbox{gray}}
\put(319,18){\colorbox{gray}}
\put(319,22){\colorbox{gray}}
\put(319,27){\colorbox{gray}}
\put(324,18){\colorbox{gray}}
\put(324,22){\colorbox{gray}}
\put(324,27){\colorbox{gray}}

\put(330,18){\colorbox{gray}}
\put(330,22){\colorbox{gray}}
\put(330,27){\colorbox{gray}}
\put(334,18){\colorbox{gray}}
\put(334,22){\colorbox{gray}}
\put(334,27){\colorbox{gray}}
\put(339,18){\colorbox{gray}}
\put(339,22){\colorbox{gray}}
\put(339,27){\colorbox{gray}}
\put(345,18){\colorbox{gray}}
\put(345,22){\colorbox{gray}}
\put(345,27){\colorbox{gray}}
\put(349,18){\colorbox{gray}}
\put(349,22){\colorbox{gray}}
\put(349,27){\colorbox{gray}}
\put(354,18){\colorbox{gray}}
\put(354,22){\colorbox{gray}}
\put(354,27){\colorbox{gray}}

\put(360,18){\colorbox{gray}}
\put(360,22){\colorbox{gray}}
\put(360,27){\colorbox{gray}}
\put(364,18){\colorbox{gray}}
\put(364,22){\colorbox{gray}}
\put(364,27){\colorbox{gray}}
\put(369,18){\colorbox{gray}}
\put(369,22){\colorbox{gray}}
\put(369,27){\colorbox{gray}}
\put(375,18){\colorbox{gray}}
\put(375,22){\colorbox{gray}}
\put(375,27){\colorbox{gray}}
\put(379,18){\colorbox{gray}}
\put(379,22){\colorbox{gray}}
\put(379,27){\colorbox{gray}}
\put(384,18){\colorbox{gray}}
\put(384,22){\colorbox{gray}}
\put(384,27){\colorbox{gray}}

\put(390,18){\colorbox{gray}}
\put(390,22){\colorbox{gray}}
\put(390,27){\colorbox{gray}}
\put(394,18){\colorbox{gray}}
\put(394,22){\colorbox{gray}}
\put(394,27){\colorbox{gray}}
\put(399,18){\colorbox{gray}}
\put(399,22){\colorbox{gray}}
\put(399,27){\colorbox{gray}}
\put(405,18){\colorbox{gray}}
\put(405,22){\colorbox{gray}}
\put(405,27){\colorbox{gray}}
\put(409,18){\colorbox{gray}}
\put(409,22){\colorbox{gray}}
\put(409,27){\colorbox{gray}}
\put(414,18){\colorbox{gray}}
\put(414,22){\colorbox{gray}}
\put(414,27){\colorbox{gray}}

\put(300,0){\framebox(30,15){\tiny $2n$}}
\put(330,0){\framebox(30,15){\tiny $2(n-1)$}}
\put(360,0){\framebox(30,15){$\cdots$}}
\put(390,0){\framebox(30,15){\tiny $6$}}
\put(420,0){\framebox(30,15){\tiny $4$}}
\put(450,0){\framebox(30,15)}
\put(300,15){\framebox(30,15){\tiny $2(n-1)$}}
\put(330,15){\framebox(30,15){$\cdots$}}
\put(360,15){\framebox(30,15){\tiny $6$}}
\put(390,15){\framebox(30,15){\tiny $4$}}
\put(420,15){\framebox(30,15)}
\put(450,15){\framebox(30,15)}
\put(300,30){\framebox(30,15){\vspace{6pt}$\vdots$}}
\put(330,30){\framebox(30,15){\vspace{6pt}$\ddots$}}
\put(360,30){\framebox(30,15){\vspace{6pt}$\ddots$}}
\put(390,30){\framebox(30,15)}
\put(420,30){\framebox(30,15)}
\put(450,30){\framebox(30,15)}
\put(300,45){\framebox(30,15){\tiny $6$}}
\put(330,45){\framebox(30,15){\tiny $4$}}
\put(360,45){\framebox(30,15)}
\put(390,45){\framebox(30,15)}
\put(420,45){\framebox(30,15)}
\put(450,45){\framebox(30,15)}
\put(300,60){\framebox(30,15){\tiny $4$}}
\put(330,60){\framebox(30,15)}
\put(360,60){\framebox(30,15)}
\put(390,60){\framebox(30,15)}
\put(420,60){\framebox(30,15)}
\put(450,60){\framebox(30,15)}
\put(300,75){\framebox(30,15)}
\put(330,75){\framebox(30,15)}
\put(360,75){\framebox(30,15)}
\put(390,75){\framebox(30,15)}
\put(420,75){\framebox(30,15)}
\put(450,75){\framebox(30,15)}

\put(300,100){\vector(0,-1){120}}
\put(290,90){\vector(1,0){210}}

\put(292,-20){\tiny $i$}
\put(500,95){\tiny $j$}
\end{picture}
\end{center}
\vspace{10pt}
\caption{The values $2(i-j)$ and $2(i-j+1)$ for $1 \leq j < i \leq n$.}
\label{picture:proofHilbert}
\end{figure}

This equality leads us to the equality
\begin{align*} 
\prod_{1 \leq j < i \leq n} \frac{1}{1-t^{2(i-j)}} = \prod_{1 \leq j < i \leq n} \frac{1}{1-t^{2(i-j+1)}} \cdot \prod_{k=1}^{n-1} (1+t^2+t^4+ \cdots +t^{2k}),
\end{align*}
so by \eqref{eq:Hilbert1_proof1} one has
\begin{align*} 
&\Hilb(\C[x_{ij} \mid 1 \leq j < i \leq n]/(F_{i,j} \mid j \in [n-1] \ \textrm{and} \ h(j) < i \leq n),t) \\
= &\prod_{1 \leq j < i \leq n} \frac{1}{1-t^{2(i-j+1)}} \cdot \prod_{k=1}^{n-1} (1+t^2+t^4+ \cdots +t^{2k}) \cdot \prod_{1 \leq j \leq n-1 \atop h(j) < i \leq n} (1-t^{2(i-j+1)}) \\
= &\prod_{1 \leq j \leq n-1 \atop j < i \leq h(j)} \frac{1}{1-t^{2(i-j+1)}} \cdot \prod_{k=1}^{n-1} (1+t^2+t^4+\cdots+t^{2k}),
\end{align*}
as desired.
\end{proof}

\begin{lemma} \label{lemma:regular_sequence2}
The polynomials $E_1^{(n)},\ldots,E_n^{(n)}, q_{rs} \ (1 \leq r<s \leq n)$ form a regular sequence in the polynomial ring $\C[x_1,\ldots,x_n, q_{rs} \mid  1 \leq r < s \leq n]$.
\end{lemma}

\begin{proof}
From Lemma~\ref{lemma:regular_sequence_fact2} it suffices to show that the solution set of the equations $E_1^{(n)}=0,\ldots,E_n^{(n)}=0, q_{rs}=0 \ (1 \leq r<s \leq n)$ in $\C^{\frac{1}{2}n(n+1)}$ (with the variables $x_1,\ldots,x_n, q_{rs} \ (1 \leq r < s \leq n)$) consists only of the origin $\{0\}$.
Since $q_{rs}=0$ for all $1 \leq r<s \leq n$, $E_1^{(n)}=0,\ldots,E_n^{(n)}=0$ implies that $e_1(x_1, \ldots, x_n)=0,\ldots,e_n(x_1, \ldots, x_n)=0$ where $e_i(x_1, \ldots, x_n)$ is the (ordinary) $i$-th elementary symmetric polynomial in the variables $x_1,\ldots,x_n$. 
Then one can easily see that $x_i=0$ for all $i \in [n]$. 
In fact, $x_1x_2 \cdots x_n = e_n(x_1, \ldots, x_n) = 0$ implies that some $x_i$ must be equal to $0$. 
Without loss of generality, we may assume that $x_n = 0$. 
This implies that $e_i(x_1, \ldots, x_{n-1}) = 0$ for all $i \in [n-1]$.
Proceeding in this manner, we conclude that $x_i=0$ for all $i \in [n]$.
Thus, the equations $E_1^{(n)}=0,\ldots,E_n^{(n)}=0, q_{rs}=0 \ (1 \leq r<s \leq n)$ implies that $x_i=0$ for all $i \in [n]$ and $q_{rs}=0$ for all $1 \leq r<s \leq n$, as desired. 
\end{proof}

\begin{proposition} \label{proposition:Hilbert2}
The Hilbert series of $Q_n/(q_{rs} \mid 2 \leq s \leq n \ \textrm{and} \ 1 \leq r \leq n-h(n+1-s))$ equipped with a grading in \eqref{eq:grading2} and \eqref{eq:grading3} is equal to 
\begin{align*}
&\Hilb(Q_n/(q_{rs} \mid 2 \leq s \leq n \ \textrm{and} \ 1 \leq r \leq n-h(n+1-s)),t) \\
= &\prod_{1 \leq j \leq n-1 \atop j < i \leq h(j)} \frac{1}{1-t^{2(i-j+1)}} \cdot \prod_{k=1}^{n-1} (1+t^2+t^4+\cdots+t^{2k}).
\end{align*}
\end{proposition}

\begin{proof}
Recall from the definition \eqref{eq:Qn} that 
\begin{align*} 
Q_n = \C[x_1,\ldots,x_n, q_{rs} \mid  1 \leq r < s \leq n]/(E_1^{(n)}, \ldots, E_n^{(n)}).
\end{align*}
Lemma~\ref{lemma:regular_sequence2} implies that the monomials $q_{rs} \ (2 \leq s \leq n, 1 \leq r \leq n-h(n+1-s))$ form a regular sequence in $Q_n$ by the definition of a regular sequence.
Hence by Lemma~\ref{lemma:regular_sequence_fact1}, we have 
\begin{align} \label{eq:Hilbert2_proof1}
&\Hilb(Q_n/(q_{rs} \mid 2 \leq s \leq n \ \textrm{and} \ 1 \leq r \leq n-h(n+1-s)),t) \\
= &\Hilb(Q_n,t) \cdot \prod_{2 \leq s \leq n \atop 1 \leq r \leq n-h(n+1-s)} (1-t^{\deg q_{rs}}) \notag \\
= &\Hilb(Q_n,t) \cdot \prod_{2 \leq s \leq n \atop 1 \leq r \leq n-h(n+1-s)} (1-t^{2(s-r+1)}). \notag
\end{align}
Since a subsequence of a regular sequence is again a regular sequence from the definition of a regular sequence, the polynomials $E_1^{(n)},\ldots,E_n^{(n)}$ form a regular sequence in the polynomial ring $\C[x_1,\ldots,x_n, q_{rs} \mid  1 \leq r < s \leq n]$ by Lemma~\ref{lemma:regular_sequence2}. 
By using Lemma~\ref{lemma:regular_sequence_fact1} again, one has 
\begin{align} \label{eq:Hilbert2_proof2}
\Hilb(Q_n,t) =& \Hilb(\C[x_1,\ldots,x_n, q_{rs} \mid  1 \leq r < s \leq n],t) \cdot \prod_{k=1}^n (1-t^{\deg E_k^{(n)}}) \\
= &\frac{1}{(1-t^2)^n} \cdot \prod_{1 \leq r < s \leq n} \frac{1}{1-t^{2(s-r+1)}} \cdot \prod_{k=1}^n (1-t^{2k}) \ \ \ \ \ \textrm{(by Lemma~\ref{lemma:grading1})} \notag \\
=& \prod_{1 \leq r < s \leq n} \frac{1}{1-t^{2(s-r+1)}} \cdot \prod_{k=1}^{n-1} (1+t^2+t^4+ \cdots +t^{2k}). \notag
\end{align}
By \eqref{eq:Hilbert2_proof1} and \eqref{eq:Hilbert2_proof2} we obtain
\begin{align*} 
&\Hilb(Q_n/(q_{rs} \mid 2 \leq s \leq n \ \textrm{and} \ 1 \leq r \leq n-h(n+1-s)),t) \\
= & \prod_{1 \leq r < s \leq n} \frac{1}{1-t^{2(s-r+1)}} \cdot \prod_{k=1}^{n-1} (1+t^2+t^4+ \cdots +t^{2k}) \cdot \prod_{2 \leq s \leq n \atop 1 \leq r \leq n-h(n+1-s)} (1-t^{2(s-r+1)}) \\
= & \prod_{1 \leq j < i \leq n} \frac{1}{1-t^{2(i-j+1)}} \cdot \prod_{k=1}^{n-1} (1+t^2+t^4+ \cdots +t^{2k}) \cdot \prod_{1 \leq j \leq n-1 \atop h(j)+1 \leq i \leq n} (1-t^{2(i-j+1)}) \\
& \hspace{100pt} (\textrm{by setting} \ i=n+1-r \ \textrm{and} \ j=n+1-s \ \textrm{in the third product}) \\
= &\prod_{1 \leq j \leq n-1 \atop j < i \leq h(j)} \frac{1}{1-t^{2(i-j+1)}} \cdot \prod_{k=1}^{n-1} (1+t^2+t^4+\cdots+t^{2k}),
\end{align*}
as desired.
\end{proof}

\bigskip

%%%%%%%%%%%%%%%%%%%%%%%%%%%%%%%%%%
\section{Proof of Theorem~\ref{theorem:main1}}
\label{section:Proof of main theorem 1}
%%%%%%%%%%%%%%%%%%%%%%%%%%%%%%%%%%

We now prove Theorem~\ref{theorem:main1}.

\begin{proof}[Proof of Theorem~\ref{theorem:main1}]
We first note that there exists a canonical isomorphism
\begin{align*}
&Q_n/(q_{rs} \mid 2 \leq s \leq n \ \textrm{and} \ 1 \leq r \leq n-h(n+1-s)) \\
\cong & \C[x_1,\ldots,x_n, q_{rs} \mid 2 \leq s \leq n, n-h(n+1-s)<r<s]/(\hE_1^{(n)}, \ldots, \, \hE_n^{(n)})
\end{align*}
by Definition~\ref{definition:hEi(n)} and \eqref{eq:Qn}.
Under the identification above and the presentation in \eqref{eq:global_section_Fij}, we can rewrite 
$\varphi_h$ in \eqref{eq:varphih} as 
\begin{align*} 
\varphi_h: \Gamma(\ZZ(N,h)_e, \mathcal{O}_{\ZZ(N,h)_e}) \twoheadrightarrow \frac{\C[x_1,\ldots,x_n, q_{rs} \mid 2 \leq s \leq n, n-h(n+1-s)<r<s]}{(\hE_1^{(n)}, \ldots, \, \hE_n^{(n)})},
\end{align*}
which is defined by 
\begin{align*}
\varphi_h(x_{ij}) = \, \hE_{i-j}^{(n-j)}.
\end{align*} 
Here, by slightly abuse of notation we used the same symbol $\varphi_h$ for the map above.
The map $\varphi_h$ is surjective and this preserves the gradings on both graded $\C$-algebras  from \eqref{eq:grading1} and Lemma~\ref{lemma:grading1}. 
It follows from Propositions~\ref{proposition:Hilbert1} and \ref{proposition:Hilbert2} that the two sides of $\varphi_h$ have identical Hilbert series.
Therefore, we conclude that $\varphi_h$ is an isomorphism.
\end{proof}

\begin{remark}
Our proof gives an alternative proof of the isomorphism \eqref{eq:Petiso}.
\end{remark}

\begin{corollary}
There is an isomorphism of $\C$-algebras
\begin{align*}
\Gamma(\ZZ(N,id)_e, \mathcal{O}_{\ZZ(N,id)_e}) \cong H^*(\Fl(\C^n)).
\end{align*}
\end{corollary}

\begin{proof}
Applying the isomorphism \eqref{eq:maintheorem1-1} to the case when $h=id$, we obtain
\begin{align*} 
\Gamma(\ZZ(N,id)_e, \mathcal{O}_{\ZZ(N,id)_e}) \cong \C[x_1,\ldots,x_n]/(e_1(x_1,\ldots,x_n), \ldots, e_n(x_1,\ldots,x_n)), 
\end{align*}
where $e_i(x_1, \ldots, x_n)$ denotes the $i$-th elementary symmetric polynomial in the variables $x_1,\ldots,x_n$. 
As is well-known, the right hand side above is a presentation for the cohomology ring $H^*(\Fl(\C^n))$ (e.g.  \cite[Section~10.2, Proposition~3]{Ful97}). 
\end{proof}

We constructed the isomorphism of graded $\C$-algebras
\begin{align} \label{eq:varphih_iso} 
\varphi_h: \Gamma(\ZZ(N,h)_e, \mathcal{O}_{\ZZ(N,h)_e}) \xrightarrow{\cong} \frac{\C[x_1,\ldots,x_n, q_{rs} \mid 2 \leq s \leq n, n-h(n+1-s)<r<s]}{(\hE_1^{(n)}, \ldots, \, \hE_n^{(n)})},
\end{align}
which sends $x_{ij}$ to $\hE_{i-j}^{(n-j)}$ for all $1 \leq j < i \leq n$ from the homomorphism $\varphi$ in \eqref{eq:varphi}. 
The following result follows from Proposition~\ref{proposition:key}. 

\begin{proposition} \label{proposition:inverse_map_varphih}
The inverse map of $\varphi_h$ in \eqref{eq:varphih_iso} is given by
\begin{align*}
\varphi_h^{-1}(x_s) &= x_{n-s+1 \, n-s} - x_{n-s+2 \, n-s+1} \ \ \ \textrm{for} \ s \in [n] \\ 
\varphi_h^{-1}(q_{rs}) &= -F_{n+1-r,n+1-s} \ \ \ \textrm{for} \ 2 \leq s \leq n \ \textrm{and} \ n-h(n+1-s)<r<s
\end{align*}
with the convention that $x_{n-s+2 \, n-s+1}=0$ whenever $s=1$ and $x_{n-s+1 \, n-s}=0$ whenever $s=n$.
\end{proposition}

We conclude the following result from the discussion above. 

\begin{corollary} \label{corollary:main_commutative_diagram}
Let $h:[n] \rightarrow [n]$ be a Hessenberg function.
Then the following commutative diagram holds
\begin{align*} 
  \begin{CD}
     \C[x_{ij} \mid 1 \leq j < i \leq n] @>{\varphi}>{\cong}> \displaystyle\frac{\C[x_1,\ldots,x_n, q_{rs} \mid 1 \leq r<s \leq n]}{(E_1^{(n)}, \ldots, \, E_n^{(n)})} \\
  @VV{F_{i,j}=0 \ (1 \leq j \leq n-1 \ {\rm and} \ h(j) < i \leq n)}V    @VV{q_{rs}=0 \ (2 \leq s \leq n \ {\rm and} \ 1 \leq r \leq n-h(n+1-s))}V \\
     \Gamma(\ZZ(N,h)_e, \mathcal{O}_{\ZZ(N,h)_e})   @>{\varphi_{h}}>{\cong}> \displaystyle\frac{\C[x_1,\ldots,x_n, q_{rs} \mid 2 \leq s \leq n, n-h(n+1-s)<r<s]}{(\hE_1^{(n)}, \ldots, \, \hE_n^{(n)})}
  \end{CD}
\end{align*}
where both vertical arrows denote the canonical surjective maps under the presentation \eqref{eq:global_section_Fij}. 
In particular, if $h$ is indecomposable, then we obtain 
the following commutative diagram 
\begin{align*} 
  \begin{CD}
     \C[\Fl(\C^n) \cap \Omega_e^\circ] @>{\varphi}>{\cong}> \displaystyle\frac{\C[x_1,\ldots,x_n, q_{rs} \mid 1 \leq r<s \leq n]}{(E_1^{(n)}, \ldots, \, E_n^{(n)})} \\
  @VV{F_{i,j}=0 \ (1 \leq j \leq n-2 \ {\rm and} \ h(j) < i \leq n)}V    @VV{q_{rs}=0 \ (3 \leq s \leq n \ {\rm and} \ 1 \leq r \leq n-h(n+1-s))}V \\
     \C[\Hess(N,h) \cap \Omega_e^\circ]   @>{\varphi_{h}}>{\cong}> \displaystyle\frac{\C[x_1,\ldots,x_n, q_{rs} \mid 2 \leq s \leq n, n-h(n+1-s)<r<s]}{(\hE_1^{(n)}, \ldots, \, \hE_n^{(n)})}
  \end{CD}
\end{align*}
where the left vertical arrow is induced from the inclusion $\Hess(N,h) \cap \Omega_e^\circ \hookrightarrow \Fl(\C^n) \cap \Omega_e^\circ$. 
Note that both vertical arrows are surjective. 
\end{corollary}

\bigskip

%%%%%%%%%%%%%%%%%%%%%%%%%%%%%%%%%%
\section{Jacobian matrix}
\label{section:Jacobian matrix}
%%%%%%%%%%%%%%%%%%%%%%%%%%%%%%%%%%

It is an interesting and challenging problem to find an explicit description of the singular locus of (regular nilpotent) Hessenberg varieties $\Hess(N,h)$. 
There are already partial results for the problem stated above in \cite{AbeInsko, EPS, IY}. 
For the rest of the paper, we will analyze the singular locus of $\Hess(N,h) \cap \Omega_e^\circ$ for some Hessenberg function $h$ as an application of our result. 
The isomorphism \eqref{eq:maintheorem1-2} in Theorem~\ref{theorem:main1} yields that 
if $h$ is indecomposable, then the singular locus of the open set $\Hess(N,h) \cap \Omega_e^\circ$ in $\Hess(N,h)$ is isomorphic to the singular locus of the zero set defined by $n$ polynomials $\hE_1^{(n)}, \ldots, \, \hE_n^{(n)}$ in $\C^{n+\sum_{j=1}^n (h(j)-j))}$ with the variables $x_1,\ldots,x_n, q_{rs} \ (2 \leq s \leq n, n-h(n+1-s)<r<s)$. 
In this section we give an explicit formula for partial derivatives $\partial E_i^{(n)}/\partial x_s \ (1 \leq s \leq n)$ and $\partial E_i^{(n)}/\partial q_{rs} \ (1 \leq r < s \leq n)$ for each $i \in [n]$. 

For positive integers $a$ and $b$ with $1 \leq a \leq b \leq n$, we set 
\begin{align*}
[a, b] \coloneqq \{a, a + 1, \ldots , b\}.
\end{align*}

As in Definition~\ref{definition:Eij}, we introduce the following polynomials.

\begin{definition} \label{definition:Eijab}
Let $M_{[a,b]}$ be the matrix of size $(b-a+1) \times (b-a+1)$ defined as
\begin{align*}
M_{[a,b]} \coloneqq \left(
 \begin{array}{@{\,}ccccc@{\,}}
     x_a & q_{a \, a+1} & q_{a \, a+2} & \cdots & q_{a \, b} \\
     -1 & x_{a+1} & q_{a+1 \, a+2} & \cdots & q_{a+1 \, b} \\ 
      0 & \ddots & \ddots & \ddots & \vdots \\ 
      \vdots & \ddots & -1 & x_{b-1} & q_{b-1 \, b} \\
      0 & \cdots & 0 & -1 & x_b 
 \end{array}
 \right).
\end{align*}
We define polynomials $E_1^{[a,b]}, \ldots, E_{b-a+1}^{[a,b]} \in \C[x_a,\ldots,x_b,q_{rs} \mid a \leq r < s \leq b]$ by 
\begin{align*}
\det(\lambda I_{b-a+1} - M_{[a,b]}) = \lambda^{b-a+1} - E_1^{[a,b]} \lambda^{b-a} + E_2^{[a,b]} \lambda^{b-a-1} + \cdots + (-1)^{b-a+1} E_{b-a+1}^{[a,b]}.
\end{align*}
\end{definition}

Note that if $a=1$, then we obtain
\begin{align} \label{eq:Ei1b}
E_i^{[1,b]} = E_i^{(b)} \ \ \ \textrm{for each} \ i \in [b]
\end{align}
by the definition. 
In what follows, we use the symbol
\begin{align*}
[a,a-1] \coloneqq \emptyset \ \ \ \textrm{for each} \ a \in [n+1]
\end{align*}
and we take the following convention 
\begin{align} \label{eq:convention_Eab}
E_i^{[a,b]} \coloneqq
\begin{cases} 
1 \ \ \ &\textrm{if} \ i=0, \\
0 \ \ \ &\textrm{if} \ i > b-a+1 
\end{cases}
\end{align}
for $b \geq a-1$.
By the same argument for the proof of Lemma~\ref{lemma:recursive_qij}, one can prove the lemma below.
We leave the detail to the reader.

\begin{lemma} \label{lemma:recursive_qijab}
Let $1 \leq a \leq b \leq n$. 
For $1 \leq i \leq b-a+1$, we have
\begin{align*}
E_i^{[a,b]} = E_i^{[a,b-1]} + E_{i-1}^{[a,b-1]} x_b + \sum_{k=1}^{i-1} E_{i-1-k}^{[a,b-1-k]} q_{b-k \, b} 
\end{align*}
with the convention that $\sum_{k=1}^{i-1} E_{i-1-k}^{[a,b-1-k]} q_{b-k \, b}=0$ whenever $i=1$. 
\end{lemma}

\begin{lemma} \label{lemma:partial_derivatives}
\begin{enumerate}
\item Let $s \in [n]$. 
For $i \in [n]$, we have
\begin{align*}
\frac{\partial}{\partial x_s} E_i^{(n)} = \sum_{k=0}^{i-1} E_{i-1-k}^{[1,s-1]} E_{k}^{[s+1,n]}.
\end{align*}
\item Let $1 \leq r < s \leq n$. Then
\begin{align*}
\frac{\partial}{\partial q_{rs}} E_{i+s-r}^{(n)} = \begin{cases}
0 & \textrm{if} \ 1-(s-r) \leq i \leq 0, \\
\sum_{k=0}^{i-1} E_{i-1-k}^{[1,r-1]} E_{k}^{[s+1,n]} & \textrm{if} \ 1 \leq i \leq n-(s-r).
\end{cases}
\end{align*}
\end{enumerate}
\end{lemma}

\begin{proof}
(1) By Definition~\ref{definition:Eij} one has
\begin{align} \label{eq:proof_partial_derivatives}
\det(\lambda I_n - M_n) = \lambda^n - E_1^{(n)} \lambda^{n-1} + E_2^{(n)} \lambda^{n-2} + \cdots + (-1)^n E_n^{(n)}.
\end{align}
We think of $\lambda$ as a variable in the equality above and we partial differentiate the both sides with respect to $x_s$. 
Then $\frac{\partial}{\partial x_s} E_i^{(n)}$ is equal to the coefficient of $\lambda^{n-i}$ for $\frac{\partial}{\partial x_s}\det(\lambda I_n - M_n)$ multiplied by $(-1)^{i}$. 
Since the variable $x_s$ appears in only the $(s,s)$-th entry of the matrix $(\lambda I_n - M_n)$, from the cofactor expansion along the $s$-th column for $\det(\lambda I_n - M_n)$ we can write 
\begin{align*}
\det(\lambda I_n - M_n) = (\lambda-x_s) \det(\lambda I_{s-1} - M_{[1,s-1]}) \det(\lambda I_{n-s} - M_{[s+1,n]}) + F
\end{align*}
for some polynomial $F \in \C[x_1,\ldots,\widehat{x_s},\ldots,x_n,q_{ij} \mid 1 \leq i < j \leq n]$. 
Here, the caret sign $\widehat{}$ over $x_s$ means that the entry is omitted.
Hence, we obtain
\begin{align*}
\frac{\partial}{\partial x_s}\det(\lambda I_n - M_n) &= - \det(\lambda I_{s-1} - M_{[1,s-1]}) \det(\lambda I_{n-s} - M_{[s+1,n]}) \\
&= - \left(\sum_{u=0}^{s-1} (-1)^u E_u^{[1,s-1]} \lambda^{s-1-u} \right) \left( \sum_{v=0}^{n-s} (-1)^v E_v^{[s+1,n]} \lambda^{n-s-v} \right) \\
& \hspace{30pt} \textrm{(by Definition~\ref{definition:Eijab})} \\
&= - \sum_{u=0}^{s-1} \sum_{v=0}^{n-s} (-1)^{u+v} E_u^{[1,s-1]}E_v^{[s+1,n]} \lambda^{n-1-(u+v)} \\ 
&= - \sum_{\ell=0}^{n-1} \left( \sum_{k=0}^{\ell} (-1)^{\ell} E_{\ell-k}^{[1,s-1]}E_k^{[s+1,n]} \right) \lambda^{n-1-\ell}. 
\end{align*}
Therefore, we conclude that the coefficient of $\lambda^{n-i}$ for $\frac{\partial}{\partial x_s}\det(\lambda I_n - M_n)$ multiplied by $(-1)^{i}$ is equal to
\begin{align*}
\frac{\partial}{\partial x_s} E_i^{(n)} = \sum_{k=0}^{i-1} E_{i-1-k}^{[1,s-1]}E_k^{[s+1,n]}, 
\end{align*}
as desired.

(2) We partial differentiate the both sides of \eqref{eq:proof_partial_derivatives} with respect to $q_{rs}$, then $\frac{\partial}{\partial q_{rs}} E_j^{(n)}$ is equal to the coefficient of $\lambda^{n-j}$ for $\frac{\partial}{\partial q_{rs}}\det(\lambda I_n - M_n)$ multiplied by $(-1)^{j}$. 
Since the variable $q_{rs}$ appears in only the $(r,s)$-th entry of the matrix $(\lambda I_n - M_n)$, one can see from similar arguments as in the previous case that
\begin{align*}
\frac{\partial}{\partial q_{rs}}\det(\lambda I_n - M_n) &= (-1)^{r+s} \frac{\partial}{\partial q_{rs}} (-q_{rs}) \det(\lambda I_{r-1} - M_{[1,r-1]}) \det(\lambda I_{n-s} - M_{[s+1,n]}) \\
&= (-1)^{r+s+1} \left(\sum_{u=0}^{r-1} (-1)^u E_u^{[1,r-1]} \lambda^{r-1-u} \right) \left( \sum_{v=0}^{n-s} (-1)^v E_v^{[s+1,n]} \lambda^{n-s-v} \right) \\
&= (-1)^{r+s+1} \sum_{u=0}^{r-1} \sum_{v=0}^{n-s} (-1)^{u+v} E_u^{[1,r-1]}E_v^{[s+1,n]} \lambda^{n-1-(u+v+s-r)} \\ 
&= (-1)^{r+s+1} \sum_{\ell=s-r}^{n-1} \left( \sum_{k=0}^{\ell-s+r} (-1)^{\ell-s+r} E_{\ell-s+r-k}^{[1,r-1]}E_k^{[s+1,n]} \right) \lambda^{n-1-\ell}. 
\end{align*}
Thus, the coefficient of $\lambda^{n-j}$ for $\frac{\partial}{\partial q_{rs}}\det(\lambda I_n - M_n)$ multiplied by $(-1)^{j}$ is 
\begin{align*}
\frac{\partial}{\partial q_{rs}} E_j^{(n)} = 
\begin{cases}
0 \ \ \ & \textrm{if} \ 1 \leq j \leq s-r, \\
\sum_{k=0}^{j-1-s+r} E_{j-1-s+r-k}^{[1,r-1]}E_k^{[s+1,n]} \ \ \ & \textrm{if} \ s-r+1 \leq j \leq n, 
\end{cases}
\end{align*}
as desired. 
This completes the proof. 
\end{proof}

In what follows, we set 
\begin{align} \label{eq:qss}
q_{ss} \coloneqq x_s \ \ \ \textrm{for} \ s \in [n],
\end{align}
and we see the Jacobian matrix $\big(\frac{\partial}{\partial q_{rs}} E_{i}^{(n)} \big)_{i,(r,s)}$ below.

\begin{example} \label{ex:Jacobi_n=34Flag}
Let $n=3$. Then the Jacobian matrix $\big(\frac{\partial}{\partial q_{rs}} E_{i}^{(3)} \big)_{i,(r,s)}$ is described as 
\begin{center}
\begin{picture}(300,110)
\put(0,70){$E_1^{(3)}$}
\put(0,40){$E_2^{(3)}$}
\put(0,10){$E_3^{(3)}$}
\put(45,95){$\frac{\partial}{\partial x_{1}}$}
\put(100,95){$\frac{\partial}{\partial x_{2}}$}
\put(155,95){$\frac{\partial}{\partial x_{3}}$}
\put(190,95){$\frac{\partial}{\partial q_{12}}$}
\put(225,95){$\frac{\partial}{\partial q_{23}}$}
\put(250,95){$\frac{\partial}{\partial q_{13}}$}
\put(25,40){\(
\left(
\begin{array}{cccccc}
1 & 1 & 1 & 0 & 0 & 0 \\
 &  &  &  &  &  \\
E_1^{[2,3]} & E_1^{[1,1]}+E_1^{[3,3]} & E_1^{[1,2]} & 1 & 1 & 0 \\
 &  &  &  &  &  \\
E_2^{[2,3]} & E_1^{[1,1]}E_1^{[3,3]} & E_2^{[1,2]} & E_1^{[3,3]} & E_1^{[1,1]} & 1
\end{array}
\right)
\)}
\end{picture}
\end{center}
by Lemma~\ref{lemma:partial_derivatives}.
Also, the Jacobian matrix $\big(\frac{\partial}{\partial q_{rs}} E_{i}^{(4)} \big)_{i,(r,s)}$ for $n=4$ is given by 
\begin{center}
\begin{picture}(450,105)
\put(0,73){\tiny $E_1^{(4)}$}
\put(0,52){\tiny $E_2^{(4)}$}
\put(0,31){\tiny $E_3^{(4)}$}
\put(0,10){\tiny $E_4^{(4)}$}
\put(35,91){$\frac{\partial}{\partial x_{1}}$}
\put(91,91){$\frac{\partial}{\partial x_{2}}$}
\put(175,91){$\frac{\partial}{\partial x_{3}}$}
\put(232,91){$\frac{\partial}{\partial x_{4}}$}
\put(262,91){$\frac{\partial}{\partial q_{12}}$}
\put(310,91){$\frac{\partial}{\partial q_{23}}$}
\put(355,91){$\frac{\partial}{\partial q_{34}}$}
\put(386,91){$\frac{\partial}{\partial q_{13}}$}
\put(417,91){$\frac{\partial}{\partial q_{24}}$}
\put(440,91){$\frac{\partial}{\partial q_{14}}$}
\put(20,41){\tiny \(
\left(
\begin{array}{cccccccccc}
1 & 1 & 1 & 1 & 0 & 0 & 0 & 0 & 0 & 0 \\
 &  &  &  &  &  &  &  &  &  \\
E_1^{[2,4]} & E_1^{[1,1]}+E_1^{[3,4]} & E_1^{[1,2]}+E_1^{[4,4]} & E_1^{[1,3]} & 1 & 1 & 1 & 0 & 0 & 0 \\
 &  &  &  &  &  &  &  &  &  \\
E_2^{[2,4]} & E_1^{[1,1]}E_1^{[3,4]}+E_2^{[3,4]} & E_2^{[1,2]}+E_1^{[1,2]}E_1^{[4,4]} & E_2^{[1,3]} & E_1^{[3,4]} & E_1^{[1,1]}+E_1^{[4,4]} & E_1^{[1,2]} & 1 & 1 & 0 \\
 &  &  &  &  &  &  &  &  &  \\
E_3^{[2,4]} & E_1^{[1,1]}E_2^{[3,4]} & E_2^{[1,2]}E_1^{[4,4]} & E_3^{[1,3]} & E_2^{[3,4]} & E_1^{[1,1]}E_1^{[4,4]} & E_2^{[1,2]} & E_1^{[4,4]} & E_1^{[1,1]} & 1
\end{array}
\right).
\)}
\end{picture}
\end{center}
The Jacobian matrices above have a full rank. 
In general, one can verify from Lemma~\ref{lemma:partial_derivatives} that the rank of the Jacobian matrix $\big(\frac{\partial}{\partial q_{rs}} E_{i}^{(n)} \big)_{i,(r,s)}$ is full for arbitrary $n$.
This fact also follows from the well-known fact that $\Fl(\C^n) \cap \Omega_e^\circ$ is smooth and Theorem~\ref{theorem:main1} for the case when $h=(n,\ldots,n)$.
\end{example}

Let $h: [n] \to [n]$ be a Hessenberg function. 
As in Definition~\ref{definition:hEi(n)}, for $0 \leq i \leq n$ and $a,b$ with $b \geq a-1$, we define 
\begin{align} \label{eq:hEiab}
\hE_i^{[a,b]} \coloneqq E_i^{[a,b]}|_{q_{rs}=0 \ (2 \leq s \leq n \ \textrm{and} \ 1 \leq r \leq n-h(n+1-s))}. 
\end{align}
Note that we take the convention in \eqref{eq:convention_Eab}.
By the definition \eqref{eq:hEiab} it is straightforward to see that for arbitrary Hessenberg function $h: [n] \to [n]$, $i \in [n]$, and $(r,s)$ with $2 \leq s \leq n, n-h(n+1-s)<r \leq s$, 
\begin{align*} 
\left. \frac{\partial}{\partial q_{rs}} \, \hE_{i}^{(n)} = \left( \frac{\partial}{\partial q_{rs}} E_{i}^{(n)} \right) \right|_{q_{uv}=0 \ (2 \leq v \leq n \ \textrm{and} \ 1 \leq u \leq n-h(n+1-v))}.
\end{align*}
Combining this and Lemma~\ref{lemma:partial_derivatives}, we have that for $1 \leq r \leq s \leq n$, 
\begin{align} \label{eq:partial_derivatives}
\frac{\partial}{\partial q_{rs}} \, \hE_{i+s-r}^{(n)} = \begin{cases}
0 & \textrm{if} \ 1-(s-r) \leq i \leq 0, \\
1 & \textrm{if} \ i = 1, \\
\sum_{k=0}^{i-1} \, \hE_{i-1-k}^{[1,r-1]} \, \hE_{k}^{[s+1,n]} & \textrm{if} \ 2 \leq i \leq n-(s-r). 
\end{cases}
\end{align}

For $2 \leq m \leq n-1$, we define a Hessenberg function $h_m:[n] \rightarrow [n]$ by 
\begin{align} \label{eq:hm}
h_m=(m,n,\ldots,n).
\end{align}
The zero set of $\{\hmE_1^{(n)}, \ldots, \, \hmE_n^{(n)}\}$ in $\C^{\frac{1}{2}n(n+1)-(n-m)}$ with the variables $x_1.\ldots,x_n,q_{rs} \ (2 \leq s \leq n, n-h_m(n+1-s)<r < s)$ is denoted by $V\big(\, \hmE_1^{(n)}, \ldots, \, \hmE_n^{(n)}\big)$, i.e. 
\begin{align} \label{eq:zero_set_hmE}
V\big(\, \hmE_1^{(n)}, \ldots, \, \hmE_n^{(n)}\big) \coloneqq \{ (a,p) \in \C^{\frac{1}{2}n(n+1)-(n-m)} \mid \, \hmE_i^{(n)}(a,p)=0 \ \textrm{for} \ i \in [n] \}
\end{align}
where $(a,p) \coloneqq (a_1,\ldots,a_n,p_{rs})_{2 \leq s \leq n, n-h_m(n+1-s)<r < s}$.

\begin{proposition} \label{proposition:singular_locus_hm_quantum}
Let $2 \leq m \leq n-1$ and $h_m$ be the Hessenberg function in \eqref{eq:hm}. 
Then, the singular locus of $V\big(\, \hmE_1^{(n)}, \ldots, \, \hmE_n^{(n)}\big)$ in \eqref{eq:zero_set_hmE} is given by the solution set of the equations 
\begin{align*}
\frac{\partial}{\partial q_{rs}} \, \hmE_{n}^{(n)} = 0 \ \ \ \textrm{for all} \ 2 \leq s \leq n \ \textrm{and} \ n-h_m(n+1-s)<r \leq s.
\end{align*}
Here, we recall our convention \eqref{eq:qss} that $q_{rs}=x_s$ whenever $r=s$.
\end{proposition}

Before we prove Proposition~\ref{proposition:singular_locus_hm_quantum}, we give an example of the singular locus of the zero set $V\big(\, \hmE_1^{(3)}, \, \hmE_2^{(3)}, \, \hmE_3^{(3)}\big)$ for $m=2$ by using Proposition~\ref{proposition:singular_locus_hm_quantum}.

\begin{example} \label{ex:singular_locus_n=3}
Consider the case when $h=(2,3,3)$ for $n=3$. 
The Jacobian matrix $\big(\frac{\partial}{\partial q_{rs}} \, \hE_{i}^{(3)} \big)_{i,(r,s) \neq (1,3)}$ is obtained from $\big(\frac{\partial}{\partial q_{rs}} E_{i}^{(3)} \big)_{i,(r,s)}$ by forgetting the quantum parameter $q_{13}$. 
As seen in Example~\ref{ex:Jacobi_n=34Flag}, the Jacobian matrix $\big(\frac{\partial}{\partial q_{rs}} \, \hE_{i}^{(3)} \big)_{i,(r,s) \neq (1,3)}$ is shown as 
\begin{center}
\begin{picture}(300,110)
\put(0,70){$\hE_1^{(3)}$}
\put(0,40){$\hE_2^{(3)}$}
\put(0,10){$\hE_3^{(3)}$}
\put(53,95){$\frac{\partial}{\partial x_{1}}$}
\put(118,95){$\frac{\partial}{\partial x_{2}}$}
\put(183,95){$\frac{\partial}{\partial x_{3}}$}
\put(222,95){$\frac{\partial}{\partial q_{12}}$}
\put(263,95){$\frac{\partial}{\partial q_{23}}$}
\put(30,40){\(
\left(
\begin{array}{ccccc}
1 & 1 & 1 & 0 & 0 \\
 &  &  &  &  \\
\hE_1^{[2,3]} & \hE_1^{[1,1]}+ \, \hE_1^{[3,3]} & \hE_1^{[1,2]} & 1 & 1 \\
 &  &  &  &  \\
\hE_2^{[2,3]} & \hE_1^{[1,1]} \, \hE_1^{[3,3]} & \hE_2^{[1,2]} & \hE_1^{[3,3]} & \hE_1^{[1,1]} 
\end{array}
\right).
\)}
\end{picture}
\end{center}
It follows from Proposition~\ref{proposition:singular_locus_hm_quantum} that the singular locus of $V\big(\, \hE_1^{(3)}, \, \hE_2^{(3)}, \, \hE_3^{(3)}\big)$ is given by the solution set of the equations $\hE_2^{[2,3]}=\, \hE_1^{[1,1]} \, \hE_1^{[3,3]}=\, \hE_2^{[1,2]}=\, \hE_1^{[3,3]}=\, \hE_1^{[1,1]}=0$.
The resulting solution is as follows:
\begin{align*}
x_1=0, \ x_3=0, \ q_{12}=0, \ q_{23}=0.
\end{align*}
Then it follows from \eqref{eq:maintheorem1-2} and Proposition~\ref{proposition:inverse_map_varphih} that the singular locus of $\Pet_3 \cap \Omega_e^\circ$ is given by the solution set of the equations $x_{32}=0, \ x_{21}=0, \ F_{3,2}=0, \ F_{2,1}=0$. This is equivalent to 
\begin{align*}
x_{21}=0, \ x_{32}=0, \ x_{31}=0.
\end{align*}
Hence, the singular locus of $\Pet_3 \cap \Omega_e^\circ$ is $\{eB \}$.
Note that the singular locus of the Peterson variety $\Pet_n$ is given by \cite{IY}. 
We will explain this in Appendix~\ref{section:Singular locus of Peterson variety}.
\end{example}

\begin{proof}[Proof of Proposition~\ref{proposition:singular_locus_hm_quantum}]
Recall that a point $(a,p)\coloneqq(a_1,\ldots,a_n,p_{rs})_{2 \leq s \leq n, n-h_m(n+1-s)<r < s}$ in $\C^{\frac{1}{2}n(n+1)-(n-m)}$ is a singular point of $V\big(\, \hmE_1^{(n)}, \ldots, \, \hmE_n^{(n)}\big)$ if and only if the rank of the Jacobian matrix $\big(\frac{\partial}{\partial q_{rs}} \, \hmE_{i}^{(n)}(a,p) \big)_{i,(r,s)}$ is not full.
It is clear that the Jacobian matrix $\big(\frac{\partial}{\partial q_{rs}} \, \hmE_{i}^{(n)}(a,p) \big)_{i,(r,s)}$ is not full if the $n$-th row vector is zero. 
Thus, it is enough to prove that if a point $(a,p)$ in $\C^{\frac{1}{2}n(n+1)-(n-m)}$ is a singular point of $V\big(\, \hmE_1^{(n)}, \ldots, \, \hmE_n^{(n)}\big)$, then the $n$-th row vector $\big( \frac{\partial}{\partial q_{rs}} \, \hmE_{n}^{(n)}(a,p) \big)_{(r,s)}$ of the Jacobian matrix is zero. 

Since the column vector with respect to $\frac{\partial}{\partial q_{rs}}$ of the Jacobian matrix $\big(\frac{\partial}{\partial q_{rs}} \, \hmE_{i}^{(n)}(a,p) \big)_{i,(r,s)}$ is of the form $(\underbrace{0,\ldots,0}_{s-r},1,*,\ldots,*)^t$ by \eqref{eq:partial_derivatives}, the first $n-1$ row vectors of the Jacobian matrix $\big(\frac{\partial}{\partial q_{rs}} \, \hmE_{i}^{(n)}(a,p) \big)_{i,(r,s)}$ are linearly independent.  
By the assumption that $\big(\frac{\partial}{\partial q_{rs}} \, \hmE_{i}^{(n)}(a,p) \big)_{i,(r,s)}$ does not have full rank, the $n$-th row vector must be written as a linear combination of the first $n-1$ row vectors, i.e. 
\begin{align} \label{eq:proof_singular_locus_hm_quantum}
\left(\frac{\partial}{\partial q_{rs}} \, \hmE_{n}^{(n)}(a,p) \right)_{(r,s)} = \sum_{i=1}^{n-1} c_i \left(\frac{\partial}{\partial q_{rs}} \, \hmE_{i}^{(n)}(a,p) \right)_{(r,s)}
\end{align}
for some $c_1,\ldots,c_{n-1} \in \C$. 
We note that the pair $(r,s)$ can be taken as $2 \leq s \leq n$ and $n-h_m(n+1-s)<r \leq s$ in the equality above.
Recall that we denote the singular point $(a_1,\ldots,a_n,p_{rs})_{2 \leq s \leq n, n-h_m(n+1-s)<r < s}$ of $V\big(\, \hmE_1^{(n)}, \ldots, \, \hmE_n^{(n)}\big)$ by $(a,p)$.

\smallskip

\noindent
\textbf{Claim~1. The coefficients $c_i$ of \eqref{eq:proof_singular_locus_hm_quantum} must be $(-1)^{n-i+1} a_n^{n-i}$ for $i \in [n-1]$.} 

We prove Claim~1 by descending induction on $i$. 
The base case is $i=n-1$. 
Comparing the $(r,s)=(1,n-1)$-th component of \eqref{eq:proof_singular_locus_hm_quantum}, we have 
\begin{align*}
\hmE_1^{[n,n]}(a,p) = c_1 \cdot 0 + c_2 \cdot 0 + \cdots + c_{n-2} \cdot 0 + c_{n-1} \cdot 1 
\end{align*}
by \eqref{eq:partial_derivatives}.
Since $\hmE_1^{[n,n]}=x_n$, the equality above implies that $a_n=c_{n-1}$.
This shows the base case.

We now assume that $i<n-1$ and the assertion of the claim holds for arbitrary $k$ with $k \geq i+1$.
It follows from \eqref{eq:partial_derivatives} that the $(r,s)=(n-i,n-1)$-th component of \eqref{eq:proof_singular_locus_hm_quantum} gives 
\begin{align*}
&\hmE_{n-i-1}^{[1,n-i-1]}(a,p) \cdot \,\hmE_1^{[n,n]}(a,p) \\
=&c_i \cdot 1 + c_{i+1} \cdot \left( \hmE_1^{[1,n-i-1]}(a,p)+\,\hmE_1^{[n,n]}(a,p) \right) \\
&+ c_{i+2} \cdot \left( \hmE_2^{[1,n-i-1]}(a,p)+\,\hmE_1^{[1,n-i-1]}(a,p)\cdot\,\hmE_1^{[n,n]}(a,p) \right) \\
&+ c_{i+3} \cdot \left( \hmE_3^{[1,n-i-1]}(a,p)+\,\hmE_2^{[1,n-i-1]}(a,p)\cdot\,\hmE_1^{[n,n]}(a,p) \right) \\
&+ \cdots + c_{n-1} \cdot \left( \hmE_{n-i-1}^{[1,n-i-1]}(a,p)+\,\hmE_{n-i-2}^{[1,n-i-1]}(a,p)\cdot\,\hmE_1^{[n,n]}(a,p) \right). 
\end{align*}
Since $\hmE_1^{[n,n]}(a,p)=a_n$ and $c_k=(-1)^{n-k+1} a_n^{n-k}$ for all $i+1 \leq k \leq n-1$ by our inductive assumption, $\hmE_{n-i-1}^{[1,n-i-1]}(a,p) \cdot a_n$ is equal to
\begin{align*}
&c_i + \sum_{k=1}^{n-i-1} (-1)^{n-i-k+1} a_n^{n-i-k} \left( \hmE_k^{[1,n-i-1]}(a,p)+\,\hmE_{k-1}^{[1,n-i-1]}(a,p)\cdot a_n \right) \\
= & c_i + (-1)^{n-i} a_n^{n-i-1} \cdot a_n + a_n \cdot \,\hmE_{n-i-1}^{[1,n-i-1]}(a,p).
\end{align*}
%\begin{align*}
%\hmE_{n-i-1}^{[1,n-i-1]}(a,p) \cdot a_n = &c_i + (-1)^{n-i} a_n^{n-i-1} \cdot \left( \hmE_1^{[1,n-i-1]}(a,p)+a_n \right) \\
%&+ (-1)^{n-i-1} a_n^{n-i-2} \cdot \left( \hmE_2^{[1,n-i-1]}(a,p)+\,\hmE_1^{[1,n-i-1]}(a,p)\cdot a_n \right) \\
%&+ (-1)^{n-i-2} a_n^{n-i-3} \cdot \left( \hmE_3^{[1,n-i-1]}(a,p)+\,\hmE_2^{[1,n-i-1]}(a,p)\cdot a_n \right) \\
%&+ \cdots + a_n \cdot \left( \hmE_{n-i-1}^{[1,n-i-1]}(a,p)+\,\hmE_{n-i-2}^{[1,n-i-1]}(a,p)\cdot a_n \right) \\
%=&c_i + (-1)^{n-i} a_n^{n-i-1} \cdot a_n + a_n \cdot \,\hmE_{n-i-1}^{[1,n-i-1]}(a,p).
%\end{align*}
This yields that $c_i=(-1)^{n-i+1} a_n^{n-i}$ and we proved Claim~1.

It follows from Claim~1 and \eqref{eq:proof_singular_locus_hm_quantum} that
\begin{align} \label{eq:proof_singular_locus_hm_quantum_2}
\left(\frac{\partial}{\partial q_{rs}} \, \hmE_{n}^{(n)}(a,p) \right)_{(r,s)} = \sum_{i=1}^{n-1} (-1)^{n-i+1} a_n^{n-i} \left(\frac{\partial}{\partial q_{rs}} \, \hmE_{i}^{(n)}(a,p) \right)_{(r,s)}
\end{align}
for all $2 \leq s \leq n$ and $n-h_m(n+1-s)<r \leq s$.

\smallskip

\noindent
\textbf{Claim~2. It holds that $p_{in}=0$ for all $n-m+1 \leq i \leq n-1$.} 

We show Claim~2 by descending induction on $i$. 
The base case is $i=n-1$. 
The $(r,s)=(1,n-2)$-th component of \eqref{eq:proof_singular_locus_hm_quantum_2} is equal to 
\begin{align*}
\hmE_2^{[n-1,n]}(a,p) = -a_n^2 \cdot 1 + a_n \cdot \, \hmE_1^{[n-1,n]}(a,p) 
\end{align*}
from \eqref{eq:partial_derivatives}. 
Since $\hmE_2^{[n-1,n]}(a,p)=a_{n-1}a_n+p_{n-1 \, n}$ and $\hmE_1^{[n-1,n]}(a,p)=a_{n-1}+a_n$, we have $p_{n-1 \, n}=0$, which proves the base case.

Suppose now that $i < n-1$ and that the claim holds for any $k$ with $k \geq i+1$, i.e. $p_{kn}=0$ for all $i+1 \leq k \leq n-1$. 
By \eqref{eq:partial_derivatives} the $(r,s)=(1,i-1)$-th component of \eqref{eq:proof_singular_locus_hm_quantum_2} is 
\begin{align} \label{eq:proof_singular_locus_hm_quantum_3}
\hmE_{n-i+1}^{[i,n]}(a,p) = &(-1)^{n-i+2}a_n^{n-i+1} \cdot 1 + (-1)^{n-i+1}a_n^{n-i} \cdot \, \hmE_1^{[i,n]}(a,p) \\
&+ (-1)^{n-i}a_n^{n-i-1} \cdot \, \hmE_2^{[i,n]}(a,p) +\cdots+a_n \cdot \, \hmE_{n-i}^{[i,n]}(a,p). \notag
\end{align}
Here, we note that 
\begin{align*}
\det (\lambda I_{n-i+1} - M_{[i,n]})|_{q_{kn}=0 \ (i+1 \leq k \leq n-1)} = (-1)^{n-i} (-q_{in}) + \det (\lambda I_{n-i} - M_{[i,n-1]}) \cdot (\lambda-x_n)
\end{align*}
by the cofactor expansion along the last column. 
The left hand side is written as
\begin{align*}
\sum_{\ell=0}^{n-i+1} \big( (-1)^\ell E_\ell^{[i,n]}|_{q_{kn}=0 \ (i+1 \leq k \leq n-1)} \big) \lambda^{n-i+1-\ell}
\end{align*}
and the right hand side is 
\begin{align*}
\lambda^{n-i+1} + \sum_{\ell=1}^{n-i} \big( (-1)^\ell (x_n E_{\ell-1}^{[i,n-1]} + E_\ell^{[i,n-1]}) \big) \lambda^{n-i+1-\ell} + (-1)^{n-i+1} (x_n E_{n-i}^{[i,n-1]} + q_{in})
\end{align*}
by definition.
Thus, we obtain that  
\begin{align*}
E_\ell^{[i,n]}|_{q_{kn}=0 \ (i+1 \leq k \leq n-1)} &= x_n E_{\ell-1}^{[i,n-1]} + E_\ell^{[i,n-1]} \ \ \ \textrm{for} \ 1 \leq \ell \leq n-i; \\
E_{n-i+1}^{[i,n]}|_{q_{kn}=0 \ (i+1 \leq k \leq n-1)} &=x_n E_{n-i}^{[i,n-1]} + q_{in}. 
\end{align*}
In particular, by our inductive hypothesis $p_{kn}=0$ for all $i+1 \leq k \leq n-1$, one has
\begin{align*}
\hmE_\ell^{[i,n]}(a,p) &= a_n \, \hmE_{\ell-1}^{[i,n-1]}(a,p) + \, \hmE_\ell^{[i,n-1]}(a,p) \ \ \ \textrm{for} \ 1 \leq \ell \leq n-i; \\
\hmE_{n-i+1}^{[i,n]}(a,p) &=a_n \, \hmE_{n-i}^{[i,n-1]}(a,p) + p_{in}. 
\end{align*}
Substituting these equalities to \eqref{eq:proof_singular_locus_hm_quantum_3}, the left hand side of 
\eqref{eq:proof_singular_locus_hm_quantum_3} is 
\begin{align*} 
a_n \, \hmE_{n-i}^{[i,n-1]}(a,p) + p_{in}.
\end{align*}
On the other hand, the right hand side of 
\eqref{eq:proof_singular_locus_hm_quantum_3} is 
%\begin{align*} 
%&(-1)^{n-i+2}a_n^{n-i+1} + (-1)^{n-i+1}a_n^{n-i} \cdot \big( a_n + \, \hmE_1^{[i,n-1]}(a,p) \big) \\
%&+ (-1)^{n-i}a_n^{n-i-1} \cdot \big( a_n \, \hmE_{1}^{[i,n-1]}(a,p) + \, \hmE_2^{[i,n-1]}(a,p) \big)\\
%&+ (-1)^{n-i-1}a_n^{n-i-2} \cdot \big( a_n \, \hmE_{2}^{[i,n-1]}(a,p) + \, \hmE_3^{[i,n-1]}(a,p) \big)\\
%&+ \cdots + a_n \cdot \big( a_n \, \hmE_{n-i-1}^{[i,n-1]}(a,p) + \, \hmE_{n-i}^{[i,n-1]}(a,p) \big) \\
%=& a_n \cdot \, \hmE_{n-i}^{[i,n-1]}(a,p).
%\end{align*}
\begin{align*}
&(-1)^{n-i+2}a_n^{n-i+1} + \sum_{k=1}^{n-i} (-1)^{n-i-k+2} a_n^{n-i-k+1} \cdot \left(a_n \,  \hmE_{k-1}^{[i,n-1]}(a,p)+\,\hmE_{k}^{[i,n-1]}(a,p) \right) \\
= & a_n \cdot \, \hmE_{n-i}^{[i,n-1]}(a,p).
\end{align*}
Hence, we obtain that $p_{in}=0$ as desired. This proves Claim~2.

\smallskip

\noindent
\textbf{Claim~3. We have $a_n=0$.} 

It follows from Theorem~\ref{theorem:main1} that 
\begin{align*}
\Hess(N,h_m) \cap \Omega_e^\circ \cong V\big(\, \hmE_1^{(n)}, \ldots, \, \hmE_n^{(n)}\big).
\end{align*}
We denote by $b=(b_{ij})_{1 \leq j < i \leq n} \in \Hess(N,h_m) \cap \Omega_e^\circ \subset \C^{\frac{1}{2}n(n-1)}$ the image of the singular point $(a,p)$ of $V\big(\, \hmE_1^{(n)}, \ldots, \, \hmE_n^{(n)}\big)$ under the isomorphism above. 
One can see from Proposition~\ref{proposition:inverse_map_varphih} and Claim~2 that 
\begin{align} \label{eq:proof_claim3_1}
F_{i,1}(b)=-q_{n+1-i \, n}(a,p)=-p_{n+1-i \, n}=0 \ \ \ \textrm{for} \ 2 \leq i \leq m.
\end{align}
Since $b \in \Hess(N,h_m) \cap \Omega_e^\circ$, we also have 
\begin{align} \label{eq:proof_claim3_2}
F_{i,1}(b)=0 \ \ \ \textrm{for} \ m+1 \leq i \leq n
\end{align}
by Lemma~\ref{lemma:defining equation}. 
It follows from \eqref{eq:proof_claim3_1}, \eqref{eq:proof_claim3_2}, and Lemma~\ref{lemma:defining equation} again that the point $b=(b_{ij})_{1 \leq j < i \leq n}$ belongs to $\Hess(N,h_1) \cap \Omega_e^\circ$ where $h_1$ is the decomposable Hessenberg function defined by $h_1 \coloneqq (1,n,\ldots,n)$. 
As seen in Definition~\ref{definition:indecomposable} and surrounding
discussion, every flag $V_\bullet \in \Hess(N,h_1)$ has $V_1 = \C \cdot (1,0,\ldots,0)^t$ which implies that $b_{i1}=0$ for all $2 \leq i \leq n$. 
It then follows from Proposition~\ref{proposition:inverse_map_varphih} that 
\begin{align*} 
a_n=x_n(a,p)=-x_{21}(b)=-b_{21}=0, 
\end{align*}
as desired. 
We have proven Claim~3.

Combining Claim~3 and \eqref{eq:proof_singular_locus_hm_quantum_2}, we conclude that the $n$-th row vector $\big( \frac{\partial}{\partial q_{rs}} \, \hmE_{n}^{(n)}(a,p) \big)_{(r,s)}$ of the Jacobian matrix is zero. 
This completes the proof. 
\end{proof}

By using Proposition~\ref{proposition:singular_locus_hm_quantum}, we will explicitly describe the singular locus of $\Hess(N,h_m) \cap \Omega_e^\circ$ in Section~\ref{section:Singular locus of Hess(N,hm)}.
For this purpose, we will first study the singularity of $\Hess(N,h_2) \cap \Omega_e^\circ$ in next section.

\bigskip

%%%%%%%%%%%%%%%%%%%%%%%%%%%%%%%%%%
\section{Cyclic quotient singularity}
\label{section:Cyclic quotient}
%%%%%%%%%%%%%%%%%%%%%%%%%%%%%%%%%%

In this section we analyze the singularity of $\Hess(N,h_2) \cap \Omega_e^\circ$ where $h_2=(2,n,\ldots,n)$. 
In fact, we see that the singularity is related with a cyclic quotient singularity. 
Then we can explicitly describe the singular locus of $\Hess(N,h_2) \cap \Omega_e^\circ$. 

First, we study relations between $F_{i,j}$'s and $\tilde F^{\langle\mm_j\rangle}_{i,j}$'s, which are 
defined in \eqref{eq:Fi1explicit} and \eqref{eq:Fijexplicit}, and Definition~\ref{definition:tildeFij}, respectively. 

\begin{lemma} \label{lemma:Fij_tildeFij}
Let $1 \leq j \leq n-1$ and $j \leq \mm_j <n$. 
For $i>\mm_j$, we have
\begin{align*} 
F_{i,j} = \tilde F^{\langle\mm_j\rangle}_{i,j} - \sum_{\ell=\mm_j+1}^{i-1} x_{i\ell} F_{\ell,j}
\end{align*}
where we take the convention that $\sum_{\ell=\mm_j+1}^{\mm_j} x_{i\ell} F_{\ell,j}=0$.
\end{lemma}

\begin{proof}
It suffices to show that  
\begin{align} \label{eq:Fi1proof2} 
\tilde F^{\langle i-k \rangle}_{i,j} = \tilde F^{\langle \mm_j \rangle}_{i,j} - \sum_{\ell=\mm_j+1}^{i-k} x_{i\ell} F_{\ell,j} \ \ \ \textrm{for} \ 1 \leq k \leq i-\mm_j. 
\end{align}
Indeed, \eqref{eq:Fi1proof2} for $k=1$ is the desired equality by \eqref{eq:Fij_and_tildeFij}.
We prove \eqref{eq:Fi1proof2} by descending induction on $k$.
The base case is $k=i-\mm_j$, which is clear.
Now suppose that $k < i-\mm_j$ and assume by induction that the claim is true for $k+1$, i.e. 
 \begin{align} \label{eq:Fi1proof3} 
\tilde F^{\langle i-k-1 \rangle}_{i,j} = \tilde F^{\langle \mm_j \rangle}_{i,j} - \sum_{\ell=\mm_j+1}^{i-k-1} x_{i\ell} F_{\ell,j}. 
\end{align}
Then we show \eqref{eq:Fi1proof2}. 
The left hand side of \eqref{eq:Fi1proof2} is equal to 
\begin{align*}
 \tilde F^{\langle i-k-1 \rangle}_{i,j} - x_{i\,i-k} F_{i-k,j} 
\end{align*}
by using the cofactor expansion along the second-to-last column.
Combining this with the inductive hypothesis \eqref{eq:Fi1proof3}, we have proven \eqref{eq:Fi1proof2}.
\end{proof}

\begin{proposition} \label{proposition:Fij_ideal}
Let $1 \leq j \leq n-1$ and $j \leq \mm_j<n$.
For $i>\mm_j$, the ideal 
\begin{align*}
(F_{\mm_j+1,j},F_{\mm_j+2,j},\ldots,F_{i,j})
\end{align*}
is equal to the ideal
\begin{align*}
(\tilde F^{\langle\mm_j\rangle}_{\mm_j+1,j},\tilde F^{\langle\mm_j\rangle}_{\mm_j+2,j},\ldots,\tilde F^{\langle\mm_j\rangle}_{i,j})
\end{align*}
in the polynomial ring $\C[x_{ij} \mid 1 \leq j < i \leq n]$. 
In particular, if $h:[n] \to [n]$ is an indecomposable Hessenberg function such that $h \neq (n,\ldots,n)$, then we have 
\begin{align} \label{eq:coordinate_ring_tildeFij}
\C[\Hess(N,h) \cap \Omega_e^\circ] \cong \C[x_{ij} \mid 1 \leq j < i \leq n]/(\tilde F^{\langle h(j) \rangle}_{i,j} \mid j \in J_h \ \textrm{and} \ h(j) < i \leq n), 
\end{align}
where $J_h \coloneqq \{j \in [n-2] \mid h(j) <n \}$.
\end{proposition}

\begin{proof}
We prove the first statement by induction on $i$.
The base case $i = \mm_j+1$ is clear since $F_{\mm_j+1,j}=\tilde F^{\langle\mm_j\rangle}_{\mm_j+1,j}$ by \eqref{eq:Fij_and_tildeFij}. 
We proceed to the inductive step. Suppose that $i > \mm_j+1$ and that the claim holds for $i-1$.
Then we have
\begin{align*}
(F_{\mm_j+1,j},\ldots,F_{i-1,j},F_{i,j}) &=(F_{\mm_j+1,j},\ldots,F_{i-1,j}, \tilde F^{\langle\mm_j\rangle}_{i,j}) \ \ \ \textrm{(from Lemma~\ref{lemma:Fij_tildeFij})} \\
&=(\tilde F^{\langle\mm_j\rangle}_{\mm_j+1,j},\ldots,\tilde F^{\langle\mm_j\rangle}_{i-1,j},\tilde F^{\langle\mm_j\rangle}_{i,j})  \ \ \ \textrm{(by the inductive assumption)}
\end{align*}
as desired.
The isomorphism \eqref{eq:coordinate_ring_tildeFij} follows from \eqref{eq:coordinate_ring_Fij} and the former statement by setting $\mm_j=h(j)$.
\end{proof}

\begin{example} \label{example:mj=j,j+1}
Let $\mm_j=j+1$.
Then one has 
\begin{align}
&\tilde F^{\langle 2 \rangle}_{i,1}=\left|
 \begin{array}{@{\,}ccc@{\,}}
     1 & 0 & x_{21} \\
     x_{21} & 1 & x_{31} \\
     x_{i1} & x_{i2} & x_{i+1 \, 1} 
 \end{array}
 \right| \ \textrm{for} \ i > 2; \label{eqFi1m1=2def} \\
&\tilde F^{\langle j+1 \rangle}_{i,j}=\left|
 \begin{array}{@{\,}cccc@{\,}}
     1 & 0 & 0 & 1 \\
     x_{j \, j-1} & 1 & 0 & x_{j+1 \, j} \\
     x_{j+1 \, j-1} & x_{j+1 \, j} & 1 & x_{j+2 \, j} \\
     x_{i \, j-1} & x_{ij} & x_{i \, j+1} & x_{i+1 \, j} 
 \end{array}
 \right| \ \textrm{for} \ 2 \leq j \leq n-2 \ \textrm{and} \ i > j+1. \notag  
\end{align}
By \eqref{eq:coordinate_ring_tildeFij}, the coordinate ring $\C[\Pet_n \cap \Omega_e^\circ]$ is 
\begin{align} \label{eq:coordinate_ring_PetOmega} 
\C[\Pet_n \cap \Omega_e^\circ] \cong \C[x_{ij} \mid 1 \leq j < i \leq n]/(\tilde F^{\langle j+1 \rangle}_{i,j} \ \textrm{for all} \ 1 \leq j \leq n-2 \ \textrm{and} \ j+1 < i \leq n). 
\end{align}
\end{example}

We now explain the cyclic quotient singularity which is also called the type $A$-singularity.
Let $\zeta$ be a primitive $n$-th root of unity and $\Cyclic_n$ the cyclic group of order $n$ generated by $\zeta$. 
Consider the action of $\Cyclic_n$ on $\C^2$ defined by $\zeta \cdot (x,y)=(\zeta x, \zeta^{-1} y)$ for $\zeta \in \Cyclic_n$ and $(x,y) \in \C^2$.
This induces the action of $\Cyclic_n$ on the polynomial ring $\C[x,y]$ (which is the coordinate ring of $\C^2$) and it is given by $\zeta \cdot x=\zeta x$ and $\zeta \cdot y=\zeta^{-1} y$.
As is well-known, the $\Cyclic_n$-invariants in $\C[x,y]$ is isomorphic to  
\begin{align*}
\C[\C^2/\Cyclic_n] \cong \C[x,y]^{\Cyclic_n} \cong \C[X,Y,Z]/(XY-Z^n),
\end{align*}
which sends $X$ to $x^n$, $Y$ to $y^n$, and $Z$ to $xy$. 
The quotient space $\C^2/\Cyclic_n$ is called the \emph{cyclic quotient singularity} or the \emph{type $A_{n-1}$-singularity}.

\begin{example} \label{example:n=3cyclic}
Let $n=3$. 
As seen in Example~\ref{example:mj=j,j+1}, the coordinate ring of $\Pet_3 \cap \Omega_e^\circ$ is given by
\begin{align*} 
\C[\Pet_3 \cap \Omega_e^\circ] \cong \C[x_{21}, x_{31}, x_{32}]/(\tilde F^{\langle 2 \rangle}_{3,1}). 
\end{align*}
Here, one can compute $\tilde F^{\langle 2 \rangle}_{3,1}$ as
\begin{align*}
\tilde F^{\langle 2 \rangle}_{3,1} &=\left|
 \begin{array}{@{\,}ccc@{\,}}
     1 & 0 & x_{21} \\
     x_{21} & 1 & x_{31} \\
     x_{31} & x_{32} & 0 
 \end{array}
 \right|
 =x_{21}^2x_{32}-x_{21}x_{31}-x_{31}x_{32} \\
 &=-x_{21}^3+(x_{21}^2-x_{31})(x_{21}+x_{32}).
\end{align*} 
Thus, we have 
\begin{align*}
\C[\Pet_3 \cap \Omega_e^\circ] \cong \C[X,Y,Z]/(XY-Z^3),
\end{align*}
which sends $X$ to $x_{21}^2-x_{31}$, $Y$ to $x_{21}+x_{32}$, and $Z$ to $x_{21}$. 
The ring isomorphism above yields the isomorphism $\Pet_3 \cap \Omega_e^\circ \cong \C^2/\Cyclic_3$.
\end{example}

In Example~\ref{example:n=3cyclic}, we constructed the polynomial $XY-Z^n$ from $\tilde F^{\langle 2 \rangle}_{3,1}$ for $n=3$.
We now generalize this fact to arbitrary $n$. 
More specifically, we construct the polynomial $XY-Z^n$ from $\tilde F^{\langle 2 \rangle}_{i,1} \ (2 < i \leq n)$ defined in \eqref{eqFi1m1=2def}. 
By the cofactor expansion along the second column, we have
\begin{align} \label{eq:Fi1m1=2}
&\tilde F^{\langle 2 \rangle}_{i,1}=\left|
 \begin{array}{@{\,}cc@{\,}}
     1 & x_{21} \\
     x_{i1} & x_{i+1 \, 1} 
 \end{array}
 \right| - x_{i2}\left|
 \begin{array}{@{\,}cc@{\,}}
     1 & x_{21} \\
     x_{21} & x_{31} 
 \end{array}
 \right| \ \textrm{for} \ i > 2.
\end{align}

\begin{lemma} \label{lemma:cyclic}
For $2 < i \leq n$, we have 
\begin{align*}
&\tilde F^{\langle 2 \rangle}_{n,1}+x_{21}\tilde F^{\langle 2 \rangle}_{n-1,1}+x_{21}^2\tilde F^{\langle 2 \rangle}_{n-2,1}+\cdots+x_{21}^{n-i} \tilde F^{\langle 2 \rangle}_{i,1} \\
=& \left|
 \begin{array}{@{\,}cc@{\,}}
     1 & x_{21}^{n-i+1} \\
     x_{i1} & 0 
 \end{array}
 \right| - (x_{n2}+x_{21}x_{n-1 \, 2}+x_{21}^2x_{n-2 \, 2}+\cdots+x_{21}^{n-i}x_{i2})\left|
 \begin{array}{@{\,}cc@{\,}}
     1 & x_{21} \\
     x_{21} & x_{31} 
 \end{array}
 \right|.
\end{align*}
\end{lemma}

\begin{proof}
We prove this by descending induction on $i$. 
The base case $i=n$ is the equality \eqref{eq:Fi1m1=2} for $i=n$. 
Suppose now that $i < n$ and that the claim holds for $i+1$.
Then we have
\begin{align*}
&\tilde F^{\langle 2 \rangle}_{n,1}+x_{21}\tilde F^{\langle 2 \rangle}_{n-1,1}+x_{21}^2\tilde F^{\langle 2 \rangle}_{n-2,1}+\cdots+x_{21}^{n-i-1} \tilde F^{\langle 2 \rangle}_{i+1,1}+x_{21}^{n-i} \tilde F^{\langle 2 \rangle}_{i,1} \\
=& \left|
 \begin{array}{@{\,}cc@{\,}}
     1 & x_{21}^{n-i} \\
     x_{i+1 \, 1} & 0 
 \end{array}
 \right| - (x_{n2}+x_{21}x_{n-1 \, 2}+x_{21}^2x_{n-2 \, 2}+\cdots+x_{21}^{n-i-1}x_{i+1 \, 2})\left|
 \begin{array}{@{\,}cc@{\,}}
     1 & x_{21} \\
     x_{21} & x_{31} 
 \end{array}
 \right| + x_{21}^{n-i} \tilde F^{\langle 2 \rangle}_{i,1} \\
 & \hspace{80pt} \textrm{(by our descending induction hypothesis)} \\
=& \left|
 \begin{array}{@{\,}cc@{\,}}
     1 & x_{21}^{n-i} \\
     x_{i+1 \, 1} & 0 
 \end{array}
 \right| - (x_{n2}+x_{21}x_{n-1 \, 2}+x_{21}^2x_{n-2 \, 2}+\cdots+x_{21}^{n-i-1}x_{i+1 \, 2})\left|
 \begin{array}{@{\,}cc@{\,}}
     1 & x_{21} \\
     x_{21} & x_{31} 
 \end{array}
 \right| \\
 &+ x_{21}^{n-i} \left|
 \begin{array}{@{\,}cc@{\,}}
     1 & x_{21} \\
     x_{i1} & x_{i+1 \, 1} 
 \end{array}
 \right| - x_{21}^{n-i} x_{i2}\left|
 \begin{array}{@{\,}cc@{\,}}
     1 & x_{21} \\
     x_{21} & x_{31} 
 \end{array}
 \right| \hspace{30pt} \textrm{(by \eqref{eq:Fi1m1=2})} \\
=& \left|
 \begin{array}{@{\,}cc@{\,}}
     1 & x_{21}^{n-i+1} \\
     x_{i1} & 0 
 \end{array}
 \right| - (x_{n2}+x_{21}x_{n-1 \, 2}+x_{21}^2x_{n-2 \, 2}+\cdots+x_{21}^{n-i}x_{i2})\left|
 \begin{array}{@{\,}cc@{\,}}
     1 & x_{21} \\
     x_{21} & x_{31} 
 \end{array}
 \right|
\end{align*}
as desired. 
\end{proof}

\begin{proposition} \label{proposition:cyclic}
We set
\begin{align*}
\begin{cases}
X= x_{21}^2-x_{31} \\
Y= x_{21}^{n-2}+x_{32}x_{21}^{n-3}+\cdots+x_{n-1 \, 2}x_{21}+x_{n2} \\
Z= x_{21}.
\end{cases}
\end{align*}
Then one can write
\begin{align*}
\tilde F^{\langle 2 \rangle}_{n,1}+x_{21}\tilde F^{\langle 2 \rangle}_{n-1,1}+x_{21}^2\tilde F^{\langle 2 \rangle}_{n-2,1}+\cdots+x_{21}^{n-3} \tilde F^{\langle 2 \rangle}_{3,1} = XY-Z^n.
\end{align*}
\end{proposition}

\begin{proof}
By using Lemma~\ref{lemma:cyclic} for $i=3$, we obtain
\begin{align*}
&\tilde F^{\langle 2 \rangle}_{n,1}+x_{21}\tilde F^{\langle 2 \rangle}_{n-1,1}+x_{21}^2\tilde F^{\langle 2 \rangle}_{n-2,1}+\cdots+x_{21}^{n-3} \tilde F^{\langle 2 \rangle}_{3,1} \\
=& \left|
 \begin{array}{@{\,}cc@{\,}}
     1 & x_{21}^{n-2} \\
     x_{31} & 0 
 \end{array}
 \right| - (x_{n2}+x_{21}x_{n-1 \, 2}+x_{21}^2x_{n-2 \, 2}+\cdots+x_{21}^{n-3}x_{32})\left|
 \begin{array}{@{\,}cc@{\,}}
     1 & x_{21} \\
     x_{21} & x_{31} 
 \end{array}
 \right| \\
=& -x_{21}^{n-2}x_{31} - (x_{n2}+x_{21}x_{n-1 \, 2}+x_{21}^2x_{n-2 \, 2}+\cdots+x_{21}^{n-3}x_{32})(x_{31}-x_{21}^2) \\
=& -x_{21}^n+x_{21}^{n-2}(x_{21}^2-x_{31}) + (x_{n2}+x_{21}x_{n-1 \, 2}+x_{21}^2x_{n-2 \, 2}+\cdots+x_{21}^{n-3}x_{32})(x_{21}^2-x_{31})\\
=& -x_{21}^n + (x_{n2}+x_{21}x_{n-1 \, 2}+x_{21}^2x_{n-2 \, 2}+\cdots+x_{21}^{n-3}x_{32}+x_{21}^{n-2})(x_{21}^2-x_{31})\\
=&XY-Z^n.
\end{align*}
\end{proof}

\begin{theorem} \label{theorem:cyclic_quotient_h2}
If $h_2=(2,n,\ldots,n)$, then there is an isomorphism of $\C$-algebras
\begin{align*} 
\C[\Hess(N,h_2) \cap \Omega_e^\circ] \cong \frac{\C[X,Y,Z]}{(XY-Z^n)} \otimes \C[x_{32},x_{42},\ldots, x_{n-1 \, 2}] \otimes \C[x_{ij} \mid 3 \leq j < i \leq n]
\end{align*}
which sends 
\begin{equation}  
\begin{array}{l} \label{eq:correspondence cyclic}
X \mapsto x_{21}^2-x_{31}; \\
Y \mapsto x_{21}^{n-2}+x_{32}x_{21}^{n-3}+\cdots+x_{n-1 \, 2}x_{21}+x_{n2}; \\
Z \mapsto x_{21}; \\
x_{i2} \mapsto x_{i2} \ {\rm for} \ 3 \leq i \leq n-1; \\
x_{ij} \mapsto x_{ij} \ {\rm for} \ 3 \leq j < i \leq n. 
\end{array}
\end{equation}
In other words, 
\begin{align*}
\Hess(N,h_2) \cap \Omega_e^\circ \cong \C^2/\Cyclic_n \times \C^{\frac{1}{2}(n-1)(n-2)-1}.
\end{align*}
\end{theorem}

\begin{remark}
We have seen the case when $n=3$ in Example~\ref{example:n=3cyclic}.
\end{remark}

\begin{proof}[Proof of Theorem~\ref{theorem:cyclic_quotient_h2}]
By \eqref{eq:coordinate_ring_tildeFij} we have
\begin{align} \label{eq:proof_main2_1}
\C[\Hess(N,h_2) \cap \Omega_e^\circ] \cong \C[x_{ij} \mid 1 \leq j < i \leq n]/(\tilde F^{\langle 2 \rangle}_{3,1}, \tilde F^{\langle 2 \rangle}_{4,1}, \ldots, \tilde F^{\langle 2 \rangle}_{n,1}).
\end{align}
Put
\begin{align*} 
P_n \coloneqq (x_{21}^2-x_{31})(x_{21}^{n-2}+x_{32}x_{21}^{n-3}+\cdots+x_{n-1 \, 2}x_{21}+x_{n2})-x_{21}^n.
\end{align*}
It then follows from Proposition~\ref{proposition:cyclic} that 
\begin{align} \label{eq:proof_main2_2} 
\C[x_{ij} \mid 1 \leq j < i \leq n]/(\tilde F^{\langle 2 \rangle}_{3,1}, \ldots, \tilde F^{\langle 2 \rangle}_{n-1,1}, \tilde F^{\langle 2 \rangle}_{n,1}) \cong \C[x_{ij} \mid 1 \leq j < i \leq n]/(\tilde F^{\langle 2 \rangle}_{3,1}, \ldots, \tilde F^{\langle 2 \rangle}_{n-1,1}, P_n). 
\end{align}
By the definition \eqref{eqFi1m1=2def}, $\tilde F^{\langle 2 \rangle}_{i,1}=0$ if and only if
\begin{align*} 
x_{i+1 \, 1}=x_{21}x_{i1}+x_{31}x_{i2}-x_{21}^2x_{i2}
\end{align*}
for $3 \leq i \leq n-1$.
These equalities for $i=n-1,n-2,\ldots,3$ lead us to the isomorphism
\begin{align} \label{eq:proof_main2_3}
\C[x_{ij} \mid 1 \leq j < i \leq n]/(\tilde F^{\langle 2 \rangle}_{3,1}, \ldots, \tilde F^{\langle 2 \rangle}_{n-1,1}, P_n) \cong \C[x_{21},x_{31},x_{ij} \mid 2 \leq j < i \leq n]/(P_n). 
\end{align}
It is straightforward to see that
\begin{align} \label{eq:proof_main2_4}
& \C[x_{21},x_{31},x_{ij} \mid 2 \leq j < i \leq n]/(P_n) \\
\cong &\C[X,Y,Z]/(XY-Z^n) \otimes \C[x_{32},x_{42},\ldots, x_{n-1 \, 2}] \otimes \C[x_{ij} \mid 3 \leq j < i \leq n], \notag
\end{align}
which sends 
\begin{equation*}  
\begin{array}{l} 
X \mapsto x_{21}^2-x_{31}; \\
Y \mapsto x_{21}^{n-2}+x_{32}x_{21}^{n-3}+\cdots+x_{n-1 \, 2}x_{21}+x_{n2}; \\
Z \mapsto x_{21}; \\
x_{i2} \mapsto x_{i2} \ {\rm for} \ 3 \leq i \leq n-1; \\
x_{ij} \mapsto x_{ij} \ {\rm for} \ 3 \leq j < i \leq n. 
\end{array}
\end{equation*}
In fact, the inverse map is given by
\begin{equation*}  
\begin{array}{l} 
x_{21} \mapsto Z; \\
x_{31} \mapsto -X+Z^2; \\
x_{n2} \mapsto Y-Z^{n-2}-x_{32}Z^{n-3}-\cdots-x_{n-1 \, 2}Z; \\
x_{i2} \mapsto x_{i2} \ {\rm for} \ 3 \leq i \leq n-1; \\
x_{ij} \mapsto x_{ij} \ {\rm for} \ 3 \leq j < i \leq n.  
\end{array}
\end{equation*}
Combining  \eqref{eq:proof_main2_1}, \eqref{eq:proof_main2_2}, \eqref{eq:proof_main2_3}, and  \eqref{eq:proof_main2_4}, we conclude that 
\begin{align*} 
\C[\Hess(N,h_2) \cap \Omega_e^\circ] \cong \frac{\C[X,Y,Z]}{(XY-Z^n)} \otimes \C[x_{32},x_{42},\ldots, x_{n-1 \, 2}] \otimes \C[x_{ij} \mid 3 \leq j < i \leq n].
\end{align*}
We complete the proof. 
\end{proof}

For a polynomial $f \in \C[\Omega_e^\circ] \cong \C[x_{ij} \mid 1 \leq j < i \leq n]$, we denote the zero set of $f$ by
\begin{align} \label{zero_set_Vf}
V(f)\coloneqq\{g \in \Omega_e^\circ \cong \C^{\frac{1}{2}n(n-1)} \mid f(g) = 0\}.
\end{align}

\begin{corollary} \label{corollary:singular_locus_h2}
The singular locus of $\Hess(N,h_2) \cap \Omega_e^\circ$ is given by
\begin{align*}
\Sing(\Hess(N,h_2) \cap \Omega_e^\circ) = \bigcap_{i=2}^n V(x_{i1}) \cap V(x_{n2}). 
\end{align*}
\end{corollary}

\begin{proof}
By Theorem~\ref{theorem:cyclic_quotient_h2} we have 
\begin{align*}
\Hess(N,h_2) \cap \Omega_e^\circ \cong \C^2/\Cyclic_n \times \C^{\frac{1}{2}(n-1)(n-2)-1} \cong V(XY-Z^n).
\end{align*}
Here, $V(XY-Z^n)$ denotes the hypersurface defined by the equation $XY-Z^n=0$ in $\C^{\frac{1}{2}(n-1)(n-2)+2}$ with the variables $X,Y,Z,x_{32},x_{42},\ldots, x_{n-1 \, 2}, x_{ij} \ (3 \leq j < i \leq n)$. 
The singular locus of $V(XY-Z^n)$ is the solution set of the equations 
\begin{align*}
X = 0, \ Y=0, \ Z=0.
\end{align*}
It follows from the correspondence \eqref{eq:correspondence cyclic} that the image of the ideal $(X,Y,Z)$ under the isomorphism $\C[\Hess(N,h_2) \cap \Omega_e^\circ] \cong \frac{\C[X,Y,Z]}{(XY-Z^n)} \otimes \C[x_{32},x_{42},\ldots, x_{n-1 \, 2}] \otimes \C[x_{ij} \mid 3 \leq j < i \leq n]$ is 
\begin{align*}
(x_{21}^2-x_{31}, x_{21}^{n-2}+x_{32}x_{21}^{n-3}+\cdots+x_{n-1 \, 2}x_{21}+x_{n2}, x_{21}) = (x_{31}, x_{n2}, x_{21})
\end{align*}
in $\C[\Hess(N,h_2) \cap \Omega_e^\circ]$. 
Hence by the presentation \eqref{eq:coordinate_ring_tildeFij}, the singular locus of $\Hess(N,h_2)$ is the solution set of the equations 
\begin{align} \label{eq:proof_singular_locus_h2_1}
\begin{cases}
x_{21}=0 \\
x_{31}=0 \\
x_{n2}=0 \\
\tilde F^{\langle 2 \rangle}_{i,1}=0 \ \ \ \textrm{for} \ 3 \leq i \leq n-1.
\end{cases}
\end{align} 
Note that the equation $\tilde F^{\langle 2 \rangle}_{n,1}=0$ can be removed above since $\tilde F^{\langle 2 \rangle}_{n,1}=0$ is derived from the equations \eqref{eq:proof_singular_locus_h2_1} by Proposition~\ref{proposition:cyclic}. 
It then follows from the equality $\tilde F^{\langle 2 \rangle}_{i,1} = x_{i+1 \, 1}+x_{21}^2x_{i2}-x_{21}x_{i1}-x_{31}x_{i2}$ for $3 \leq i \leq n-1$ that the solution set of the equations \eqref{eq:proof_singular_locus_h2_1} is given by $x_{i1}=x_{n2}=0$ for $2 \leq i \leq n$, as desired. 
\end{proof}

\bigskip

%%%%%%%%%%%%%%%%%%%%%%%%%%%%%%%%%%
\section{Singular locus of $\Hess(N,h_m) \cap \Omega_e^\circ$}
\label{section:Singular locus of Hess(N,hm)}
%%%%%%%%%%%%%%%%%%%%%%%%%%%%%%%%%%

We now give an explicit description for the singular locus of $\Hess(N,h_m) \cap \Omega_e^\circ$ where $h_m$ is defined in \eqref{eq:hm}. 

\begin{theorem} \label{theorem:singular_locus_hm}
Let $2 \leq m \leq n-1$.
Then, the singular locus of $\Hess(N,h_m) \cap \Omega_e^\circ$ is described as
\begin{align*}
\Sing(\Hess(N,h_m) \cap \Omega_e^\circ) = \bigcap_{i=2}^n V(x_{i1}) \cap \bigcap_{j=2}^{m} V(x_{nj}), 
\end{align*}
where the notation $V(f)$ is defined in \eqref{zero_set_Vf}. 
\end{theorem}

\begin{proof}
We prove this by induction on $m$. 
The base case is $m=2$, which follows from Corollary~\ref{corollary:singular_locus_h2}. 
Suppose now that $m > 2$ and that the claim holds for
$m-1$. 
It follows from Proposition~\ref{proposition:singular_locus_hm_quantum} that the singular locus of $V\big(\, \hmE_1^{(n)}, \ldots, \, \hmE_n^{(n)}\big)$ is given by the solution set of the equations 
\begin{align} \label{eq:proof_singular_locus_hm_1}
\frac{\partial}{\partial q_{rs}} \, \hmE_{n}^{(n)} = 0 \ \ \ \textrm{for all} \ 2 \leq s \leq n \ \textrm{and} \ n-h_m(n+1-s)<r \leq s.
\end{align}
By Lemma~\ref{lemma:recursive_qijab} we have 
\begin{align} \label{eq:proof_singular_locus_hm_2}
\begin{cases}
E_1^{[n,n]} &= x_n \\
E_2^{[n-1,n]} &= E_{1}^{[n-1,n-1]} x_n + q_{n-1 \, n} \\
E_3^{[n-2,n]} &= E_{2}^{[n-2,n-1]} x_n + E_{1}^{[n-2,n-2]} q_{n-1 \, n} + q_{n-2 \, n} \\
&\vdots \\
E_{m}^{[n-m+1,n]} &= E_{m-1}^{[n-m+1,n-1]} x_n + \cdots + E_1^{[n-m+1,n-m+1]} q_{n-m+2 \, n} + q_{n-m+1 \, n}. 
\end{cases}
\end{align}
Since it holds that 
\begin{align*}
\hmE_i^{[n-i+1,n]}=\frac{\partial}{\partial q_{1 \, n-i}} \, \hmE_{n}^{(n)} = 0 \ \ \ \textrm{for all} \ 1 \leq i \leq m
\end{align*}
by \eqref{eq:partial_derivatives} and \eqref{eq:proof_singular_locus_hm_1}, the equalities \eqref{eq:proof_singular_locus_hm_2} lead us to the following equation
\begin{align*}
\left(
 \begin{array}{@{\,}ccccc@{\,}}
     1 &  &  &  & \\
     \hmE_1^{[n-1,n-1]} & 1 & &  &  \\ 
     \hmE_2^{[n-2,n-1]}  & \hmE_1^{[n-2,n-2]} & 1 & & \\ 
     \vdots & \vdots & \ddots & \ddots & \\
     \hmE_{m-1}^{[n-m+1,n-1]} & \hmE_{m-2}^{[n-m+1,n-2]} & \cdots & \hmE_1^{[n-m+1,n-m+1]} & 1 
 \end{array}
 \right) \left(
 \begin{array}{@{\,}c@{\,}}
     x_n \\
     q_{n-1 \, n} \\ 
     q_{n-2 \, n} \\ 
     \vdots \\ 
     q_{n-m+1 \, n}  
 \end{array}
 \right)
= \left(
 \begin{array}{@{\,}c@{\,}}
     0 \\
     0 \\ 
     0 \\ 
     \vdots \\ 
     0 \\
 \end{array}
 \right),
\end{align*}
which yields that $x_n = q_{n-1 \, n} = q_{n-2 \, n} = \cdots = q_{n-m+1 \, n} = 0$.
Since we have $q_{n-m+1 \, n} = 0$ and 
\begin{align*} 
\left. \left( \frac{\partial}{\partial q_{rs}} \, \hmE_{n}^{(n)} \right) \right|_{q_{n-m+1 \, n}=0} = \frac{\partial}{\partial q_{rs}} \, ^{h_{m-1}}E_{n}^{(n)}
\end{align*}
for $2 \leq s \leq n$ and $n-h_{m-1}(n+1-s)<r \leq s$, the equations \eqref{eq:proof_singular_locus_hm_1} imply that  
\begin{align} \label{eq:proof_singular_locus_hm_1_rewrite}
\begin{cases}
\frac{\partial}{\partial q_{rs}} \, ^{h_{m-1}}E_{n}^{(n)} &= 0 \ \ \ \textrm{for all} \ 2 \leq s \leq n \ \textrm{and} \ n-h_{m-1}(n+1-s)<r \leq s \\
\frac{\partial}{\partial q_{n-m+1 \, n}} \, \hmE_{n}^{(n)} &= 0 \\
q_{n-m+1 \, n} &= 0.
\end{cases}
\end{align}
Conversely, the equations \eqref{eq:proof_singular_locus_hm_1_rewrite} yield \eqref{eq:proof_singular_locus_hm_1}. 
In fact, since one can write $\, ^{h_{m}}E_{n}^{(n)} = \, ^{h_{m-1}}E_{n}^{(n)}+q_{n-m+1 \, n} \cdot F$ for some polynomial $F$ by Definition~\ref{definition:hEi(n)}, we have $\frac{\partial}{\partial q_{rs}} \, ^{h_{m}}E_{n}^{(n)} = \frac{\partial}{\partial q_{rs}} \, ^{h_{m-1}}E_{n}^{(n)}+q_{n-m+1 \, n} \cdot \frac{\partial}{\partial q_{rs}} F$ for all $2 \leq s \leq n$ and $n-h_{m-1}(n+1-s)<r \leq s$.  
Hence, \eqref{eq:proof_singular_locus_hm_1} is equivalent to \eqref{eq:proof_singular_locus_hm_1_rewrite}. 
One has 
\begin{align*}
\frac{\partial}{\partial q_{n-m+1 \, n}} \, \hmE_{n}^{(n)} = \, \hmE_{n-m}^{[1,n-m]} = \, \hmE_{n-m}^{(n-m)} 
\end{align*}
from \eqref{eq:Ei1b} and \eqref{eq:partial_derivatives}, so the singular locus of $V\big(\, \hmE_1^{(n)}, \ldots, \, \hmE_n^{(n)}\big)$ is given by the solution set of the equations 
\begin{align*} 
\begin{cases}
\frac{\partial}{\partial q_{rs}} \, ^{h_{m-1}}E_{n}^{(n)} &= 0 \ \ \ \textrm{for all} \ 2 \leq s \leq n \ \textrm{and} \ n-h_{m-1}(n+1-s)<r \leq s \\
\hmE_{n-m}^{(n-m)} &= 0 \\
q_{n-m+1 \, n} &= 0.
\end{cases}
\end{align*}
By Corollary~\ref{corollary:main_commutative_diagram} the following commutative diagram holds   
\begin{align} \label{eq:proof_singular_locus_hm_3}
  \begin{CD}
     \C[\Fl(\C^n) \cap \Omega_e^\circ] @>{\varphi}>{\cong}> \displaystyle\frac{\C[x_1,\ldots,x_n, q_{rs} \mid 1 \leq r<s \leq n]}{(E_1^{(n)}, \ldots, \, E_n^{(n)})} \\
  @VV{F_{i,1}=0 \ (m+1 \leq i \leq n)}V    @VV{q_{rn}=0 \ (1 \leq r \leq n-m)}V \\
     \C[\Hess(N,h_m) \cap \Omega_e^\circ]   @>{\varphi_{h_m}}>{\cong}> \displaystyle\frac{\C[x_1,\ldots,x_n, q_{rs} \mid 2 \leq s \leq n, n-h_m(n+1-s)<r<s]}{(\hmE_1^{(n)}, \ldots, \, \hmE_n^{(n)})}
  \end{CD}
\end{align}
where both vertial arrows are surjective.  
The ideal defining the singular locus of the zero set $V\big(\, \hmE_1^{(n)}, \ldots, \, \hmE_n^{(n)}\big)$ is
\begin{align*}
\left(\frac{\partial}{\partial q_{rs}} \, ^{h_{m-1}}E_{n}^{(n)} \middle| \ 2 \leq s \leq n, n-h_{m-1}(n+1-s)<r \leq s \right) + (\hmE_{n-m}^{(n-m)})+(q_{n-m+1 \, n}). 
\end{align*}
The image of the ideal above under the isomorphism $\varphi_{h_m}^{-1}$ in \eqref{eq:proof_singular_locus_hm_3} is equal to
\begin{align*}
\left(\varphi^{-1}\left(\frac{\partial}{\partial q_{rs}} \, ^{h_{m-1}}E_{n}^{(n)}\right) \middle| \ 2 \leq s \leq n, n-h_{m-1}(n+1-s)<r \leq s \right) + (x_{nm})+(F_{m,1}) 
\end{align*}
in $\C[\Hess(N,h_m) \cap \Omega_e^\circ]$ by the definition \eqref{eq:varphih_iso} and Proposition~\ref{proposition:inverse_map_varphih}. 
Hence, the singular locus of $\Hess(N,h_m) \cap \Omega_e^\circ$ is the solution set of the equations 
\begin{align} \label{eq:proof_singular_locus_hm_4} 
\begin{cases}
\varphi^{-1}\left(\frac{\partial}{\partial q_{rs}} \, ^{h_{m-1}}E_{n}^{(n)} \right) = 0 \ \ \ \textrm{for all} \ 2 \leq s \leq n \ \textrm{and} \ n-h_{m-1}(n+1-s)<r \leq s \\
F_{m,1} = 0 \\
F_{i,1} = 0 \ \ \ \textrm{for all} \ m+1 \leq i \leq n 
\end{cases}
\end{align}
and 
\begin{align*} 
x_{nm} = 0 
\end{align*}
by Lemma~\ref{lemma:defining equation}. 
Since the singular locus of $\Hess(N,h_{m-1}) \cap \Omega_e^\circ$ is the solution set of the equations \eqref{eq:proof_singular_locus_hm_4} from Proposition~ \ref{proposition:singular_locus_hm_quantum}, Theorem~\ref{theorem:main1}, and Lemma~\ref{lemma:defining equation} again, we can describe 
\begin{align*}
\Sing(\Hess(N,h_m) \cap \Omega_e^\circ) = \Sing(\Hess(N,h_{m-1}) \cap \Omega_e^\circ) \cap V(x_{nm}). 
\end{align*}
One can see from our inductive assumption that $\Sing(\Hess(N,h_{m-1}) \cap \Omega_e^\circ)= \bigcap_{i=2}^n V(x_{i1}) \cap \bigcap_{j=2}^{m-1} V(x_{nj})$, so we complete the proof.
\end{proof}

Let $\SS_n$ be the permutation group on $[n]$. 
For $w \in \SS_n$, the Schubert cell $X_w^\circ$ is defined to be the $B$-orbit of the permutation flag $wB$ in the flag variety $\GL_n(\C)/B$. 
The Schubert variety $X_w$ is defined by the closure of the Schubert cell $X_w^\circ$, i.e. $X_w=\overline{BwB/B}$.
We put $F_{i} \coloneqq \Span_\C \{e_1,\ldots,e_i \}$ for $i \in [n]$ where $e_1,\ldots, e_n$ are the standard basis for $\C^n$. 
Under the identification $\Fl(\C^n) \cong \GL_n(\C)/B$, one can describe the Schubert variety $X_w$ as
\begin{align*}
X_w=\{V_\bullet \in \Fl(\C^n) \mid \dim(V_p \cap F_q) \geq r_w(p,q) \ \textrm{for all} \ p,q \in [n] \}
\end{align*}
where $r_w(p,q)=|\{i \in [p] \mid w(i) \leq q \}|$ (e.g. \cite[Section~10.5]{Ful97}). 
For $2 \leq m \leq n-1$, we define the permutation $w_m \in \SS_n$ as 
\begin{align} \label{eq:_definition_wm}
w_m(i)=\begin{cases} 1 \ &{\rm if} \ i=1, \\
n+1-i \ &{\rm if} \ 2 \leq i \leq m, \\
n \ &{\rm if} \ i=m+1, \\
n+2-i \ &{\rm if} \ m+2 \leq i \leq n. 
\end{cases}
\end{align}
Then one can verify from \cite{GasRei} that the Schubert variety $X_{w_m}$ is described as
\begin{align*}
X_{w_m} = \{V_\bullet \in \Fl(\C^n) \mid \dim(V_1 \cap F_1) \geq 1 \ \textrm{and} \ \dim(V_m \cap F_{n-1}) \geq m \}.
\end{align*} 
In other words, $V_\bullet \in X_{w_m}$ if and only if $V_1 = F_1=\Span_\C \{e_1 \}$ and $V_m \subset F_{n-1}=\Span_\C \{e_1,\ldots,e_{n-1} \}$.
In particular, we have  
\begin{align*}
X_{w_m} \cap \Omega_e^\circ = \bigcap_{i=2}^n V(x_{i1}) \cap \bigcap_{j=2}^{m} V(x_{nj}).
\end{align*}
Hence, we obtain the following result from Theorem~\ref{theorem:singular_locus_hm}.

\begin{corollary} \label{corollary:singular_locus_hm}
Let $2 \leq m \leq n-1$.
Then, the singular locus of $\Hess(N,h_m) \cap \Omega_e^\circ$ is equal to
\begin{align*}
\Sing(\Hess(N,h_m) \cap \Omega_e^\circ) = X_{w_m} \cap \Omega_e^\circ.
\end{align*}
\end{corollary}

\begin{remark}
Let $A$ be a nilpotent matrix (not necessarily regular nilpotent).
Then it is known from \cite[Theorem~5]{EPS} that the singular locus of $\Hess(A,h_{n-1})$ is 
\begin{align*}
\Sing(\Hess(A,h_{n-1})) = \Hess(A,h=(1,n-1,\ldots,n-1,n)). 
\end{align*}
Consider the regular nilpotent case.
The Hessenberg function $h=(1,n-1,\ldots,n-1,n)$ is decomposable, so every flag $V_\bullet \in \Hess(N,h)$ has $V_{1} = \Span_\C \{e_1\}$ and $V_{n-1} = \Span_\C \{e_1,\ldots,e_{n-1}\}$ (see Definition~\ref{definition:indecomposable} and surrounding discussion). 
Hence the result of \cite[Theorem~5]{EPS} for $A=N$ gives 
\begin{align*}
\Sing(\Hess(N,h_{n-1})) = X_{w_{n-1}}.
\end{align*}
\end{remark}

\bigskip

\appendix

%%%%%%%%%%%%%%%%%%%%%%%%%%%%%%%%%%
\section{Singular locus of Peterson variety}
\label{section:Singular locus of Peterson variety}
%%%%%%%%%%%%%%%%%%%%%%%%%%%%%%%%%%

The singular locus of the Peterson variety $\Pet_n$ has been studied in \cite{IY}.
Combining the result of \cite{IY} and some result in \cite{AHKZ}, we describe the decomposition for the singular locus of the Peterson variety into irreducible components.

It is known that the flag variety $\Fl(\C^n)$ has a decomposition into Schubert cells
\begin{align} \label{eq:decomposition_flag}
\Fl(\C^n) = \bigsqcup_{w \in \SS_n} X_w^\circ,
\end{align}
where each $X_w^\circ$ is isomorphic to $\C^{\ell(w)}$ and $\ell(w)$ denotes the length of $w$ (e.g. \cite{Ful97}). 
Tymoczko generalized this fact to the Hessenberg varieties in \cite{Tym06}. 
In what follows, we explain the work in \cite{Tym06} for the case of the Peterson variety $\Pet_n$.

Let $I$ be a subset of $[n-1]$. 
We may regard $[n-1]$ as the set of vertices of the Dynkin diagram in type $A_{n-1}$.
Then, $I \subset [n-1]$ can be decompose into the connected components of the Dynkin diagram of type $A_{n-1}$:
\begin{align*} 
I = I_1 \sqcup I_2 \sqcup \cdots \sqcup I_m.
\end{align*}
In other words, each $I_j \ (1 \leq j \leq m)$ denotes a maximal consecutive subset of $[n-1]$. 
To each connected component $I_j$, one can assign the longest element $w_0^{(I_j)}$ of the permutation subgroup $\SS_{I_j}$ on $I_j$.
Then, we define the permutation $w_I \in \SS_n$ by  
\begin{align*}
w_I \coloneqq w_0^{(I_1)} w_0^{(I_2)} \cdots w_0^{(I_m)}. 
\end{align*}

\begin{example}\label{eg: wJ}
Let $n=9$ and $I=\{1,2,3\} \sqcup \{6,7\}$. 
Then, the one-line notation of $w_I$ is 
\begin{align*}
w_I=432158769.
\end{align*}
\end{example}

The Schubert cell $X_v^\circ$ intersects with the Peterson variety $\Pet_n$ if and only if $v=w_I$ for some $I \subset [n-1]$ by \cite{Tym06} (see also \cite[Lemma~3.5]{AHKZ}). 
We set 
\begin{align*}
X_I^\circ \coloneqq X_{w_I}^\circ \cap \Pet_n \ \textrm{and} \ X_I \coloneqq X_{w_I} \cap \Pet_n = \overline{X_{w_I}^\circ} \cap \Pet_n. 
\end{align*}
By intersecting with the Peterson variety $\Pet_n$, the decomposition in \eqref{eq:decomposition_flag} yields that 
\begin{align*} 
\Pet_n = \bigsqcup_{J \subset [n-1]} X_J^\circ.
\end{align*}
It is known from \cite{Tym06} that $X_J^\circ \cong \C^{|J|}$ for any $J \subset [n-1]$. 
In general, it follows from \cite[Equation~(3.7)]{AHKZ} that for each $I \subset [n-1]$, we have
\begin{align} \label{eq:decomposition_XI}
X_I = \bigsqcup_{J \subset I} X_J^\circ.
\end{align}

It is known that $X_I$ is a regular nilpotent Hessenberg variety for a certain Hessenberg function $h_I$ as described below. 
For $I \subset [n-1]$, we define a Hessenberg function $h_I : [n] \rightarrow [n]$ by
\begin{align*}
h_I(i) \coloneqq \begin{cases}
i+1 \ \ \ &\textrm{if} \ i \in I, \\
i \ \ \ &\textrm{if} \ i \notin I. 
\end{cases}
\end{align*}
Note that if $I=[n-1]$, then $h_I=(2,3,4,\ldots,n,n)$ is the Hessenberg function for the Peterson case. Otherwise, $h_I$ is decomposable (Definition~\ref{definition:indecomposable}). 

\begin{proposition} $($\cite[Proposition~3.4]{AHKZ}$)$ \label{proposition:XIHess(NhI)}
For a subset $I$ of $[n-1]$, we have
\begin{align*}
X_I = \overline{X_I^\circ} = \Hess(N,h_I)
\end{align*}
where $N$ is the regular nilpotent element defined in \eqref{eq:regular nilpotent}. 
\end{proposition}

It follows from Theorem~\ref{theorem:property_Hess(N,h)} and Proposition~\ref{proposition:XIHess(NhI)} that $X_I$ is irreducible for any $I \subset [n-1]$. 
For positive integers $a,b$ with $a \leq b$, we denote by $[a, b]$ the set $\{a, a + 1, \ldots , b\}$.
The singular locus of the Peterson variety $\Pet_n$ is described in \cite{IY} as follows.

\begin{theorem} $($\cite[Theorem~4]{IY}$)$ \label{theorem:singular_locus}
The singular locus of $\Pet_n$ is given by
\begin{align*}
\Sing(\Pet_n) = \bigsqcup_{J \subset [n-1] \atop J \neq [n-1], [2,n-1],[1,n-2]} X_J^\circ. 
\end{align*}
\end{theorem}

\begin{lemma} \label{lemma:singular_locus}
We have
\begin{align*}
\{J \subset [n-1] \mid J \neq [n-1], [2,n-1],[1,n-2] \} = \bigcup_{j=2}^{n-2} \{ J \subset [n-1] \mid J \not\ni j \}  \cup \{J \subset [2,n-2] \}. 
\end{align*}
\end{lemma}

\begin{proof}
We first show that the left hand side is included in the right hand side.
For this, we take a subset $J$ of $[n-1]$ such that $J \neq [n-1], [2,n-1],[1,n-2]$. 
Note that $|J| \leq n-2$.

\noindent
\textbf{Case (i): } Suppose that $|J| \leq n-3$.
If $J$ contains $[2,n-2]$, then $J=[2,n-2]$ since $|J| \leq n-3$.
Otherwise, we have $J \not\ni j$ for some $2 \leq j \leq n-2$.
In both cases, $J$ belongs to the right hand side.

\noindent
\textbf{Case (ii): } Suppose that $|J| = n-2$.
Since $J \neq [2,n-1],[1,n-2]$, we see that $J = [n-1] \setminus \{j \}$ for some $2 \leq j \leq n-2$, which belongs to the right hand side.

\noindent
Hence, we proved that the left hand side is included in the right hand side.

Conversely, let $J$ be a subset of $[n-1]$ appeared in the right hand side. 
If $J \not\ni j$ for some $2 \leq j \leq n-2$, then we have that $J \neq [n-1], [2,n-1],[1,n-2]$. 
If $J \subset [2,n-2]$, then it is clear that $J \neq [n-1], [2,n-1],[1,n-2]$. 
Thus, the right hand side is included in the left hand side, so the equality holds.
\end{proof}

The following proposition gives the decomposition for the singular locus of the Peterson variety into irreducible components.

\begin{proposition} \label{proposition:decomposition_singular_locus}
The singular locus of $\Pet_n$ is decomposed into irreducible components as follows:
\begin{align*}
\Sing(\Pet_n) &= \left( \bigcup_{2 \leq j \leq n-2} X_{[n-1] \setminus \{j\}} \right) \cup X_{[2,n-2]} \\
&=\left( \bigcup_{2 \leq j \leq n-2} \Hess(N,h_{[n-1] \setminus \{j\}}) \right) \cup \Hess(N,h_{[2,n-2]}).
\end{align*}
\end{proposition}

\begin{proof}
By Theorem~\ref{theorem:singular_locus} and Lemma~\ref{lemma:singular_locus} we have 
\begin{align*}
\Sing(\Pet_n) = \left(\bigcup_{2 \leq j \leq n-2} \bigsqcup_{J \subset [n-1] \setminus \{j\}} X_J^\circ \right) \cup \left( \bigsqcup_{J \subset [2,n-2]} X_J^\circ \right). 
\end{align*}
By using the decomposition in \eqref{eq:decomposition_XI}, the right hand side above is equal to 
\begin{align*}
\left( \bigcup_{2 \leq j \leq n-2} X_{[n-1] \setminus \{j\}} \right) \cup X_{[2,n-2]}, 
\end{align*} 
as desired. 
Also, this coincides with $\left( \bigcup_{2 \leq j \leq n-2} \Hess(N,h_{[n-1] \setminus \{j\}}) \right) \cup \Hess(N,h_{[2,n-2]})$ by Proposition~\ref{proposition:XIHess(NhI)}.
\end{proof}

\end{document}